\documentclass[12pt]{amsart}
\usepackage{bbm, amssymb, amsmath, amsfonts}

\usepackage{graphicx, epsfig}

\usepackage{xcolor}

 \makeindex

\def\cal{\mathcal}

\def\Bbb{\mathbb}

\def\Im{\text{\rm Im\,}}

\def\pad{\phi^a}

\def\pred{\phi_{\rm red}}

\def \supp {\text{\rm supp\,}}

\def\dist{\text{\rm dist\,}}

\def\A{{\cal A}}

\def\F{{\cal F}}

\def\M{{\cal M}}
\def\N{{\cal N}}
\def\S{{\cal S}}
\def\T{{\cal T}}

\def\CC{{\Bbb C}}

\def\NN{{\Bbb N}}
\def\bN{{\Bbb N}}
\def\bR{{\Bbb R}}
\def\RR{{\Bbb R}}

\def\vp{{\varphi}}
\def\al{{\alpha}}
\def\be{{\beta}}
\def\ga{{\gamma}}
\def\Ga{{\Gamma}}
\def\la{{\lambda}}
\def\om{{\omega}}

\def\x{(x_1,x_2)}
\def\y{(y_1,y_2)}
\def\pa{{\partial}}
\def\all{{\rm all}}

\def\ve{{\varepsilon}}
\def\si{{\sigma}}

\def\de{{\delta}}

\def\Om{{\Omega}}
\def\ka{{\kappa}}

\def\aol{{\left[ \begin{matrix} \al \\ l\end{matrix}\right ]}}

\def\dotol{{\left[ \begin{matrix} \cdot \\ l\end{matrix}\right ]}}

\def\bpm{\begin{pmatrix}}
\def\epm{\end{pmatrix}}

\def\noi{\noindent}
\def\bee{\begin{enumerate}}
\def\ee{\end{enumerate}}

\def\qed{\smallskip\hfill Q.E.D.\medskip}

\newtheorem{thm}{Theorem}[section]

\newtheorem{prop}[thm]{Proposition}
\newtheorem{proposition}[thm]{Proposition}
\newtheorem{cor}[thm]{Corollary}
\newtheorem{lemma}[thm]{Lemma}
\newtheorem{remark}[thm]{Remark}
\newtheorem{remarks}[thm]{Remarks}

\newtheorem{assumption}[thm]{Assumption}
\newtheorem{example}[thm]{Example}

\newtheorem{conjecture}{Conjecture}

\begin{document}


\title[problems of harmonic analysis related to  hypersurfaces]
{ Estimates for maximal functions  associated to hypersurfaces in $\bR^3$
with height $h<2:$ Part II\\
 \small A \lowercase{geometric conjecture and its proof for generic 2-surfaces} }

\author[S. Buschenhenke]{Stefan Buschenhenke}
\address{Mathematisches Seminar, C.A.-Universit\"at Kiel,
Heinrich-Hecht-Platz 6, D-24118 Kiel, Germany} \email{{\tt
buschenhenke@math.uni-kiel.de}}
\urladdr{{http://analysis.math.uni-kiel.de/buschenhenke/}}


\author[I. A. Ikromov]{Isroil A. Ikromov}
\address{Institute  of Mathematics, 
University Boulevard 15, 140104, Samarkand, Uzbekistan}
 \email{{\tt ikromov1@rambler.ru}}

\author[D. M\"uller]{Detlef M\"uller}
\address{Mathematisches Seminar, C.A.-Universit\"at Kiel,
Heinrich-Hecht-Platz 6, D-24118 Kiel, Germany} \email{{\tt
mueller@math.uni-kiel.de}}
\urladdr{{http://analysis.math.uni-kiel.de/mueller/}}

\thanks{2020 {\em Mathematical Subject Classification.}
42B25}
\thanks{{\em Key words and phrases.}
  Maximal operator, hypersurface, oscillatory integral, Newton diagram}
\thanks {We acknowledge the support for this work by the Deutsche Forschungsgemeinschaft under  DFG-Grant MU 761/11-2.
 }

\begin{abstract} 
In this article, we continue the study of 
$L^p$-boundedness   of the maximal operator $\M_S$ associated to averages along isotropic dilates of a given,  smooth hypersurface $S$ of finite type  in 3-dimensional Euclidean space which satisfies a natural transversality condition. An essentially complete answer to  this problem  had been given about ten years ago   by the last named two authors in joint work with  M. Kempe for the case where  the height $h$  of the given surface is at least two. The case where $h<2,$ and  where $S$ is contained in a sufficiently small neighborhood of a point $x^0\in S$ at which both principal curvatures vanish, had been treated in the first article (Part I)  of this series.

Here  we continue the study of the case  $h<2,$  by assuming   that  exactly one of the principal curvatures of $S$ does not vanish at $x^0.$ Such surfaces exhibit singularities of type $\A$ in the sense of Arnol'd's classification. We distinguish between two sub-types $\A^-$ and $\A^+.$ Under the assumption that $S$ is analytic, denoting by $p_c$ the minimal  Lebesgue exponent such that $\M_S$ is $L^p$-bounded for $p>p_c,$ we show that for sub-type $\A^-$ we have $p_c=\max\{3/2,  h\},$ whereas for surfaces of sub-type $\A^+$ which do not belong to an exceptional subclass $\A^e,$ we have $p_c=\max\{3/2, p_e, h\}.$  Here, $p_e:=2n_e/(n_e+1),$ where $n_e$ is a new quantity, called the effective multiplicity, which can be determined from Newton polyhedra associated to the given surface $S.$ Our conjecture is that  $p_c=\max\{3/2, p_e, h\}$ also for surfaces of class 
$\A^e.$

 We also state a conjecture on how the critical exponent $p_c$ might  be determined by means of a  geometric measure theoretic  condition, which measures in some way the order of contact of arbitrary ellipsoids with $S,$ even for hypersurfaces in arbitrary dimension, and show that this conjecture  holds  indeed true for all classes of 2-hypersurfaces $S$ for which we have gained an essentially complete understanding of $\M_S$ so far.

For surfaces of type $\A^e,$ we show that the methods devised in this paper allow at least to prove that $p_c\le \max\{3/2,2n/(n+1)\}.$ Our results lead in particular to a proof of a conjecture by Iosevich-Sawyer-Seeger  for arbitrary analytic 2-surfaces.

The study of the afore-mentioned  conjecture for surfaces of type $\A^e,$ which bears amazing connections to  combinations of cone multipliers with Fourier integral operators, will be left to the third paper in this series.
  \color{black}
  \end{abstract}

\maketitle


\tableofcontents

\thispagestyle{empty}

\setcounter{equation}{0}
\section{Introduction}\label{introduction}

Let $S$ be a smooth hypersurface in $\RR^d$ and let  $0\ne \rho\in
C_0^\infty(S)$ be a smooth non-negative function with compact
support.  Consider the associated averaging operators $A_t, t>0,$
given by
$$
A_tf(x):=\int_{S} f(x-ty) \rho(y) \,d\si(y),
$$
where $d\si$ denotes the surface measure on $S.$   The associated
maximal operator is given by
\begin{equation*}
\M f(x)=\M _\rho f(x):=\sup_{t>0}|A_tf(x)|, \quad ( x\in \RR^d).
\end{equation*}
If $S$ is compact and $\rho\equiv 1,$ we also  write $\M=:\M_S.$ 
We shall be interested in the question of $L^p$-boundedness of $\M,$ i.e., we would like to determine the range of  all $p\ge 1$ such that 
\begin{equation}\label{Mb}
\|\M f(x)\|_p\le C_p\|f\|_p \quad \text{for all} \ f\in \S.
\end{equation}
We therefore define the critical  exponent 
\begin{eqnarray*}
p_c&:=&\inf\{p\ge 1: \eqref{Mb}  \text{ holds true}\},
\end{eqnarray*}
so that \eqref{Mb} holds true when $p>p_c,$ but fails to be true for $p<p_c$ (what happens when $p=p_c$ is not captured by this critical exponent).

Clearly, since $\rho$ has compact support, this  problem can be localized to considering the contributions to $\M$ by  small neighborhoods of points $x^0\in S.$ 
Fixing such a point $x^0\in S,$  and assuming  that $\rho(x^0)\ne 0$ (in order to exclude mitigating effects through the vanishing of our  density $\rho$ at $x^0$), let us define the following local  critical exponent associated to this point:
\begin{eqnarray*}
p_c(x^0)&:=&\inf\{p\ge 1: \text{there is a neighborhood } U \text{ of }  x^0 \text{ such that }\\
&&\qquad\eqref{Mb}  \text{ holds true whenever }  \supp \rho\subset U \}.
\end{eqnarray*}

 \color{black}
 
Note  that by testing $\M$ on the characteristic function of the
unit ball in $\RR^d,$ it is easy to see  that a necessary condition
for $\M_S$ to be  bounded  on $L^p(\RR^d)$  is that  $p> d/(d-1),$  so that $p_c\ge d/(d-1),$ 
provided the transversality Assumption \ref{s1.1} below is satisfied.
\smallskip

In 1976,  E.~M.~Stein \cite{stein-sphere} proved that, conversely,
if $S$ is the Euclidean unit sphere in $\RR^d,\ d\ge 3, $  then the
corresponding spherical maximal operator is bounded  on $L^p(\RR^d)$
for every $p> d/(d-1).$ The analogous result in dimension  $d=2$ was
later proven by J.~Bourgain \cite{bourgain85}. The key property of
spheres which allows to prove  such  results is the non-vanishing of
the Gaussian curvature on spheres.  These results became  the
starting point for intensive  studies of various classes of maximal
operators associated to subvarieties. Stein's  monograph
\cite{stein-book} is an excellent reference to many of these
developments.
\smallskip

Here, we continue our study of this question for maximal functions $\M$ associated to analytic 2-hypersurfaces  in $\RR^3.$   As in the preceding article \cite{IKM-max}, and the first article \cite{bdim19} of this series, we shall work under the following transversality assumption on $S:$ 
\begin{assumption}[Transversality]\label{s1.1}
The affine tangent plane $x+T_xS$ to $S$ through $x$ does not pass
through the origin   for every $x\in S.$ Equivalently,
$x\notin T_xS$ for every $x\in S,$ so that $0\notin S$ and $x$ is
transversal to $S$ for every point $x\in S.$
\end{assumption}
Let us now fix a point $x^0\in S.$  We recall that the transversality
assumption allows us to find a linear change of coordinates in
$\RR^3$ so that in the new coordinates
 $S$ can locally be represented as the graph of a function $\phi,$  and that the norm of $\M$ when acting on $L^p(\RR^3)$ is invariant under such a linear change of coordinates.
 More precisely, after applying a suitable  linear change  of coordinates  to $\RR^3$ we may assume that
$x^0=(0,0,1),$  and that within a sufficiently small neighborhood $U$ of $x^0,$ $S$  is given
as the  graph 
$$
U\cap S=\{(x_1,x_2,1+ \phi\x): \x\in \Om \}
$$
of a smooth function $1+\phi$ defined on an open neighborhood $\Om$
of $0\in\RR^2$ and satisfying the conditions
\begin{equation}\label{1.2}
\phi(0,0)=0,\, \nabla \phi(0,0)=0.
\end{equation}
Moreover, assuming now that $\supp\rho\subset U,$ the measure $\mu=\rho d\si$ is then explicitly given by
$$
\int f\, d\mu=\int f(x,1+\phi(x)) \eta(x) \,dx,
$$
with a smooth, non-negative  bump function $\eta\in
C_0^\infty(\Om),$ and we may write for  $(y,y_3)\in \RR^2\times \RR$
\begin{equation*}
A_tf(y,y_3)=f*\mu_t(y,y_3)=\int_{\bR^2} f(y-tx,
y_3-t(1+\phi(x)))\eta(x) \, dx,
\end{equation*}
where $\mu_t$ denotes the norm preserving scaling  of the measure
$\mu$ given by $\int f\, d\mu_t=\int f(tx, t(1+\phi(x))) \eta(x)
\,dx.$
\smallskip

Let us briefly recall some basic notions concerning Newton polyhedra - for further ones we refer to \cite{bdim19}, and \cite{IMmon}. 
We begin by looking at the Taylor series 
$$\phi(x_1,x_2)\sim\sum_{\al_1,\al_2=0}^\infty c_{\al_1,\al_2} x_1^{\al_1} x_2^{\al_2}$$
of $\phi$ centered at  the origin.
The set
$$\T(\phi):=\{(\al_1,\al_2)\in\bN^2: c_{\al_1,\al_2}=\frac 1{\al_1!\al_2!}\partial_{ 1}^{\al_1}\partial_{ 2}^{\al_2} \phi(0,0)\ne 0\}      \index{T@$\T(\phi)$ (Taylor support)}
$$
will be called the {\it Taylor support}    of $\phi$ at $(0,0).$  We  shall always assume that the function $\phi$ is of {\it finite type} \index{finite type function} at every point, i.e., that the associated graph $S$ of $\phi$ is of finite type.  Since we are also assuming that $\phi(0,0)=0$ and $\nabla \phi(0,0)=0,$ the finite type assumption at the origin   just means that 
$$\T(\phi)  \ne \emptyset.$$ 
 The
{\it Newton polyhedron}   $\N(\phi)$ \index{N@$\N(\phi)$ (Newton polyhedron)} of $\phi$ at the origin is
defined to be the convex hull of the union of all the quadrants
$(\al_1,\al_2)+\bR^2_+$ in $\bR^2,$ with $(\al_1,\al_2)\in\T(\phi).$  
The {\it height} of $S$ at the
point $x^0$  is  defined by $h(x^0,S):=h(\phi),$  where $h(\phi)$ is
the height of $\phi$ in the sense of Varchenko (which can be
computed by means of Newton polyhedra attached to $\phi$ with respect to local coordinate systems).  For the notions of height, and adaptedness of coordinates, we refer to \cite{Va} and  \cite{IM-ada}. The
height  is invariant under affine linear changes of coordinates in
the ambient space  $\RR^3.$ 
\medskip

We recall what is known on this question so far. 
 \begin{itemize}
\item If $h(x^0,S)\ge 2,$  and if the density $\rho$ is supported in a sufficiently small neighborhood  of $x^0,$ then  the condition $p>h(x^0,S)$  is sufficient for $\M$ to be $L^p$-bounded, and  if $\rho(x^0)\ne 0,$ this result is  sharp (with the possible exception of the endpoint $p=h(x^0,S),$ when $S$ is non-analytic) (see \cite{IKM-max}, and  also \cite{IU} for some classes of hypersurfaces in $\RR^n$). For an alternative approach to some of these results based on ''damping'' techniques, see also 
\cite {greenblatt12}, \cite {greenblatt13}.

\item  If $h(x^0,S)<2,$ and if either both principal curvatures of $S$ vanish at $x^0,$  or both do not vanish, then under the same kind of assumptions on $\rho$  as before the condition $p>\max\{3/2, h(x^0,S\})$  is sufficient, and necessary too, for \eqref{Mb} to hold (see \cite{bdim19}).
\color{black}

\item  If $h(x^0,S)<2,$ and if exactly one of  the two principal curvatures of $S$  at $x^0$ vanishes, then it follows from  Theorem 3.1 of \cite{bdim19} that we may assume that,  in a suitable linear coordinate system (recall  here that our problem is invariant under linear changes of coordinates!)),  the functions $\phi$ is of the form
\begin{equation}\label{AD}
\phi(x_1,x_2)=b(x_1,x_2)(x_2-\psi(x_1))^2 +b_0(x_1),
\end{equation}
where $b,b_0$ and $\psi$ are  smooth functions, and  $b(0,0)\ne 0.$  
Moreover, either $\psi$ is flat at $0,$ or $\psi(x_1)=x_1^m\omega(x_1)$, where $\om$ is smooth with $\om(0)\ne0$ and $m\ge 2.$ We shall view the case where $\psi$  is flat as the case where formally $m=\infty.$

Moreover, the function $b_0$ is of finite type at the origin and thus  can be written as  
$$
b_0(x_1)=x_1^n \beta(x_1),
$$ 
where $n\ge 3$ is a positive integer  and $\beta$ is a smooth function with $\beta(0)\ne 0.$ Note here that if  $b_0$ were flat at the origin, then we would have $h(x^0,S)=2$  (compare the subsequent discussion), so this case  cannot arise here.

This means that $\phi$ has a {\it singularity of  type $\A_{n-1}$ }in the sense of Arnol'd's classification of singularities (cf. \cite{agv}), with finite $n\ge 3.$ 

We remark that a representation of $\phi$ in the normal form \eqref{AD} is possible if and only if  the coordinates $\x$ are linearly adapted in the sense of  \cite{IMmon}.  Throughout this paper, let us therefore assume that  our coordinates $\x$ are linearly adapted.

The case $n=3$ of $\A_{n-1}$-type singularities had already been treated in \cite{bdim19} (compare Theorem 7.1). In this case, under the same kind of assumptions on $\rho$  as before, the condition $p>3/2$  is sufficient, and necessary too, for \eqref{Mb} to hold.
 \end{itemize}
 \medskip

 These results  clearly imply that,  for all these classes of surfaces addressed so far, we have that
\begin{equation}\label{pcrit1}
p_c(x^0)=\max\{3/2, h\},
\end{equation}
with $h:=h(x^0,S).$

  Let us also remark that matters change drastically  when the transversality assumption fails, as has been shown by E. Zimmermann in his doctoral thesis \cite{Zi}.  Zimmermann studied the case where the hypersurface  $S$ passes through the origin and proved, among other things,  that for analytic $S$ and $\supp \rho$ sufficiently small, the condition $p>2$ is always sufficient for the $L^p$- boundedness of $\M_S.$

 \medskip
Under Assumption 1.1, however, what remains to be understood are the maximal functions  associated to  functions $\phi$ of the form  \eqref{AD} that  exhibit singularities of type $\A_{n-1},$ with $n\ge 4.$ Their study  will indeed present the most difficult challenges among all cases.  In view of our discussion in Theorem 3.1 of \cite{bdim19}  we shall distinguish the following two subcases, assuming that $\phi$  given by \eqref{AD} has a singularity of  type $\A_{n-1}, n\ge 4$:
 \smallskip

To begin with, let us recall from  \cite{bdim19} when the coordinates $x$ are adapted to $\phi$ (assuming  that $\phi$ is represented in the normal form \eqref{AD}).
 \smallskip
 
 \noindent{\bf Case of adapted coordinates.} The coordinates $x=\x$ are  adapted to $\phi$  if and only if  $n\le 2m$ (with the understanding that $m:=\infty$ if $\psi$ is flat at the origin).
 \smallskip
 
In this case  the principal weight associated to the principal edge of the Newton polyhedron $\N(\phi)$ is given by $\ka:= ( 1/n,1/2),$ and
$$
d(\phi)=h(\phi)=\frac {2n}{n+2},
$$
where $d(\phi)$ and $h(\phi)$ denote the Newton distance and the height of $\phi$ in the sense of Varchenko (see 
\cite{IMmon}, also for further notions used throughout this article).

\smallskip

 \noindent{\bf Case of non-adapted coordinates.} The coordinates $x=\x$ are  not adapted to $\phi$  if and only if $n>2m.$ 
 \smallskip
 In this case, adapted coordinates are given by $(y_1,y_2):=(x_1,x_2-x_1^m\om(x_1),$ and in these coordinates $\phi$ is given by
 $$\pad(y_1,y_2)=b(y_1,y_2+y_1^m\om(y_1)) y_2^2 +b_0(y_1).$$
 
Here,  the principal weight is given by $\ka:= ( 1/(2m),1/2),$  and Newton distance and height are given by 
$$
d(\phi)=\frac {2m}{m+1}, \qquad h(\phi)=\frac {2n}{n+2}.
$$
 In particular we see that $h(\phi)\le 3/2$ if and only if $n\le 6.$ ·
\medskip

As it turns out, for our purposes a slightly different distinction, namely between   the cases where $n<2m,$ and where $n\ge 2m,$  becomes more natural.
\medskip

 \noindent{\bf Functions  of type  $\A_{n-1}^+$.} By definition, this will be the case where $n<2m.$ 
Using the notation from \cite{IMmon}, notice  that this class  of functions can equivalently be characterized by the following property of their Newton polyhedra, without any recourse to the normal form \eqref{AD} (which will indeed not be of any real use in the study of this case).
\smallskip

\noindent {\bf Property $(\A_{n-1}^+):$ }{\it \ The principal face $\pi(\phi)$ of the Newton polyhedron of $\phi$ is the line segment $[(0,2 ),(n,0)]$ with endpoints $(0,2)$ and $(n,0),$ and these are the only points of the Taylor support $\T(\phi)$ of $\phi$ which are on this line segment, i.e., the principal part $\phi_{\rm pr}$ of $\phi$ is of the form $c_1x_2^2+  c_2x_1^n,$ with $c_1\ne 0\ne c_2.$ }

\smallskip
Note that when $n\ge 3,$  property $\A_{n-1}^+$ in particular implies that  the coordinates $x$ are adapted to $\phi$  (compare Proposition 1.2 in \cite{IMmon}), that  \eqref{1.2} holds true, and that the Hessian matrix of $\phi$ at the origin is of the form
$$
H\phi(0,0)=\left(
\begin{array}{ccc}
  0  &  0 \\
  0   &  \pa^2_2\phi(0,0) \\
\end{array}
\right), 
\text { where  } \ \pa^2_2\phi(0,0)\ne 0.
$$

\smallskip
 \noindent{\bf Functions  of type  $\A_{n-1}^-$.} This will be the remaining  case where $n\ge 2m.$ This class can again be characterized just by means of Newton polyhedra: either the coordinates $x$ are not adapted to $\phi$ (this is the case where $ n>2m$), or the principal face $\pi(\phi)$ is contained in the line segment $[(0,2 ),(n,0)],$ and besides the point $(0,2)$ there is a point  $P$ different from the endpoints $(0,2)$ and $(n,0)$ which is contained in the Taylor support  $\T(\phi).$ This is the case where $n=2m$ in the normal form \eqref{AD}, and here  we have $P=(m,1).$   Note also that in this particular case the endpoint $(n,0)$ may or may not be in $\T(\phi).$ 

\medskip

\noi {\bf Remark. }{\it 
It turns out that the study of the maximal operator $\M$ for  $\A_{n-1}$-type singularities requires very refined information on the resolution of singularities. For this reason, we shall assume in our main theorems that $\phi$ is analytic (by this we mean ``real analytic'').  }

\subsection{Case of  $\A_{n-1}^-$}
It is interesting to note that  case  $\A_{n-1}^-,$ which includes all cases where the coordinates $x$ are not adapted to $\phi,$ turns out to be ``easier'' to handle  than case  $\A_{n-1}^+$. Indeed, it will be settled in a complete way (at least for analytic surfaces) in this paper by the following result:
\begin{thm}\label{thm-a-}
Assume that $S$ is the graph of $1+\phi,$ and accordingly $x^0=(0,0,1),$  where $\phi$ is  analytic and of type $\A_{n-1}^-$. Then, if the density $\rho$ is supported in a sufficiently small neighborhood  of $x^0,$   the condition \eqref{pcrit1}, i.e., $p>\max\{3/2, h(x^0,S)\},$   is still  sufficient for $\M$ to be $L^p$-bounded, and  if $\rho(x^0)\ne 0,$ it is also necessary.
In particular, here \eqref{pcrit1} still holds true.
\end{thm}

\subsection{Case of   $\A_{n-1}^+$} \label{A+}

The new feature arising in  this case  is that, in contrast to all previous situations, it will no longer just be the ``order of contact'' of $S$ at the point $x^0$ with  small balls, or slightly thickened hyperplanes, which basically determines the range of Lebesgue exponents  $p$ for which \eqref{Mb} holds true, but possibly also in some sense the ``order of contact'' with slightly thickened lines. Some instances of this kind of phenomenon had already been observed in  articles by  Nagel, Seeger,  Wainger \cite{nagel-seeger-wainger}, and Iosevich and Sawyer \cite{iosevich-sawyer1},  \cite{iosevich-sawyer}, \cite{io-sa-seeger}. 

\medskip
In order to formulate our conjecture   for this case, and  the results  that we are able to prove in this paper  towards this conjecture, we need to introduce a number of further quantities.  So, let us assume that $\phi$ is of type $\A_{n-1}^+.$ 
\medskip

 Besides the \textit{height}, we are here  interested in a further characteristic quantity of the surface, which we shall call the ``effective multiplicity''.

 In order to define this notion, in a first step we  decompose 
\begin{equation}\label{pred}
\phi\x=\phi(0,x_2)+\pred\x,
\end{equation}
and consider the function $\pred$  and its associated Newton polyhedron $\N(\pred).$

We then define the  number $n^x_e=n^x_e(\phi)$ by assuming that  $(n^x_e,1)$ is the point at which the horizontal line $t_2=1$ intersects the boundary of  $ \N(\pred)$ (see Figure \ref{figure1}). 
\smallskip

A look at  \eqref{AD} will reveal that the  point $(n,0)$   is a vertex  of $\N(\pred),$   and we denote by $\ga_1$ the (non-horizontal) edge of $\N(\pred)$  which has $(n,0)$  as   the right endpoint. 

Note that this edge can  possibly also be vertical, for instance if $\phi\x=x_2^2\pm x_1^n.$    We then   choose the weight 
$\ka^e=(\ka^e_1,\ka^e_2)=(1/n,\ka^e_2)$  with $\ka^e_2\ge 0$ so that the edge $\ga_1$  lies on the line $L_1:=\{t_1,t_2): \ka^e_1t_1+\ka^e_2t_2=1\}.$  Then $(n^x_e,1)\in L_1$,  which implies that (see Figure \ref{figure1}) 
\begin{equation}\label{ne}
n^x_e=\frac{1-\ka^e_2}{\ka^e_1}=(1-\ka^e_2) n.
\end{equation}

Note that clearly $n^x_e\le n.$ 
For later use, let us also  note that $\ka^e_2<1/2.$ 

Indeed, when passing from  the Newton polyhedron $\N(\phi)$ to $\N(\pred),$ at least all of the points of the Taylor support  $\T(\phi)$ of $\phi$ are removed from the   $t_2$-axis, in particular the point $(0,2),$ and since by property ($\A_{n-1}^+$)  there is also no point  of $\T(\phi)$ on  the line segment with endpoints $(0,2)$ and $(n,0),$ we see that the line $L_1$ must intersect the $t_2$-axis above  $2,$ i.e., $1/\ka^e_2>2,$  if $\ka^e_2>0,$ or must be the vertical line passing through the point $(n,0),$ if $\ka^e_2=0.$

\color{black}

\begin{figure}[h]
\centering
\includegraphics[scale=0.4]{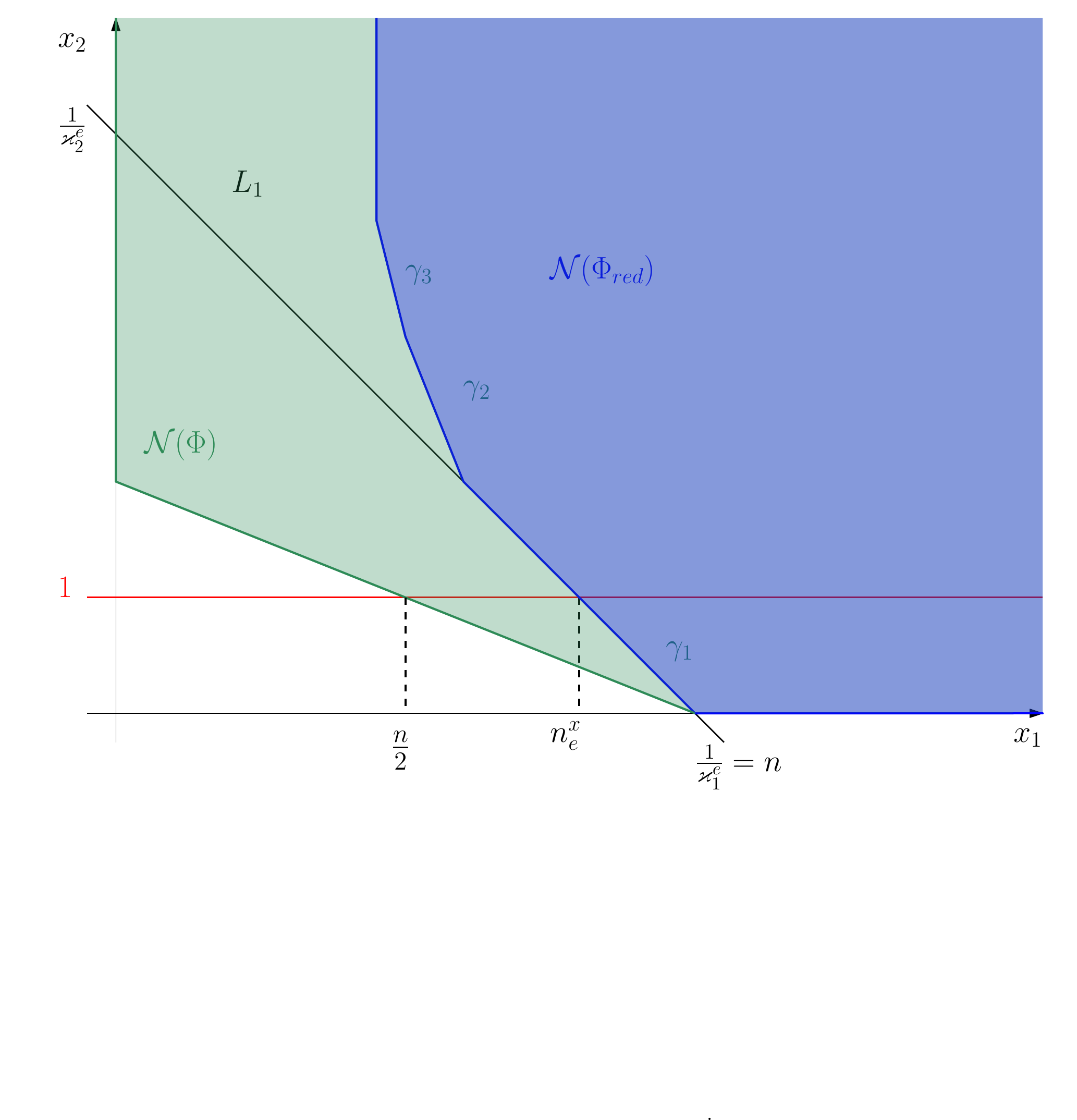}
 \caption{Effective Multiplicity} 
 \label{figure1}
\end{figure}

\medskip

\medskip

Next, somewhat in the spirit of Varchenko's algorithm for constructing adapted coordinates for $\phi$ (see, e.g., \cite{Va}, \cite{IM-ada}), we  shall allow for local coordinate changes (``non-linear shears'')  of the form
\begin{equation}\label{shear}
y_1:=x_1-\al(x_2), \quad y_2:= x_2,
\end{equation}
where $\al$ is  smooth and vanishes  at $x_2=0.$ 
Given such a coordinate system $y=(y_1,y_2),$  we express $\phi$ in these coordinates by putting 
$$
\tilde\phi(y_1,y_2):= \phi(y_1+\al(y_2),y_2).
$$

Indeed, as will be discussed  in more detail in Section \ref{lineadac}, we can even allow for a more general class of local coordinate changes  near the origin than in \eqref{shear}, of the form
\begin{equation}\label{gshear}
x_1:=\vp(y_1,y_2), \quad x_2= y_2,
\end{equation}
where $\vp$ is  smooth,  vanishes at the origin and  $\pa_1\vp(0,0)\ne 0.$ Such changes of coordinates will be called  {\it admissible}.  Here, we then put 
$\tilde\phi(y_1,y_2):= \phi(\vp(y_1,y_2), y_2).$
\smallskip

If we work in the category of analytic functions, we may and shall assume that the above changes of coordinates are analytic too, i.e., that $\al$ respectively $\vp$ are analytic.

\smallskip

{\bf Note:} The  coordinates $(y_1,y_2)$ are  also adapted to $\tilde \phi.$
\smallskip

Indeed, since $1/2>1/n,$ one easily checks that property   ($\A_{n-1}^+$) is preserved by these kind of  coordinate changes  \eqref{gshear}.

\smallskip
We next decompose again $\tilde\phi(y_1,y_2)=\tilde\phi(0,y_2)+\tilde \phi_{\rm red}(y_1,y_2)$ as we did for $\phi$ in \eqref{pred}, and compute for $\tilde \phi_{\rm red}$ as in \eqref{ne}  the corresponding number $n_e^y=(1-\tilde \ka^e_2)n,$ where $\tilde \ka^e$ is the weight corresponding to the edge $\tilde \ga_1$ with right endpoint $(n,0)$ of $\N(\tilde \phi_{\rm red}).$  Finally, we define 

$$
n_e=n_e(\phi):=\sup\limits_y n_e^y,
$$
where the supremum is taken over all analytic local coordinate systems $y=(y_1,y_2)$ of the form \eqref{gshear}. Note also  that $n_e\le n.$

We shall show in Section \ref{lineadac} that this supremum is indeed a maximum, i.e., there exists a coordinate system $y=(y_1,y_2)$ such that $n_e^y=n_e.$ This coordinate system will even be of the form \eqref{shear},  so that it  would indeed have sufficed to take the supremum over all coordinate systems \eqref{shear}  in the definition of $n_e.$ Any such coordinate system $y$ with $n^y=n_e$  will  be called  {\it line-adapted} to $\phi$ (not to be confused with the notion of ``linearly adapted'' coordinate systems introduced in \cite{IMmon}!), and the quantity $n_e$ will be called the {\it effective multiplicity} of $\phi.$ The latter notion  is motivated by the following:

\smallskip

Given $n_e,$ an important  exponent for the study of the maximal operator $\M$ associated to $\phi$ will be given by 
$$
p_e:= \frac {2n_e}{n_e+1}.
$$
For instance, every function which has a singularity of type $A_{n-1}$ can be written in suitable  local coordinates in the normal form 
$\pm x_2^2\pm x_1^n,$  which is of type $\A_{n-1}^+.$ For this normal form, we have $\ka^e_2=0,$ and $n_e=n^x_e=(1-\ka^e_2) n=n.$

Thus, the exponent $p_e$ associated to $\phi$ before is formally the same  as for the normal form $\pm x_2^2\pm x_1^{n_e}$ - this motivates our notion of effective multiplicity.

\medskip
We remark that indeed also for  $\A^-_{n-1}$ -singularities it is possible to define (in a less direct way) the notion of effective multiplicity (see Remark \ref{9effectivem}), so that this notion makes sense for any singularity of type  
$\A_{n-1}.$ However, we shall not really need to make use of this observation.

\smallskip 

\begin{remark}\label{pe-h}
If $\phi$ is of type $\A_{n-1}^+,$  with $n\ge 4,$ then $p_e>h(\phi).$ Moreover, we always have $n<2n_e,$ in particular $n_e>2,$  and $p_e\ge 3/2$ if and only if $n_e\ge 3.$ 
\end{remark}
Indeed, we have $h=d={2n}/(n+2),$ and $p_e\ge {2n^x_e}/(n^x_e+1).$ And, a look at the Newton polyhedra of $\phi$ and $\pred$ (cf. Figure \ref{figure1}) shows that the horizontal   line $t_2=1$ intersects the principal face of $\N(\phi)$ at the $t_1$-coordinate $ n /2<n_e^x\le n_e.$  But,  this implies that 
$ {2n_e}/(n_e+1)>h.$ The remaining statements are immediate.

\medskip

 Let us present some examples of  $\A_{n-1}^+$-type singularities. The first two examples show that the original coordinates $x$ may in general not be line-adapted, and that the   inequality $n_e\le n$ can indeed be strict:

\begin{example}\label{e1.3}
Let $\phi\x:= x_2^2+(x_1+x_2^\ell)^n,$ where $\ell \ge 2.$ 
\end{example}
 Then $\pred\x=(x_1+x_2^\ell)^n- x_2^{n\ell},$ so that $\ka^e_2=1/n\ell,$ if $c\ne 0,$ hence 
 $$n_e^x=(1-\frac 1{n\ell})n=n-\frac 1\ell.$$ 
 However, in the coordinates $y_1:=x_1+x_2^\ell, y_2:=x_2,$ we have $\tilde \phi(y_1,y_2)=y_2^2+y_1^n,$  hence $n^y_e=n,$ so that $n_e=n.$ In particular, in contrast to the coordinates $x,$ the coordinates $y$ are line-adapted to $\phi$ (which  follows easily from Proposition \ref{nonadapt}), and we have $p_e={2n}/(n+1).$

\begin{example}\label{e1.4}
Let $\phi\x:= (1+x_1^\al x_2^\be)x_2^2+x_1^n,\qquad \al,\be\in\NN.$
\end{example}
Here,  if $\al\ge 1,$ then $\pred\x=x_1^\al x_2^\be x_2^2+x_1^n,$  and if $\al=0,$ then $\pred\x=x_1^n.$ Thus one computes that $\ka^e_2=(n-\al)/{n(\be +2)},$ if $1\le\al< n,$  and $\ka^e_2=n,$ if $\al\ge n.$ Hence, by \eqref{ne},  $n_e^x=(n(\be +1)+\al)/(\be +2),$ if   $1\le \al< n,$  and $n_e^x=n,$ if $\al=0,$ or $\al\ge n.$ Moreover, it is easily seen that the coordinates $x$ are already line-adapted (which  follows again from Proposition \ref{nonadapt}), and thus we have $n_e=n_e^x<n$ iff $1\le \al<n.$  In particular, 
$p_e=2\big(n(\be+1)+\al\big)/\big(n(\be+1)+\al+\be+2\big),$ if  $1\le \al< n,$ and $p_e=2n/(n+1),$ if  $\al=0,$  or $\al\ge n.$

\smallskip

\begin{example}\label{e2.6}
Let $\phi\x:= \frac 1{1-x_1}x_2^2+x_1^n.$
\end{example}
In contrast to Example \ref{e1.4}, the graph of this function is   convex  for $|x_1|<1$  when $n$ is even.
Since 
$1/(1-x_1)=1+x_1+O(x_1^2),$ as in  Example \ref{e1.4} we have $n_e^x=(n+1)/{2},$ and  the coordinates $x$ are already line-adapted, so that  $n_e=(n+1)/{2}$ and   $p_e=2(n+1)/(n+3).$

\medskip

Finally, for functions of type $\A_{n-1}^+$ let us put 
$$
\tilde p_c:=\max\{3/2, p_e, h\}=\max\{3/2, p_e\},
$$
where $h:=h(\phi)=h(x^0,S).$  Note that in view of Remark \ref{pe-h} we have indeed that  $\tilde p_c=\max\{3/2, p_e\}.$
The following result will be proven in Section \ref{nec}:
\begin{proposition}\label{necA}
Assume that $S$ is the graph of $1+\phi,$ and accordingly $x^0=(0,0,1),$  where $\phi$ is of type $\A_{n-1}^+$. Then, if $\rho(x^0)\ne 0,$   the condition   $p\ge \tilde p_c$  is  necessary  for $\M$ to be $L^p$-bounded.
\end{proposition}

Our main conjecture on the boundedness of the maximal operator $\M$ for singularities of type $\A_{n-1}^+$ states that the condition stated before is essentially also sufficient:

\begin{conjecture}\label{conj1}
Assume that $S$ is the graph of $1+\phi,$ and accordingly $x^0=(0,0,1),$  where $\phi$ is of type $\A_{n-1}^+.$ Then, if the density $\rho$ is supported in a sufficiently small neighborhood  of $x^0,$   the condition   $p>\tilde p_c$   is  sufficient for $\M=\M_\rho$ to be $L^p$-bounded, so that indeed  $p_c(x^0)=\tilde p_c.$
\end{conjecture}

We shall show  that for a  large class of functions $\phi$ of type $\A_{n-1}^+$  this conjecture indeed holds true.
 To describe this class, let us assume that $\y$  is a line-adapted coordinate system in which the function $\phi\x$ is represented by $\tilde\phi\y.$

  Assuming  first that $\tilde\ka^e_2>0,$ let us  decompose the function $\tilde\phi_{\rm red}\y=\tilde\phi\y-\tilde\phi(0,y_2)$ into 
\begin{equation}\label{defp}
\tilde\phi_{\rm red}\y=p(y_1,y_2)+\tilde\phi_{\rm err}(y_1,y_2),
\end{equation}
with a $\tilde\ka^e$-homogeneous polynomial $p$ of $\tilde\ka^e$-degree 1 consisting of at least two monomial terms,  with Taylor support on the first non-horizontal edge $\tilde \ga_1$ of 
$\N(\tilde\phi_{\rm red}),$ which is compact and  has right endpoint $(n,0)$ (and a remainder term $\tilde\phi_{\rm err}$  consisting of terms of higher $\tilde\ka^e$-degree).

Assume next that $\tilde\ka^e_2=0.$ Then $\tilde \ga_1$ is vertical, and in this case $p\y$ may no longer be a polynomial, but of the form $y_1^n G\y,$ with $G$ analytic  and $G(0,0)\ne 0.$

\medskip
We shall be able to handle in this paper all functions of type $\A_{n-1}^+$, except for functions from a class that we shall denote by
$\A_{n-1}^e.$ To facilitate notation, let us assume without loss of generality  henceforth that $\beta(0)=1$  (see \eqref{AD}).
\medskip

 \noindent{\bf The class of functions  of type  $\A_{n-1}^+\setminus \A_{n-1}^e$.} 
   By definition, these are those functions $\phi$ from $\A_{n-1}^+$ which satisfy the following assumptions:
 \smallskip
 
Assume  that $y=\y$ is any line-adapted coordinate system.  Then either $\tilde\ka^e_2=0,$ or $\tilde\ka^e_2>0,$ and the $\tilde\ka^e$-homogeneous polynomial $p\y$ satisfies the following two conditions:
 
  \begin{itemize}
\item[(A1)] $\pa_{y_1}^2 p\y$ consists of  at least two distinct monomial terms, one of them of course being $ n(n-1)y_1^{n-2}.$
\item[(A2)] If $\pa_{y_1}^2 p\y$  does not  vanish of maximal possible order $n-2$ along a real, non-trivial root of $\pa_{y_1}^2 p\y;$
more precisely,  $\pa_{y_1}^2 p\y$ is not  of the form 
\begin{equation}\label{maxvan}
\pa_{y_1}^2 p\y=n(n-1)(y_1-cy_2^a)^{n-2},
\end{equation}
with $c\in \RR\setminus \{0\}$ and integer exponent $a\in \NN_{\ge 1}.$  
 \end{itemize}

 \medskip

 \noindent{\bf On the exceptional  class   $ \A_{n-1}^e$.} 
We remark that if \eqref{maxvan} {\it is} satisfied, i.e., if $\phi$ is of class $\A_{n-1}^e,$ then  $\pa_{y_1} p$ must be of the form 
$\pa_{y_1} p\y=n(y_1-cy_2^a)^{n-1}+c_1 y_2^{a(n-1)},$ with $c_1\ne 0,$ for if $c_1=0,$ then the coordinates $y$ could not have been line-adapted by Corollary \ref{nonadaiff}. Consequently, we would have 
\begin{equation*}
 p\y=(y_1-cy_2^a)^n+c_1 y_1y_2^{a(n-1)} +c_0 y_2^{an}, \qquad \text{with} \ c_1\ne  0.
\end{equation*}
Note that we could here perform a change of coordinates $u_1:= y_1-cy_2^a, u_2=y_2,$  so that in the new coordinates $u=(u_1,u_2),$ we would have 
$$
p\y=u_1^n+c_1 u_1u_2^{a(n-1)} +(c_0+c c_1) u_2^{an},
$$
where $a(n-1)\ge 3,$ if $n\ge 3,$ so that we may write $a(n-1)=\beta +2,$ with $\beta\in\NN.$ Since the last term depends on $u_2$ only,  if  then we look at the function $\tilde {\tilde \phi}(u_1,u_2) $ which represents $\phi$ in these coordinates, we  see that $\tilde {\tilde \phi}_{\rm red}(u_1,u_2) $ is of the form
$$
\tilde {\tilde \phi}_{\rm red}(u_1,u_2) = u_1^n+c_1 u_1u_2^{\beta +2}+\tilde{\tilde\phi}_{\rm err}(u_1,u_2),
$$
and by the discussion in Section \ref{lineadac} it is easy to see that the coordinates $u$ are line-adapted too. Up to the error term, we see that this is now of the form described in Example \ref{e1.4}, with $\al =1.$ 
\smallskip

On the other hand, if we assume that $(A1)$ is not satisfied, then we see that $p\y$ must be of the form
$$
p\y=y_1^n +c_1 y_1 y_2^k,
$$
with $c_1\ne 0, $ and since $\tilde\ka^e_2<1/2,$ it is easily seen that $k\ge 2,$ so that again 
\begin{equation}\label{extype}
p\y=y_1^n +c_1 y_1 y_2^{\beta+2}, \qquad \text{with} \ \beta\in\NN.
\end{equation}

\smallskip 

{\it Thus, we see that $\phi$ is of exceptional type $ \A_{n-1}^e$ if and only if there is a line-adapted coordinate system in which $p(y)$ is of the form \eqref{extype}.}
\medskip

If we look at the examples, we thus see that the  functions $\phi$ in Example \ref{e1.4}
are of type $\A_{n-1}^e$ if and only if $\al=1.$ Another example of type  $\A_{n-1}^e$ is Example \ref{e2.6}, which is a  variant of the case $\al=1,\beta=0$ in Example  \ref{e1.4} in which $\phi$  is  even convex whenever $n\ge 2$ is an even number.

\color{black}
\smallskip

\begin{example}
Let $\phi\x:= x_2^2+(x_1+x_2^\ell)x_2^k+(x_1+x_2^\ell)^n,$ where $k\ge 2, \ell \ge 1$ and $\ell (n-1)<k.$ 
\end{example}
Here one checks that $\ka_1^e=1/n, \ka_2^e=1/n\ell.$ If we pass to the admissible coordinates $\y$ with $y_1:=x_1+x_2^\ell, y_2:=x_2,$ then $\phi$ is represented by $\tilde\phi\y= y_2^2+y_1y_2^k+y_1^n,$  with $\tilde \ka_1^e=1/n, \tilde \ka^e_2=(n-1)/nk.$ And, one easily verifies by means of Proposition \ref{nonadapt} that the coordinates $\y$ are line-adapted to $\phi,$ whereas the coordinates $\x$ are not, since $1/\ka^e_2<1/\tilde\ka^e_2.$ 
Moreover, $\tilde\phi_{\rm red}\y=y_1y_2^k+y_1^n,$ which shows that $\phi$ is of  type $\A_{n-1}^e.$

\smallskip

Our main result on functions of type  $\A_{n-1}^+$ in this paper is

\begin{thm}\label{thm-a+}
Assume that $S$ is the graph of $1+\phi,$ and accordingly $x^0=(0,0,1),$  where $\phi$ is  analytic and of type $\A_{n-1}^+\setminus\A_{n-1}^e$. Then, if  the density $\rho$ is supported in a sufficiently small neighborhood  of $x^0,$  the  condition  $p>\max\{3/2, p_e\}$  is  sufficient for $\M$ to be $L^p$-bounded.  Moreover, if $\rho(x^0)\ne 0,$ it is also necessary.
I.e., in this case we have $p_c(x^0)=\max\{3/2, p_e\}.$  
\end{thm}

We shall turn to the remaining exceptional class of functions of type $\A_{n-1}^e$ in the third paper of this series. However, 
in the last Section \ref{isasee}, we shall at least show that the condition $p>\max\{3/2, 2n/(n+1)\}$  is  sufficient for $\M$ to be $L^p$-bounded. In combination with all other results, this will lead to a proof of a conjecture by Iosevich-Sawyer-Seeger  for arbitrary analytic 2-surfaces (compare Subsection \ref{isseconj}).

\subsection{A more general conjecture in arbitrary dimension}\label{genconj}
Let us come back to  more general hypersurfaces $S\subset \RR^d$ and their associated maximal operators $\M$ as discussed at the beginning of this section. Assuming that $S$ satisfies the transversality Assumption \ref{s1.1} near a given point $x^0\in S$ and is contained in a sufficiently small neighborhood of $x^0$  (as in the 3-dimensional case)  we may assume that,  after applying a suitable linear change of coordinates in $\RR^d,$ we have  $x^0=(0,1)$ with respect to the splitting $\RR^d=\RR^{d-1}\times \RR,$ and that $S$ is the graph
$$
 S=\{(x',1+ \phi(x')): x'\in \Om \}
$$
of a smooth function $1+\phi$ defined on an small  open neighborhood $\Om$
of $0\in\RR^{d-1}$ satisfying
$$
\phi(0)=0,\, \nabla \phi(0)=0.
$$

By localizing to a sufficiently small neighborhood of $x^0,$ we may also assume that
\begin{equation}\label{normali}
\Om=\Om_\ve=\{z'\in\RR^{d-1}: |z'|< \ve\} \quad   \text {and} \quad |\nabla \phi(z')| \le \frac 1{10}\quad \text{for all } z'\in \Om,
\end{equation}
where $0<\ve<1$ can be assumed to be small.
 
Let us also assume here that $\rho\ge c>0$ on $S.$  Then, without loss of generality, we may even assume that 
\begin{equation}\label{At1}
A_t f(y) = \int_\Omega f(y'-tx', y_d - t(1+\phi(x'))) \, dx' .
\end{equation}

The following proposition, which yields necessary  conditions of geometric measure theoretic type for the $L^p$- boundedness of the maximal operator 
$\M_S,$    will be proved in Section \ref{nec} under these  assumptions.

\begin{proposition}\label{prop-nec}
Let $T\subset \RR^d$ be a symmetric convex body of positive volume $|T|>0$. Assume further that \eqref{normali} is satisfied, and let   $z_d:\Om_\ve\to \RR$ be any  $C^1$- function so that 
  \begin{equation}\label{normaliz}
  z_d(0)=1,\quad \nabla z_d(0)=0, \quad \text{and}\quad  |\nabla z_d(z')|\le \frac 16 \quad \text{for all } \ z' \in \Om_\ve.
\end{equation}
Put $T(z'):=(z',z_d(z'))+T$ for $z'\in \Om_\ve,$ and 
$$
|T(z')\cap S|:=\int_{\Om_\ve} \chi_{T(z')}(x',1+\phi(x'))\, dx'.
$$
Then, for $1\le p<\infty,$ we have 
\begin{equation*}
\int_{\Om_\ve} \frac{|T(z')\cap S|^p}{|T|}\, dz' \le 4  (C_d \ve^{d-1})^{-p}\|\M\|_{L^p\to L^p}^p.
\end{equation*}
\end{proposition}

\begin{remarks}
(i) The same result would hold more generally for any star-shaped body $T,$ but that fact does not seem to be relevant for our applications.

\smallskip
\noi (ii) In view of John's theorem, we may basically  replace the class of all convex bodies $T$ in this result by  the class of all ellipsoids centered at the origin.
\smallskip

\noi (iii) Our proof will show that we even have the estimate
$$
\int_{\Om_\ve} \frac{|T(z')\cap S|^p}{|T|}\, dz' \le 4 (C_d \ve^{d-1})^{-p} \|\M_{\rm loc}\|_{L^p\to L^p}^p,
$$
where 
$$
\M_{\rm loc} f(x):=\sup_{1/2\le t\le 1}|A_tf(x)|.
$$
\end{remarks}

Under the assumptions \eqref{normali} on $\phi,$   let us put 
$$p_{\rm geom}(x^0):=\inf p,
$$ 
where the infimum is taken over all $p\ge 1$ for which there is  some $\ve >0,$ and a constant $C_{p,\ve}\in \RR$ such that 
\begin{equation*}
\int_{\Om_\ve} \frac{|T(z')\cap S|^p}{|T|}\, dz'\le C_{p,\ve}
\end{equation*}
holds true for all ellipsoids $T$ centered at the origin and all $C^1$-functions  $z_d:\Om_\ve\to \RR$ satisfying the assumptions   \eqref{normaliz}. 
\smallskip

As a corollary to  Proposition \ref{prop-nec}, we obtain
\begin{cor}\label{pgpc}
If $p<p_{\rm geom}(x^0),$ then the maximal operators $\M$ is unbounded on $L^p(\RR^d),$
 so that  in particular $p_{\rm geom}(x^0)\le p_c(x^0).$
\end{cor}

\begin{conjecture}
If $p>p_{\rm geom}(x^0),$ then the maximal operator $\M=\M_\rho$ is bounded on $L^p(\RR^d),$ provided $\rho$ is supported in a sufficiently small neighborhood of $x^0,$ so that indeed $p_c(x^0)=p_{\rm geom}(x^0).$ 
\end{conjecture}
\color{black}

  Our discussion in Section \ref{nec} will reveal  that this conjecture holds indeed true in dimension $d=3,$  with the possible exception of hypersurfaces of type $\A^e,$  for which, however, it  is  equivalent to Conjecture \ref{conj1}.

  \medskip
  \subsection{A guide through the organization of the paper and the proofs of  our main results}
Section \ref{nec} will be devoted to the proof of Proposition \ref{prop-nec}, from which we shall also deduce  the necessity of the conditions in our main Theorems \ref{thm-a-} and \ref{thm-a+}. Moreover, we shall  compare our results with results and a conjecture by 
  Iosevich, Sawyer and Seeger from their article \cite{io-sa-seeger}.
  
  In preparation of the proof of Theorem \ref{thm-a+}, in Sections \ref{lineadac} and \ref{legendretrans} we  shall study functions $\phi=\phi(x_1,x_2)$ of  type $\A^+_{n-1}$ and their Legendre transforms $\breve\phi=\breve\phi(x_1,s_2)$ with respect to the second variable $x_2.$ The latter will be crucial in order  to obtain very precise information on the partial Fourier transform of the measure $\mu$ with respect to $x_2.$ In particular, we shall show in Section \ref{lineadac} how to construct line-adapted coordinates for $\phi,$ and give a necessary and sufficient condition for a given coordinate system to be line-adapted. And, in Section \ref{legendretrans}, we shall show that the classes $\A^+_{n-1}$ and $\A^e_{n-1}$ are invariant under this Legendre transform, and that the principal parts  $\pred$ of $\phi$ and $\breve\phi_{\rm red}$ of $\breve \phi$ are ``essentially'' the same, a fact which will be crucial in order to be able to  transfer the conditions $ (A1), (A2)$ on $\pred$ into the same conditions on $\breve\phi_{\rm red}.$ 
  
  Next, in Section \ref{prepsteps}, we shall provide some auxiliary results which will be used frequently later on.  First, we recall  a well-know version of a van der Corput type estimate for one-dimensional oscillatory integrals. Then, in Lemma \ref{pinvers}, we prove an  identity for two-dimensional oscillatory integrals  which has the flavour of a stationary phase identity and which will become crucial later on whenever we shall have to deal with the most  challenging  oscillatory integrals arising in our later studies. Next, we shall recall a lemma from our preceding paper \cite{bdim19} which will be applied whenever we shall prove $L^p$ - estimates  for $p$ close to $1$  for maximal operators associated to ``microlocalized'' pieces of our measure $\mu.$ Finally, we shall perform some preliminary reductions which will allow to spectrally localize the measure $\mu$ to frequencies $\xi=(\xi_1,\xi_2,\xi_3)$ where 
  $$
   |\xi_3|\sim\la\gg 1\text { and  } |s_1|,|s_2|\ll 1,\text {  with  }   s_1:=\frac{\xi_1}{\xi_3}, s_2:=\frac{\xi_2}{\xi_3}, s_3:=\frac{\xi_3}\la.
   $$
   Here,$ \la$ is a dyadic number, and the resulting measure will be denoted by $\mu^\la.$
\smallskip

The Sections \ref{proof-} to \ref{endproof1} will be devoted to the proof of the sufficiency of the conditions in Theorem  \ref{thm-a-}.
  We begin by applying the method of stationary phase in the coordinate $x_2;$ in particular, we determine the Legendre transform 
  $\breve\phi(x_1,s_2)$  (cf. \eqref{legendamins}) and show that the Fourier transform of $\mu^\la$ can be written as 
$$
\widehat{\mu^\la}(\xi)=\la^{-\frac 12} e^{-i\la s_3\big(s_2^2B(s_2)+1\big)}  J(\la, s) \, \chi_0(s_3s')\chi_1(s_3) ,
$$
where $J(\la, s)$ denotes the one-dimensional oscillatory integral
$$
J(\la, s):=\int e^{-i\la s_3\Phi_0(x_1,s')}  \eta(x_1,  s_2)\,dx_1,
$$
with phase $\Phi_0(x_1,s'):=\breve\phi_1(x_1, s_2)+s_1x_1,$ where $s':=(s_1,s_2).$
Quite important for us will also be the  two-dimensional oscillatory integral
$$
 F(\la,s_2,s_3,y_1) := \iint e^{-i\la s_3(\Phi_0(x_1,s')-s_1y_1)}  \eta(x_1,  s_2)  \chi_0(s_3s')\chi_0(s_1) \,dx_1 ds_1,
$$
which will allow us to write
$$
\mu^\la(y+\Ga) =\la^{\frac52}
\int \Big(\int F(\la,s_2,s_3,y_1) e^{-i\la s_3(s_2^2B(s_2)-s_2 y_2-y_3)} \chi_0(s_2) ds_2\Big) \chi_1(s_3) ds_3,
$$
where we have put $\Gamma:=(0,0,1).$ 

Our estimates for the maximal operator $\M^\la$ associated to the measure $\mu^\la$ will be based on an interpolation between $L^2$ - estimates and $L^{1+\ve}$ - estimates  for $\ve>0$ sufficiently small. In view of Plancherel's theorem, for these $L^2$- estimates, we shall need to estimate the oscillatory integrals  $J(\la, s),$ whereas for our $L^{1+\ve}$-estimates,  pointwise bounds on $\mu^\la(y+\Ga)$ will be required in order to be able to apply Lemma \ref{maxproj}.

As for the oscillatory integral $J(\la, s),$ which strongly depends on the parameter $s_2,$ the strategy will be to fix a certain dyadic level of $s_2,$ say where $s_2\sim 2^{-k}\ll 1,$ and then try to decompose the support of $\eta$ in $x_1$ into intervals over which the second derivative
$$
\Phi(x_1,s_2):=\pa_{x_1}^2\Phi_0(x_1,s') =\pa_{x_1}^2\breve\phi_1(x_1, s_2)
$$
will essentially be of a certain level size for these values of $s_2,$ so that for these  parts of the oscillatory integral we can apply van der Corput's estimate of order $2.$  Clearly, in order to perform this, we shall need to understand the null set  $\Phi(x_1,s_2)=0.$ It is well-known that this set is the union of  finitely many curves, either of the form $x_1=0,$  or $x_1=r(s_2),$ with non-trivial ``real roots'' $r(s_2)$ which can be expanded as Puiseux series in $s_2$ -- these facts will be re-called in Subsection \ref{roots}, as well as their connections with the two-dimensional Newton-polyhedron $\N(\Phi)$ associated to $\Phi.$ We are here following  ideas from \cite{phong-stein}. These roots come in clusters, the first type of cluster being determined by the leading exponent  of the roots (we also have to consider complex roots!). These in return are in one to one correspondence with the compact  edges of $\N(\Phi):$ if $a=a_l$ is the leading exponent of all roots in  a given cluster, then there is a unique compact edge $\ga=\ga_l$ such that $a$ is just the modulus of the slope of this edge. 

Moreover, each compact edge $\ga_l$ comes with a certain (typically) anisotropic scaling structure, which allows to decompose our $(x_1,s_2)$ - domain into certain domains $D_l$ which are invariant under the scalings  associated to $\ga_l,$ and transition domains $E_l$ in between these  homogeneous domain (see Figure \ref{domains}). This decomposition will be explained in Section \ref{step1}. Now, on each transition domain, it turns out that we do have a good resolution of singularities of $\Phi;$ indeed, by \eqref{VLaurent},
$$
\Phi(x_1,s_2)=V(x_1,s_2)\, s_2^{A_l}x_1^{B_l},\qquad (x_1,s_2)\in E_l,
$$
where $|V(x_1,s_2)|\sim 1,$ and where $(B_l,A_l)$ is a vertex of $\N(\Phi)$ associated to $E_l.$ 
Thus, if we assume that $x_1\sim 2^{j},$ then we see that $|\Phi(x_1,s_2)|\sim 2^{-kA_l-jB_l},$ so that we can apply van der Corput's estimate to the corresponding part  $J_{j,k}(\la,s)$ of $J(\la,s)$ and get an appropriate $L^2$ - estimate for the associated maximal operator 
$\M^\la_{j,k}.$ 

However, a considerably bigger  challenge comes with the $L^{1+\ve}$ - estimates of $\M^\la_{j,k}.$ Here, we need to understand  the original phase 
$\Phi_0$ and the associated functions $F_{j,k}$ and $\mu^\la_{j,k}.$ We show  in \eqref{Phi0onEl} that, for suitable functions $g$ and $h,$ 
$$
\Phi_0(x_1,s')=\breve\phi_1(x_1, s_2)+s_1x_1=U(x_1,s_2)s_2^{A_l} x_1^{B_l+2}+x_1(s_1+g(s_2))+h(s_2),
$$
where $|U(z_1,s_2)|\sim 1,$ and one problem will be to gain a suitable control on these functions $g$ and $h$ (which is not very difficult in this first step of our resolution algorithm, but becomes more involved later on) in order to later also be able to control the integration in $s_2$ in our integral formula for $\mu^\la(y+\Ga).$ 

Once the $L^{1+\ve}$ - estimates are established, we can interpolate in order to obtain $L^p$-estimates for the maximal operators
$\M^\la_{j,k}.$ In order to show that these estimates can indeed be summed over all  dyadic $\la\gg 1$ and all indices  $j,k$ associated to the domain $E_l,$ we 
also need the important ``Geometric Lemma'' \ref{geo} and its corollary, Lemma \ref{key}, which will allow us to relate information from the Newton polyhedron to the exponents which appear in our $L^p$-estimates.

We are then left with the contributions by the homogeneous domains $D_l,$ which will contain all real ``root curves'' of the associated cluster. Given any leading  term $cs_2^a$ of some real root in the  cluster, we can then essentially localize to a narrow, homogeneous neighborhood of the curve $x_1=cs_2^a,$ which will contain all ``root curves'' with this leading term.

In order to study the contribution by this narrow domain, we perform a change of coordinates, by putting $x_1-cs_2^a=:\tilde z,$ 
and express everything in these new coordinates $(\tilde z,s_2)$ in place of $(x_1,s_2).$ In particular, we can then express the function $\Phi(x_1,s_2)$ as a function $\tilde \Phi(\tilde z,s_2).$
 All this will be done in Section \ref{step1}.

In a second step of our resolution  of singularities algorithm, we shall then consider the Newton polyhedron $\N(\tilde \Phi),$ and try to iterate the procedure from the first step, with $\Phi$ replaced by  $\tilde\Phi.$ This will be performed in Section \ref{step2}. The $L^{1+\ve}$ - estimates  for the corresponding transition domains $\tilde E_l$ will become even more involved in this Step 2. There are many subtleties arising in this process, let us just highlight one: Lemma \ref{key} requires that the first compact edge of the Newton polyhedron has a left endpoint $(B_1,A_1)$ with $A_1\ge 1.$ This is immediate in Step 1, but no longer in Step 2, where this condition may fail.  But, it turns out that the cases where this happens to fail can still be handled, since we do have a good resolution of singularities in these cases  (see Subsection \ref{A1ge1}).

Proceeding in this way with our algorithm,  in order to analyze the contributions by the new homogeneous domains $\tilde D_l,$ we have again to perform  changes of coordinates by subtracting second order terms of the Puiseux series expansions of  roots from $\tilde z,$ in order to arrive at Step 3, and so forth. All this will be explained in Section \ref{endproof1}. Iterating this procedure, we shall, step by step, narrow down our considerations to smaller and smaller homogeneous domains containing only root curves from sub-clusters of higher and higher order, which, after a finite number of steps, will only contain one, real root $r(s_2),$ but  possibly with multiplicity. However, in this case one can obtain a good resolution of singularity near this root, and can then essentially apply the same techniques that we had used for transition domains also to such  kind of  homogeneous domains.

\smallskip
The Sections \ref{proof+} to \ref{9endproof1} will be covering  the proof of the sufficiency of the conditions in Theorem  \ref{thm-a+}.
By and large, we shall be able to follow the proof of Theorem  \ref{thm-a-} and shall therefore only highlight those parts of the proof which will require different or modified arguments. Major differences are caused by the change of coordinates \eqref{shear} which leads to line-adapted coordinates.   Moreover, a crucial new tool will be the ``Multiplicity Lemma`` \ref{multi}, which, under our assumptions $(A1),(A2),$ implies that in Step 2  of our resolution algorithm, the case where $A_1<1$ can only arise when all non-trivial real roots have multiplicity $1,$ so that we can then argue in a similar way as for Theorem  \ref{thm-a-}. 
Another difference is the way how we control the functions $g$ and $h$ above; this, however,  turns out to be even easier here than in the case of singularities of type $\A^-_{n-1},$ due to the fact that here $\tilde\ka^e_2<1/2.$ 
\smallskip

Finally, in Section \ref{isasee}, we shall indicate how the proof of Theorem \ref{thm-a+} can be modified to give a proof of the Iosevich-Sawyer-Seeger conjecture for surfaces of class $\A^e_{n-1}.$

  \bigskip
  
 {\bf Conventions:}  Throughout this article, we shall use the ``variable constant'' notation, i.e.,  many constants appearing in the course of our arguments, often  denoted by  $C,$  will typically have different values  at different lines.  Moreover, we shall use  symbols such as  $\sim, \lesssim$ or $\ll$ in order to avoid  writing down constants, as explained in \cite[Chapter 1]{IMmon}. 
By $\chi_0 $ and $\chi_1$ we shall  denote  smooth cut-off functions on $\RR$ with typically  small compact supports,
where $\chi_1$ vanishes near the origin and is identically one  near $1,$ whereas $\chi_0$ is identically $1$ on a small neighborhood of the origin. 
  These cut-off functions may also vary from line to line, and may  in some instances, where several of  such functions of different variables appear within the same formula, even   designate different functions.

Also, if we speak of the {\it slope}  of a line such as a supporting line to a Newton polyhedron, then we shall actually  mean the modulus of the slope. 

\medskip
{\noindent {\bf Acknowledgement:} The authors wish to thank Spyros Dendrinos for helpful annotations on the paper and in particular  for suggesting Example \ref{e2.6}.

\setcounter{equation}{0}
\section{Necessary conditions}
\label{nec}

Let us first prove  Proposition \ref{prop-nec}.
\vspace{1em}

\noindent {\it Proof of Proposition \ref{prop-nec}.} Recall that $\Om=\Om_\ve=\{z'\in\RR^{d-1}: |z'|< \ve\},$ where $0<\ve<1,$ and let
\[
B:=\{t(z',z_d(z')) : z'\in \Omega, 1/2 \le t \le 1\}.
\]
 If $y\in B$, then $y=t(z',z_d(z'))=tz$, for suitable $t=t(y), z'=z'(y)$ and $1/2 \le t\le 1$. Then,  by \eqref{At1}, for such $y,t,z$,
\begin{eqnarray*}
A_t \chi_T(y) &=& \int_\Omega \chi_T(y'-tx', y_d - t(1+\phi(x')))\, dx' \\
&=& \int_\Omega \chi_{\frac{1}{t}(y+T)}(x',1+\phi(x'))dx' \\
& \ge & C_d \ve^{d-1} |T(z')\cap S|,
\end{eqnarray*}
since $T$ is symmetric and convex, which implies that $\frac{1}{t}(y+T) \supset T(z')$. Thus,
\[
|T(z')\cap S| \le \frac 1{C_d \ve^{d-1}} \M \chi_T(tz) 
\]
for 
$ 1/2\le t \le 1$. But, $F:(t,z') \mapsto (tz', tz_d(z'))=tz$ is a $C^1$- change of coordinates, with   Jacobian $J_F$  satisfying 
$$
|J_F(t,z')|=t^{d-1}|z_d(z')- z'\cdot \nabla z_d(z')|,$$
so that, by \eqref{normaliz}, $1/2^d\le |J_F(t,z')|\le 2.$
Hence
\begin{eqnarray*}
\int_\Omega |T(z')\cap S|^p dz' & \le & 2 (C_d \ve^{d-1})^{-p} \int_\Omega \int_{1/2}^1 \M\chi_T(tz',tz_d(z')))^p dt dz' \\
& \le & 4(C_d \ve^{d-1})^{-p}  \int_B \M \chi_T(y)^p dy \\
& \le & 4 (C_d \ve^{d-1})^{-p}  \|\M\|^p_{p\to p} |T|.
\end{eqnarray*}
\hfill Q.E.D.
\medskip

In applications one would choose the function $z_d(z')$ so that $|T(z')\cap S|$ becomes essentially as large as possible. 

\subsection{The three dimensional case}
Let us now turn our attention to the $d=3$ case. We will be looking at $T$ of the form $T=[-\de_1,\de_1]\times [-\de_2,\de_2] \times [-\de_3,\de_3]$ with $\de_i\le 1.$ Let
\[
I:=\{i\in \{1,2\}:\de_i \sim 1\}.
\]
We will be taking $\de_i \ll 1$ for $i\not\in I$. Also let $k:=|I|$ and, given $z'=(z_1,z_2)\in \Om,$  let us consider $(z_1^0,z_2^0),$ where $ z_i^0:=0,$  if $i\in I,$ and $z_i^0:=z_i,$ if $i\not\in I$.  It is then plausible that a  good  choice for the function $z_3(z_1,z_2)$ might be
\begin{equation} \label{z3}
z_3(z_1,z_2) := 1 + \phi(z_1^0,z_2^0).
\end{equation}
We deduce necessary conditions for the $L^p$-boundedness of $\M$ based on the following three examples, where the cases $k=0$ and $k=2$ can indeed by treated by this kind of choice of $z_3(z_1,z_2).$ In the case $k=1,$ we will first have to apply suitable changes of coordinates of the form \eqref{shear} before making this kind of choice. 

\begin{example}
The case $k=0$.
\end{example}
Here $I=\emptyset$, i.e., $\de_i \ll 1, i=1,2,3$, and we choose $z_3(z_1,z_2) := 1 + \phi(z_1,z_2)$. Then if we choose $\de_1=\de_2=\de_3=\de\ll 1$,  meaning that $T$ is a $\delta$-cube, we have $|T(z_1,z_2)\cap S|\sim \de^2$ for all $(z_1,z_2)\in \Om$, which implies that for all $\de\ll 1,$ 
\[
 \|\M\|^p_{p\to p} \gtrsim \int_\Om \de^{2p-3} dz_1dz_2 \sim \de^{2p-3},
\]
giving the necessary  condition $p\ge 3/2$ $\M$ to be bounded on $L^p.$

\begin{example}
The case $k=2$.
\end{example}
Here  we choose $I=\{1,2\}$, hence only $\de_3=\de\ll 1,$  and thus $T$ is essentially a $\de$-slab. We choose $(z_1^0,z_2^0)=(0,0)$, so $z_3(z_1,z_2) = 1 + \phi(0,0) =1$. Then, the condition $|1+\phi(z_1,z_2) - z_3| = |\phi(z_1,z_2)| < \de$ implies that $|T(z_1,z_2)\cap S| \sim \de^{1/h}$ for all $(z_1,z_2)\in \Om$  (i.e., $|(z_1,z_2)|<\ve$), provided $\ve>0$ is chosen sufficiently small. Here,  $h=h(x^0)$ denotes  the height of $\phi$. This in turn implies that
\[
\|\M\|^p_{p\to p} \gtrsim_\ve \int_\Om \de^{p/h-1} dz_1dz_2 \sim \de^{p/h-1},
\]
giving  the  necessary condition $p\ge h$ for $\M$ to be bounded on $L^p.$

Note that the same computations show that necessarily 
\begin{equation}\label{onpgeom}
p_{\rm geom}\ge \max\{3/2, h\}.
\end{equation}
Since we have seen in the Introduction that \eqref{pcrit1} holds true for all classes of 2-hypersurfaces, with the  exception of hypersurfaces of type $\A_{n-1}^+,$  and since by Corollary \ref{pgpc}  we have $p_{\rm geom}\le p_c,$ \eqref{onpgeom} shows that  indeed $p_{\rm geom}=p_c$ for these classes of surfaces.

\begin{example}
The case $k=1$.
\end{example}

We  can here focus  on the remaining class of functions of type $\A_{n-1}^+.$   Following the discussion in Section \ref{introduction} just before Examples \ref{e1.3} and \ref{e1.4}, we may perform any  local coordinate change as in \eqref{shear} and denote by $y=(y_1,y_2)$  the new coordinate system.

  Recall the decomposition $\phi(x_1,x_2) = \tilde\phi(y_1,y_2)=\tilde\phi(0,y_2)+\tilde \phi_{\rm red}(y_1,y_2)$. Note that $\tilde\phi(0,y_2) = y_2^2 \rho(y_2)$ for some analytic function $\rho$ at the origin. Denote by $\tilde\ga_1$ the first  edge of the Newton polyhedron $\N(\tilde \phi_{\rm red})$ of 
 $\tilde \phi_{\rm red}$ with left endpoint $(n,0)$ (which must intersect  the line $t_2=1$), and by $\tilde L_1:=\{\tilde\ka ^e_1t_1+\tilde\ka ^e_2t_2=1\}$ the line  supporting $\tilde\ga_1.$ Then   
 $\tilde\ka ^e_1n_y^e+\tilde\ka ^e_2=1,$   hence (cf. \eqref{ne})
 $$n_e^y=(1-\tilde\ka ^e_2)/\tilde\ka ^e_1.$$

  Assume first that $\tilde\ka ^e_2>0,$ so that $\tilde\ga_1$ is a compact edge. 
  
  We  choose $I=\{1\}$, so that 
\[
T=[-1,1]\times [-\de_2,\de_2]\times [-\de,\de],
\]
and then choose $z_3$ according to \eqref{z3}, but with respect to the coordinates $(y_1,y_2)$. Noting that $y_2=x_2,$ this means that we put
\[
z_3(z_1,z_2):=1+\tilde\phi(0,z_2) = 1+z_2^2\rho(z_2).
\]
Assume that  the point $(x_1,x_2,1+\phi\x)$ on $S$ lies in $T(z_1,z_2).$ Then
$$
|x_1-z_1|\le 1, \ |x_2-z_2|\le \de_2, \ |1+\phi(x_1,x_2)-z_3(z_1,z_2)| \le \de.
$$
But,  we have
\begin{eqnarray}\label{k=1-1} \nonumber
|1+\phi(x_1,x_2)-z_3(z_1,z_2)| &=&|1+\tilde\phi(y_1,y_2)-z_3(z_1,z_2)| \\ 
&\le& | y_2^2 \rho(y_2)- z_2^2 \rho(z_2)|+|\tilde \phi_{\rm red}(y_1,y_2)|.
\end{eqnarray}

Note next that since $\tilde \phi_{\rm red}$ agrees with  a $\tilde\kappa^e $-homogeneous polynomial of degree $1,$ up to some error term of higher 
$\tilde\kappa^e $-degree, we see that 
$$
|\tilde \phi_{\rm red}(y_1,y_2)|\lesssim |y_1|^{1/\tilde \ka^e_1}+|y_2|^{1/\tilde\ka^e_2},
$$
provided $\ve>0$ is sufficiently small in our definition of $\Om.$
We also  recall that $\tilde\ka^e_1=1/n.$

Thus, to make sure that $|\tilde \phi_{\rm red}(y_1,y_2)|\le\de/2,$ it suffices to assume that
\begin{equation*}
|x_1-\alpha (x_2)|=|y_1|\le c \de^{ \tilde \ka^e_1} \quad \text {and} \quad |x_2|=|y_2|\le 2c\de^{ \tilde \ka^e_2},
\end{equation*}
where $0<c<1$  is sufficiently small.
Let us therefore  assume in what follows that $|z_2|\le c\de^{ \tilde \ka^e_2}$ and $|x_2|\le 2c\de^{ \tilde \ka^e_2}.$
Then also 
$$
 | y_2^2 \rho(y_2)- z_2^2 \rho(z_2)|\lesssim |x_2-z_2|(|x_2|+|z_2|)\le \de_2 4c\de^{ \tilde \ka^e_2}.
$$
We therefore choose $\de_2:= \de^{1-\tilde\kappa^e_2}$ in the above definition of $T,$  i.e., we consider 
$$T=T_\de:=[-1,1]\times [-\de^{1-\tilde\kappa^e _2},\de^{1-\tilde\kappa^e _2}]\times [-\de,\de].$$
Then, \eqref{k=1-1} and the subsequent estimates show that  
$$
|1+\phi(x_1,x_2)-z_3(z_1,z_2)|\le\delta,
$$
if $c$ is sufficiently small. To summarize: if we define the set 
$$
A_\de(z_1,z_2):=\{\x: |x_1|\le \frac 12,  |x_2|\le 2c\de^{ \tilde \ka^e_2}, \,|x_2-z_2|\le c\de^{1-\tilde\kappa^e _2}, |x_1-\alpha (x_2)|\le c\de^{\tilde\kappa^e _1}\},
$$
then if $|z_1|\le 1/2,\, |z_2|\le c\de^{ \tilde \ka^e_2}$ and $(x_1,x_2)\in A_\de(z_1,z_2),$ we have  $(x_1,x_2,1+\phi\x)\in T_\de(z_1,z_2),$  so that
\[
 |T_\de(z_1,z_2)\cap S|\ge |A_\de(z_1,z_2)|\gtrsim \de^{1-\tilde\kappa^e _2}\de^{\tilde\ka^e_1} =\de^{\tilde\kappa^e _1+1-\tilde\kappa^e _2},
 \]
 provided $\tilde\kappa^e _2 \le 1/2.$ But, since $\N(\tilde \phi_{\rm red})\subset \N(\tilde \phi),$ and property $\A^+_{n-1}$ holds true also for $\tilde \phi,$ we must indeed have $1/\tilde\ka^e_2\ge 2.$

 \smallskip
 
 Thus, by Proposition \ref{prop-nec}, for any sufficiently small $\de>0,$ 
 \begin{eqnarray*}
 \|\M\|^p_{p\to p} & \gtrsim_\ve & \int_{|z_1|\le 1, |z_2|\le \de^{\tilde\kappa^e _2}} \frac{ |T_\de(z_1,z_2)\cap S|^p}{|T|}dz_1dz_2 \\
 & \gtrsim_\ve & \de^{\tilde\kappa^e _2} \frac{\de^{(\tilde\kappa^e _1+1-\tilde\kappa^e _2)p}}{\de^{2-\tilde\ka^e_2}} \\
 & = & \de^{(\tilde\kappa^e _1+1-\tilde\kappa^e _2)p-2(1-\tilde\ka^e_2)},
\end{eqnarray*}
so we must have that 
\[
p\ge 2\frac{1-\tilde\ka^e_2}{\tilde\kappa^e _1+1-\tilde\kappa^e _2}=\frac{2n^y_e}{n^y_e+1}.
\]
Note that the right-hand side is increasing in $n_e^y.$

Assume finally that $\tilde\ka^e_2=0.$ Then, for any $M\in\NN, M\gg1,$ let us change coordinates to $y_1^{(M)}:=y_1+y_2^M, y_2^{(M)}:=y_2,$ and consider $\phi^{(M)}_{\rm red}$ representing $\tilde \phi_{\rm red}$ in these new coordinates. Then, for $M$ sufficiently large, the edge of $\N(\phi^{(M)}_{\rm red})$ passing through the line $t_2=1$ will be compact, and it is easily seen 
that $n_e^{y^{(M)}}\to n_e^y$ as $M\to \infty.$ Thus, by what we have already shown for the case where $\tilde\ga_1$ is compact, also in this case we find that necessarily 
$p\ge 2n_e^y/(n_e^y+1).$ 

 Our estimates imply the necessary condition $p\ge p_e$ for $\M$ to be bounded on $L^p.$
\medskip

In combination, these three examples show that the condition $$p\ge \tilde p_c=\max\{3/2, p_e, h\}$$ is necessary in Case $\A_{n-1}^+$, which proves Proposition \ref{necA}. In combination with Theorem \ref{thm-a+}, we see as well that indeed 
$$
p_{\rm geom}(x^0)=p_c(x^0)=\max\{3/2, p_e, h\},
$$ 
if $\phi$ is of  type $\A_{n-1}^+,$  but not of type  $\A_{n-1}^e,$ and Conjecture \ref{conj1} would even imply that it is true for all functions of type 
$\A_{n-1}^+.$

\subsection{A comparison with a conjecture by Iosevich-Sawyer-Seeger}\label{isseconj}
A. Iosevich, E. Sawyer and A. Seeger conjectured in \cite{io-sa-seeger}  (Remark 1.6 (iii)) that, at least for convex surfaces, a \textit{sufficient} condition for the $L^p(\RR^d)$ -   boundedness of the maximal operator $\M$ of Subsection \ref {genconj} is the following:
\smallskip

For  every $ l = 0,\dots ,d - 1$ and for every affine $ l$-plane $E_l$ through $x^0$ the function $x\mapsto \dist(x, E_l)^{ -1}$ belongs to $L^{(d-l)/p}(S)$, that is
\begin{equation}\label{iss}
	\int_S \dist(x,E_l)^{-(d-l)/p} d\sigma (x) <\infty. 
\end{equation}

Let us compare this conjecture with the results described in the Introduction  for the case $d=3,$ so that we have to look at the condition \eqref{iss} for  $l=0, 1, 2$. 
\smallskip

If $l=0,$ then obviously we have  $\dist(x, E_0)\sim |x-x^0|,$ so that $\dist(x, E_0)^{ -1} \in L^{3/p}(S)$ for every $p>3/2$. 

\smallskip

The case $l=2$ had already been discussed in  \cite{IKM-max}, where it been shown that \eqref{iss} is equivalent to the condition 
$p>h(\phi)=h$. 

\smallskip
Finally, if  $l=1,$ let us here concentrate  on the case of  singularities of type $\A_{n-1}; $  for all other types  analogous considerations  apply as well.
So let us assume that $\phi$ is given by the normal form \eqref{AD}, and that $x^0=(0,0,1).$  Then the maximal possible restriction on $p$ exerted  by \eqref{iss}  appears when the line $E_1$ points in the direction of the $x_1$ - axis and passes through $x^0.$ Then 
$$
\dist(x, E_1)\sim |x_2|+|b(x_1,x_2)(x_2-x_1^m\omega(x_1))^2+x_1^n\beta(x_1)|\sim |x_2|+|x_1|^n,
$$ 
at least when $n$ is finite, and then \eqref{iss} holds if and only if $p>{2n}/(n+1).$  Note that here $2n/(n+1)>2n/(n+2)=h.$

The case where $b_0(x_1)$ in \eqref{AD}
is flat at the origin, where formally $n=\infty,$ works in a similar way.

\smallskip Thus, since here $h=2n/(n+2),$  the conditions given by \eqref{iss} require that
$$
p>\max\Big\{\frac 32, h, \frac{2n}{n+1}\Big\}=\max\Big\{\frac 32,  \frac{2n}{n+1}\Big\}.
$$
 
\smallskip

 Now, by Theorem \ref{thm-a-}, we know that for singularities of type $\A_{n-1}^-,$ the condition $p>\max\{3/2, h\}=\max\{3/2, 2n/(n+2)\}$ does suffice, which is a weaker condition, unless $n\le 3.$ We remark that  the corresponding surfaces cannot be convex, if the coordinates $\x$ are not adapted to $\phi.$

\smallskip

On the other hand,  for singularities of type $\A_{n-1}^+$ we always have $n_e\le n,$ so that $p_e\le {2n}/(n+1),$ and  by Theorem \ref{thm-a+}  for type  $\A_{n-1}^+\setminus \A_{n-1}^e$   the condition 
$$
p>\max\Big\{\frac 32, p_e\Big\}=\max\Big\{\frac 32,  \frac{2n_e}{n_e+1}\Big\}
$$
is sufficient. Since  there are many cases in which we  indeed  have strict inequality  $p_e<{2n}/(n+1),$  
we see that the  Iosevich-Sawyer-Seeger conjecture does not exactly  identify the critical exponent $p_c$ in general, even for convex surfaces.

\smallskip
A concrete example of this type is Example \ref{e1.4}, i.e., $\phi\x= (1+x_1^\al x_2^\be)x_2^2+x_1^n, $ with $2\le\al<n$ and $n\ge 5.$  We  had seen that here $n_e=(n(\be +1)+\al)/(\be +2)<n,$ so that $p_e<2n/(n+1),$ and one checks easily that $n_e>3,$ so that by Remark \ref{pe-h} we have  $p_e>3/2.$ Theorem \ref{thm-a+} thus shows that  the maximal operator is $L^p$- bounded in the wider range $p>p_e$ than $p>2n/(n+1).$ 
\smallskip

Another  example is the following variant of Example \ref{e2.6}, which is of type $\A_{n-1}^+\setminus \A_{n-1}^e.$
\begin{example}
Let $\phi\x:= \frac 1{1-x_1^2}x_2^2+x_1^n,$ with even $n\ge 4.$ 
\end{example}
The graph of this function is  even convex  for $|x_1|<1,$ provided $n$ is even.
Since 
$1/(1-x_1^2)=1+x_1^2+O(x_1^4),$ one finds that $n_e^x=(n+2)/{2},$ and  the coordinates $x$ are already line-adapted, so that  $n_e=(n+2)/{2}\ge 3,$  hence $p_e\ge3/2.$ Thus,  we see that $p_e=2(n+2)/(n+4)<2n/(n+1),$ and by Theorem \ref{thm-a+}, $\M$ is bounded on the wider range  $p>p_e$ than $p>2n/(n+1).$

\medskip

We also recall that we shall give a proof of the conjecture by  Iosevich-Sawyer-Seeger for  analytic singularities of type $\A_{n-1}^e$ in Section \ref{isasee}.

Thus,  our discussion in the Introduction and the main theorems of this paper show that our results give a proof of the conjecture by  Iosevich-Sawyer-Seeger in dimension $d=3$ (at least for analytic, but not necessarily convex surfaces $S$). 
\smallskip

\smallskip

\medskip

\setcounter{equation}{0}
\section{Existence of  line-adapted coordinates for functions of type $\A^+_{n-1}$ } \label{lineadac}

  Assume again that  $\phi$ is of type  $\A^+_{n-1}.$
Our first result will give some necessary conditions for the given coordinate system $x$ not to be line-adapted to $\phi;$ it  bears some analogies with Theorem 3.3 of \cite{IM-ada}, in which necessary conditions were given  for the given coordinate system $x$ not to be adapted to $\phi$ in the sense of Varchenko. Our proof also basically follows ideas from that paper, but is somewhat simpler.

Recall from the Introduction the decomposition $\phi\x=\phi(0,x_2)+\pred\x.$  Recall also that the  point $(n,0)$   is a vertex  of $\N(\pred),$   and that we denote by $\ga_1$ the (non-horizontal) edge of $\N(\pred)$  which has $(n,0)$  as right endpoint. Moreover, the weight
$\ka^e=(\ka^e_1,\ka^e_2)=(1/n,\ka^e_2)$ had been chosen so that   the edge $\ga_1$  lies on the line $L:=\{t_1,t_2): \ka^e_1t_1+\ka^e_2t_2=1\}.$

\begin{prop}\label{nonadapt}
  Let $\phi$ by of  type $\A^+_{n-1},$   and assume that the given coordinates $x$ are
not line-adapted to $\phi.$  Then all of the following conditions hold true:
 \begin{itemize}
\item[(a)]  The edge $\ga_1$ is compact, i.e., $\ka^e_2>0.$ 
\item[(b)]  $\ka^e_1/\ka^e_2=1/(\ka^e_2n)\in \NN.$
\item[(c)]  If we denote by $\phi_{e_1}:=(\pred)_{\ka^e}$ the $\ka^e$-principal part of $\pred,$ then  $\pa_1 \phi_{e_1}$ has a (unique) root of multiplicity $n-1$ on the unit circle $S^1,$ and this root does not lie on a coordinate axis. 
 \end{itemize}
\end{prop}

\begin{proof} If the coordinates $x$ are not line-adapted to $\phi,$  then there is an admissible system of coordinates $y=(y_1,y_2)$ such that $n_e^x<n_e^y.$   Assume it is given by $x_1=\vp(y_1, y_2), x_2=y_2,$ and let $\tilde\phi(y_1,y_2):=\phi(\vp(y_1, y_2),y_2)$ represent $\phi$ in the coordinates $y.$ Since $\pa_1\vp(0,0)\not=0,$  the implicit function theorem and a Taylor expansion show that we may write   $\vp(y_1, y_2)$  as  $\vp(y_1, y_2)=(y_1-\psi(y_2) )u(y_1, \psi(y_2)),$ with smooth functions $\psi$ and $ u,$ where  $u(0, 0)\neq0$ and  $\psi(0)=0.$ After scaling, we may assume for simplicity that $u(0,0)=1.$

 Obviously, if $\psi$ were a flat function at the origin, then the change of coordinates  would  not change the Newton polyhedra of $\phi$ and $\pred,$ so this case cannot arise here. 
 
 Thus, $\psi$ is of some finite type  $k\ge 1,$ so that $\psi(y_2)=y_2^k\om (y_2),$ with a smooth function $\om$ such that $\om(0)\ne 0.$ 
 
Part (a) is obvious, for if $\ka^e_2=0,$ then $n_e^x=n$ is already maximal possible.
 \smallskip
 
 To prove the other claims, let us first assume that $k>\ka^e_1/\ka^e_2.$ As in \cite{IM-ada} (Section 3 or Lemma 2.1 in that paper) we easily see that
\begin{eqnarray*}
\tilde\phi_{\rm red}(y_1,y_2)&=&(\pred)_{\ka^e}(\vp_{\ka^e}(y_1,y_2),y_2) -(\pred)_{\ka^e}(\vp_{\ka^e}(0,y_2),y_2) \\
&&\hskip4cm +\ \mbox{terms of higher $\ka^e$-degree,}
\end{eqnarray*}
 so that $(\tilde\phi_{\rm red})_{\ka^e}(y_1,y_2)=\phi_{e_1}(\vp_{\ka^e}(y_1,y_2),y_2) -\phi_{e_1}(\vp_{\ka^e}(0,y_2),y_2).$ Since $y_1$ has $\ka^e$-degree $\ka^e_1,$ and $y_2^k$ has $\ka^e$-degree $k\ka^e_2>\ka^e_1,$ we see that $\vp_{\ka^e}(y_1,y_2)=y_1,$ so that 
 $(\tilde\phi_{\rm red})_{\ka^e}=\phi_{e_1}.$ This would mean that the change of coordinates would not change the edge $\ga_1,$ contradicting our assumption $n_e^x<n_e^y.$ Thus, we must have $k\le\ka^e_1/\ka^e_2.$
 
 \smallskip 
 Assume next that $k<\ka^e_1/\ka^e_2.$ Then let us consider the weight $\tilde\ka^e:=(\ka^e_1, {\ka^e_1}/k).$ Note that $1/\tilde\ka^e_2<1/\ka^e_2,$ which means that the line $\tilde L:=\{(t_1,t_2): \tilde \ka^e_1t_1+\tilde\ka^e_2t_2=1\}$ is less steep then $L.$ In particular, 
 $\tilde\ka^e_1\alpha_1+\tilde\ka^e_2\alpha_2\ge 1$  for every $(\alpha_1,\alpha_2)\in \N(\pred).$ Moreover, similarly as before,
 $(\tilde\phi_{\rm red})_{\tilde\ka^e}(y_1,y_2)=(\pred)_{\tilde \ka^e}(\vp_{\tilde\ka^e}(y_1,y_2),y_2)-(\pred)_{\tilde\ka^e}(\vp_{\tilde\ka^e}(0,y_2),y_2).$ But, clearly 
$\vp_{\tilde\ka^e}(y_1,y_2)=y_1-cy_2^k,$ with $c:=\om(0)\ne 0,$ and $(\pred)_{\tilde \ka^e}=x_1^n\beta(0),$ with $\beta(0)\ne 0.$ Therefore
$(\tilde\phi_{\rm red})_{\tilde\ka^e}(y_1,y_2)=(y_1-cy_2^k)^n \beta(0)-(-cy_2^k)^n \beta(0),$ which shows that the edge $\tilde \ga_1$ associated to $\tilde\phi$ lies on the line $\tilde L.$ But this in return would imply that $n_e^x>n_e^y.$ 

We thus see that necessarily $k=\ka^e_1/\ka^e_2,$ which proves (b).

\smallskip
To prove (c), recall that $\ka^e_1/\ka^e_2=k\ge 1$ is a natural number, so that in particular $\ka^e_1\ge \ka^e_2.$   Since $\phi_{e_1}$ is a $\ka^e$-homogeneous polynomial of $\ka^e$-degree 1, arguing as before, we see that the change to the coordinates $y$ leads to the corresponding  function
$$
(\tilde\phi_{\rm red})_{\ka^e}(y_1,y_2)=\phi_{e_1}(y_1-cy_2^k,y_2) -\phi_{e_1}(-cy_2^k,y_2),
$$ 
which is again $\ka^e$-homogeneous of degree 1. However, only when the Taylor support of the latter  function consists just of the single point  $(n,0),$  then in the $y$ coordinates the edge $\tilde \ga_1$ corresponding to $\tilde\phi_{\rm red}$ can be steeper than the edge $\ga_1$ (this is equivalent to the condition $n_e^x<n_e^y$) - otherwise it would have the same slope as $\ga_1$. This means that necessarily
$$
\phi_{e_1}(y_1-cy_2^k,y_2) -\phi_{e_1}(-cy_2^k,y_2)=ay_1^n
$$ 
for some constant $a\ne 0,$ i.e., that
$$
\phi_{e_1}(x_1,x_2)= a(x_1+cx_2^k)^n+ \phi_{e_1}(-cx_2^k,x_2).
$$ 
This implies that $\pa_1\phi_{e_1}(x_1,x_2)= an(x_1+cx_2^k)^{n-1},$ so that $\pa_1\phi_{e_1}$ has indeed a (unique) root of multiplicity $n-1$ 
 on the unit circle, and since $c\ne 0,$ it does not lie on a coordinate axis.
\end{proof}

\begin{cor}
If  $\y$ is any admissible coordinate system for $\phi$ in which \eqref{extype} holds true, then these coordinates are already line-adapted  to $\phi.$ 
\end{cor}
\begin{proof} 
In this case, $ \phi_{e_1}=p$ is of the form \eqref{extype}, and it is easily seen that then  condition (c) in Proposition \ref{nonadapt} 
cannot be satisfied.
\end{proof}

\begin{prop}\label{existadapt}
 Let $\phi$ by of  type $\A^+_{n-1},$ and  assume that the given coordinates $x$ are
not line-adapted to $\phi.$  Then there exists  a local line-adapted coordinate system $y=(y_1,y_2)$ of  the form \eqref{shear}, i.e., 
$  y_1=x_1-\alpha (x_2), y_2=x_2,$ where we can choose for $\al(x_2)$ the unique smooth  local solution to the equation $\pa_1^{n-1}\phi(x_1, x_2)=0$ in $x_1.$ 
\end{prop}

\begin{remark}
Since $\al(0)=0,$ we may  locally write $\al(x_2)=x_2^l\rho(x_2),$  where  $l$ is a natural number  $l\ge 1,$  $\rho$ is smooth, and where we may assume that  $\rho(0)\ne 0,$ if $\al$ is of finite type $l$ at the origin. If $\al$ is flat at the origin, then we may choose $l$ arbitrarily large.  

Without loss of generality, we may even assume that $l\ge 2,$ for if  $\al$ is of finite type $l=1,$ we may first apply the linear change of coordinates 
$  y_1=x_1-\rho(0)x_2, y_2=x_2,$ which is harmless as the estimation of our maximal operators is invariant under linear coordinate changes, and then apply in  a second step a change of coordinates as before, but now with $l\ge 2.$ 
\end{remark}

{\it Proof of Proposition \ref{existadapt}.} In view of Proposition \ref{nonadapt} it will suffice to show that the conditions (a) -- (c) in that proposition already suffice to derive the existence of such an adapted coordinate system.

Since $\pa_1^{n-1}\phi(0, 0)=0$ but $\pa_1^{n}\phi(0, 0)\ne 0,$ the implicit function theorem guarantees locally the existence of a unique  smooth solution $x_1=\al(x_2)$ to the equation $\pa_1^{n-1}\phi(x_1, x_2)=0.$  
Applying a Taylor expansion of $\phi$ around $x_1=\al(x_2),$ we see that
\begin{eqnarray}\label{normform2}
\phi\x&=&B_n\x(x_1-\alpha (x_2))^n \nonumber\\
&+&B_{n-2}(x_2)(x_1-\alpha (x_2))^{n-2}+\dots+B_1(x_2)(x_1-\al(x_2))+B_0(x_2),
\end{eqnarray}
 where the functions $B_j$ are smooth and $B_n(0, 0)=\beta\ne 0, B_j(0)=0$ for $0=1,\dots, n-2.$ 
 
 Let  $\tilde \phi(y_1,y_2)$ henceforth  represent $\phi\x$  in the new coordinates $y_1:=x_1-\alpha (x_2), y_2:=x_2.$ 
 Since $\pa_1\pred=\pa_1\phi,$ \eqref{normform2} implies that
 \begin{eqnarray*}
\pa_1\pred\x&=&A_{n-1}(x_1-\alpha (x_2),x_2)(x_1-\alpha (x_2))^{n-1} \nonumber\\
&+&A_{n-2}(x_2)(x_1-\alpha (x_2))^{n-3}+\dots+A_1(x_2)(x_1-\alpha (x_2))+A_0(x_2),
\end{eqnarray*}
 where the functions $A_j$ are smooth and $A_{n-1}(0, 0)\ne 0.$ Recall also from (a) -- (c) in Proposition \ref{nonadapt} that $\ka_1/\ka_2=:k\in\NN$ is a positive integer. Moreover, either $\al$ is flat at the origin, or of some finite type $\ell,$ so that $\alpha (x_2)=x_2^\ell\om(x_2),$ with $\om(0)\ne 0.$ 
 But, since
 $$
  \pa_1^{n-1} \pred\x=D(x_1,x_2) (x_1-\alpha (x_2)),
 $$
 where $D(0,0)\ne 0,$ and since the $\ka$-homogeneous part of $\pa_1^{n-1} \pred$ has $\ka$-degree $1-(n-1) \ka_1=1/n=\ka_1,$ we see that we must have $\ell\ge k,$ since otherwise the homogeneity degree $\ka_2\ell=( {\ell}/k) \ka_1$ of the leading term of $\al$ would be strictly less than $\ka_1,$ and the same would apply to $\pa_1^{n-1} \pred\x.$

Consider  next the $\ka$-principal part of the function
$$
A_{n-1}(y_1,y_2) y_1^{n-1} 
+A_{n-2}(y_2)y_1^{n-3}+\dots+A_1(y_2)y_1+A_0(y_2),
$$
which is of the form
$$
A_{n-1}(0,0) y_1^{n-1} 
+a_{n-3} y_2^{l_{n-3}} y_1^{n-3}+\dots+a_1y_2^{l_1}y_1+a_0 y_2^{l_0},
$$ 
with real coefficients $a_j$ and  integer exponents $l_j>0, j=0,\dots, n-3.$ 
As in the proof of Proposition \ref{nonadapt}, we then see that 
\begin{eqnarray*}
\pa_1(\pred)_\ka\x&=&A_{n-1}(0,0) (x_1-\al_\ka(x_2))^{n-1} \\
&+&a_{n-3} y_2^{l_{n-3}}(x_1-\al_\ka(x_2) )^{n-3}+\dots+a_1 x_2^{l_1}(x_1-\al_\ka(x_2))+a_0 x_2^{l_0},
\end{eqnarray*}
where $\al_\ka$ denotes the $\ka$-principal part of $\al.$  By condition (c), this polynomial must have a root of order $n-1$ away from the coordinate axes. Clearly this is only possible if all coefficients $a_j, j=0,\dots n-3,$  are $0.$ But then  $\pa_1(\pred)_\ka=A_{n-1}(0,0) (x_1-\al_\ka(x_2))^{n-1},$ so that in the coordinates $y$ we have $\pa_1(\tilde \phi_{\rm red})_\ka=A_{n-1}(0,0)y_1^{n-1}$ plus terms of higher $\ka$-degree. This shows that in the new coordinates $y,$ the  edge $\tilde \ga_1$ of the Newton polyhedron of $ \tilde \phi_{\rm red}$  which passes through $(n,0)$ is steeper than $\ga_1,$ so that $n_e^y>n_e^x.$ We shall denote by $\tilde\ka=(1/n,\tilde\ka_2)$ the weight   associated to $\tilde \ga_1,$ which then satisfies  $\tilde\ka_2<\ka_2.$ 

Note that this shows in particular that the conditions (a) -- (c) imply that the original coordinates $x$ were not line-adapted.

\smallskip
Let us finally show that the coordinates $y$ are indeed line-adapted to $\phi.$ This is obvious if $\tilde\ka_2=0,$ for then $n^y_e=n$ is maximal possible. Let us therefore assume that $\tilde\ka_2>0.$ 

By \eqref{normform2}, $\tilde \phi(y)$ can be written as
$$
\tilde \phi(y)= \tilde B_n(y_1, y_2)y_1^n+B_{n-2}(y_2)y_1^{n-2}+\dots+ B_1(y_2)y_1+B_0(y_2),
$$
with smooth functions $B_j,$ where  $\tilde B_n(0, 0)\ne 0.$ As before, this implies that
$$
\pa_1\tilde \phi_{\rm red}(y)=\pa_1\tilde \phi(y)=\tilde A_{n-1}(y_1,y_2) y_1^{n-1} 
+A_{n-2}(y_2)y_1^{n-3}+\dots+A_1(y_2)y_1+A_0(y_2),
$$
where $\tilde A_{n-1}(0,0)\ne 0.$ Then, as before, we see that the $\tilde\ka$-homogeneous part of this function is of the form
\begin{eqnarray*}
\pa_1(\pred)_{\tilde\ka}(y_1,y_2)=\tilde A_{n-1}(0,0) y_1^{n-1} 
+\tilde a_{n-3} y_2^{\tilde l_{n-3}} y_1^{n-3}+\dots+\tilde a_1y_2^{\tilde l_1}y_1+\tilde a_0 y_2^{\tilde l_0},
\end{eqnarray*}
and not all coefficients $\tilde a_j,j=0,\dots n-3,$ can vanish. Since this function, as a polynomial in $y_1,$ contains no term with $y_1$ - exponent $n-2,$ it cannot have any root of order $n-1.$ Thus,  by Proposition \ref{nonadapt}, the coordinates $y$ must be  line-adapted to $\phi.$ 
\qed

Our proof shows that Proposition \ref{nonadapt} can even be strengthened:

\begin{cor}\label{nonadaiff}
 Let $\phi$ be of  type $\A^+_{n-1}.$   Then the given coordinates $x$ are
not line-adapted to $\phi$ if and only if all of the conditions (a)--(c) in Proposition \ref{nonadapt} hold true.
\end{cor}

\setcounter{equation}{0}
\section{On the Legendre transform in $x_2$ of functions of type $\A^+_{n-1}$ }\label{legendretrans}

Let $\phi$ be a function of type $\A^+_{n-1}.$ Recall that we were left with the case $n\ge 4,$ on which we shall concentrate here. 

Even though for the arguments to follow this would not be really necessary, we may assume that $\phi$ is given in the normal form \eqref{AD},  i.e.,
\begin{equation*}
\phi(x_1, x_2)=b(x_1,x_2)(x_2-x_1^m\omega(x_1))^2+x_1^n\beta(x_1).
\end{equation*}

We shall later apply the method of stationary phase to the partial Fourier transform of the corresponding measure $\mu$  defined in Section \ref{introduction}, so we shall need to look for the critical point in $x_2$  of the phase 
$$
\Phi(x_1, x_2, s_2):=\phi\x+s_2x_2,
$$
i.e.,  for the solution to the equation $\pa_2\Phi(x_1, x_2, s_2)=0$ with $x_2(0,0)=0$. Note that such a solution will locally exist as a smooth function,  in view of the implicit function theorem, since $\Phi(0, x_2, 0)=b(0,x_2)x_2^2.$ 
We then denote by 
$$
\breve\phi(x_1, s_2):=\Phi(x_1, x_2^c(x_1, s_2), s_2)
$$
the {\it Legendre transform}  of $\phi\x$ in $x_2$  (note that the Legendre  transform $\L:\phi\mapsto \breve \phi$ is often defined with the opposite sign, which ensures that it is an involution, i.e., $\L\circ \L=\rm{id}$, but our definition is better adapted to the Fourier transform in $x_2$ and the usage of the method of stationary phase).

We shall see that the function  $\breve\phi,$  as a function of the two variables $x_1$ and $s_2,$ will also be of type $\A^+_{n-1},$ so that we can define its associated effective multiplicity $n_e(\breve\phi)$ too, and we have  
\begin{thm}\label{equlind}
If the function $\phi$ is of  type $\A^+_{n-1},$ then so is its Legendre transform $\breve\phi,$ and both have the same effective multiplicity, i.e., $n_e(\breve\phi)=n_e(\phi).$ 

Moreover, if $\phi$ is of  type $\A^e_{n-1},$ then so is its Legendre transform $\breve\phi.$ 
\end{thm}

In order to prepare the proof of this theorem, recall from the Introduction the decomposition
$$
\phi\x=\phi(0,x_2)+\pred\x,
$$
and denote again by $e_1$ the (non-horizontal) edge of $\N(\pred)$  which has $(n,0)$  as right endpoint.  Recall also  the weight $\ka^e =(\ka^e _1,\ka^e _2)=(1/n,\ka^e _2)$   which is chosen so that the edge $e_1$  lies on the line $L_1:=\{t_1,t_2): \ka^e _1t_1+\ka^e _2t_2=1\}. $ 
Recall also that $0\le \ka^e _2<1/2.$ We shall distinguish two cases:
\smallskip

a) If $\ka^e _2>0,$ i.e., if the line $L_1$ is non-vertical, then the edge $e_1$ is compact, and  we denote by $p=(\pred)_\ka^e $ the $\ka^e $-principal part of $\pred$ (cf. \cite{IMmon}), i.e., the sum of all terms in the Taylor series expansion of $\pred$ corresponding to points on $e_1.$ Then $p$ is $\ka^e $-homogeneous of degree $1,$  i.e., if we define the dilations $\de_r,\, r\ge 0,$ associated to $\ka^e $ by 
$\de_r\x:=(r^{\ka^e _1}x_1, r^{\ka^e _2}x_2),$ then $p(\de_r\x)=rp\x.$ Thus we may further decompose 
\begin{equation}\label{kadecomp}
\phi\x=\phi(0,x_2)+p\x+\phi_{\rm err}(x_1, x_2),
\end{equation}
where $\phi_{\rm err}(x_1, x_2)$ {\it consists of terms of higher $\ka^e $-degree.} By this we mean that there is some $\ve>0$ such that
$\phi_{\rm err}(\de_r\cdot)=O(r^{1+\ve})$ for every $r>0,$ and the same applies to any $C^k$-norm of $ \phi_{\rm err}(\de_r\cdot)$  over $\Om.$ 
Indeed, since $\N(\phi_{\rm err})\subset \{(t_1,t_2): \ka^e _1t_1+\ka^e _2t_2>1\},$ where $\ka^e _1,\ka^e _2>0,$ this can indeed easily be seen by means of a Taylor approximation to sufficiently high degree.

\smallskip

b) If $\ka^e _2=0,$ i.e., if $L$ is the line $t_1=n,$ then a Taylor expansion in $x_1$ shows that we may write 
\begin{equation*}
\phi\x=\phi(0,x_2)+x_1f_1(x_2)+\cdots x_1^{n-1} f_{n-1}(x_2) + x_1^n b_n(x_2) +\phi_{\rm err}(x_1, x_2),
\end{equation*}
where the functions $f_1(x_2), \dots  f_{n-1}(x_2) $ are flat, $b_n(0) =\beta(0)\ne 0,$ and where  $\phi_{\rm err}(x_1, x_2)$  consists of terms of higher $\ka^e $-degree (actually with $\ve =1$).

We will first consider the case a), in which the following lemma will be crucial. Case b) will simply be reduced to case a) later.

The next result will be proved in the category of (real) analytic functions, but the proof works as well in the category of smooth functions.

\begin{lemma}\label{legendre}
Assume that $\ka^e =(\ka^e _1, \ka^e _2)$ is a weight such that  $0<\ka^e _2<1/2$ and $\ka^e _1>0. $ Let $\phi$ be an analytic function which can be written in the form
$$
\phi\x=x_2^2b_1(x_2)+\pred\x =x_2^2b_1(x_2)+ p(x_1, x_2)+\phi_{\rm err}(x_1, x_2),
$$
with an analytic function $b_1$ satisfying $b_1(0)\ne 0$ and an  analytic function $\pred\x$ such that $\pred (0,x_2)\equiv 0.$  We assume that $\pred$  decomposes into a non-zero $\ka^e $-homo\-geneous polynomial $p$ of $\ka^e $-degree $1$  such that $p(0,x_2)=0,$
and a remainder term $\phi_{\rm err}$ consisting of terms of higher $\ka^e $-degree.  Then the Legendre transform  $\breve\phi(x_1, s_2)$ of the function $\phi$ with respect to $x_2$   can be written in the form
$$
\breve\phi(x_1, s_2)=s_2^2B(s_2)+\breve\phi_1(x_1, cs_2)=s_2^2B(s_2)+p(x_1, cs_2)+\breve\phi_{\rm err}(x_1, cs_2),
$$
where $B$ is an analytic  function with $B(0)=-\frac1{4b_1(0)}\neq0$, $c:=-\frac1{2b_1(0)}$,
and $\breve\phi_{\rm err}$  is again  an analytic remainder term consisting of terms of higher $\ka^e $-degree. Moreover, also 
$\breve\phi_1(0, cs_2)\equiv 0.$
\end{lemma}

\begin{proof} 
We begin by studying the equation
\begin{equation*}
\partial_{2}(x_2^2 b_1(x_2)+s_2x_2)=2x_2b_1(x_2)+x_2^2\pa_2b_1(x_2)+s_2=0.
\end{equation*}
Obviously the equation has an analytic solution $x_2(s_2)$ by the implicit  function theorem, however, we need more information. For this reason, we first write
$x_2=ws_2$. Then it suffices to solve the equation 
\begin{equation*}
2wb_1(ws_2)+s_2w^2\pa_2b_1(ws_2)+1=0,
\end{equation*}
which is possible by the implicit  function theorem.
Let $w_0(s_2)$ be the unique analytic solution to this equation. Note that $w_0(0)=-1/{2b_1(0)}$. Then we expand the function $x_2^2b_1(x_2)+s_2x_2$ in $x_2$ around the point $s_2w_0(s_2)$:
\begin{align}\label{expand}
	x_2^2b_1(x_2)+s_2x_2=s_2^2B(s_2)+(x_2-s_2w_0(s_2))^2B_2(x_2, s_2),
\end{align}
where $B(s_2):=w_0(s_2)\{1+w_0(s_2)b_1(s_2w_0(s_2))\}$ and $B_2$ are analytic functions, with $B(0)=-1/{4b_1(0)}$ and $B_2(x_2, 0)=b_1(x_2)$.

We also expand
\begin{eqnarray*}\nonumber
\pred\x&=&\pred(x_1, s_2w_0(s_2))
			+ (x_2-s_2w_0(s_2))\pa_2\pred(x_1, s_2w_0(s_2))\\
			&&+  (x_2-s_2w_0(s_2))^2G(x_1, x_2, s_2),
\end{eqnarray*}
where $G(x_1, x_2, s_2)$ is an analytic function with $G(0, 0, 0)=\pa_2^2\pred(0,0)=0,$ since $\kappa^e _2<1/2$ (note, e.g,  that $\pa_2^2p$ is $\ka^e $-homogeneous of degree $1-2\ka^e _2>0$).

\smallskip
For the Legendre transform, we seek a solution $x_2^c=x_2^c(x_1,s_2)$ to the equation 
$$
\pa_2\phi(x_1,x_2^c)+s_2=0.
$$
 According to \eqref{expand} and the subsequent expansion of $\pred(x_1, x_2),$ this equation can  be 
re-written in the  form
$$
(x_2-s_2w_0(s_2))G_1(x_1, x_2, s_2)+\pa_2\pred(x_1, s_2w_0(s_2))=0,
$$ 
where $G_1$ is an analytic  function with $G_1(0, 0, 0)=2(G(0,0,0)+B_2(0,0))=2b_1(0)$.

By a similar scaling argument as before, we see that the solution to the last equation is of the form 
\begin{eqnarray}\label{x2crit}
x_2^c(x_1, s_2)&=&s_2w_0(s_2)+\pa_2\pred(x_1, s_2w_0(s_2))\, g(x_1, s_2)\nonumber\\
	&=&s_2w_0(0)+s_2^2\tilde w_0(s_2)+\pa_2\pred(x_1, s_2w_0(s_2))\, g(x_1, s_2),
\end{eqnarray}
	where $g(x_1, s_2)$ is an analytic function  (to be determined)  with $g(0, 0)=-1/{2b_1(0, 0)}$.

Then we look at 
\begin{eqnarray}\label{legend1}
\breve\phi(x_1,s_2)&=&\phi(x_1,x_2^c(x_1,s_2))+s_2 x_2^c(x_1,s_2) \nonumber\\
&=&\Big(x_2^c(x_1,s_2)^2b_1(x_2^c(x_1,s_2))+s_2x_2^c(x_1,s_2)\Big)
+\pred(x_1,x_2^c(x_1,s_2)).
\end{eqnarray}

 It is easy to see by \eqref{x2crit}  and a Taylor expansion  in $x_2$ around the point $ s_2w_0(0)$ of 
 $$
 \phi_{\rm red}(x_1, s_2w_0(0)+s_2^2\tilde w_0(s_2)+\pa_2\pred(x_1, s_2w_0(s_2)
 )\, g(x_1, s_2))
 $$
 that the term $\pred(x_1,x_2^c(x_1,s_2))=p(x_1,x_2^c(x_1,s_2))+\phi_{\rm err}(x_1,x_2^c(x_1,s_2))$ is equal to the $\ka^e $-homogeneous polynomial $p(x_1,s_2w_0(0))$ plus a  remainder term consisting of terms of higher $\ka^e $-degree. Here we use  indeed that $0<\ka^e _2<1/2.$ 
 \smallskip

  To deal with what happens if we plug $x_2^c(x_1, s_2)$  into $x_2^2b_1(x_2)+s_2x_2,$ we use \eqref{expand} and \eqref{x2crit}:
\begin{eqnarray}\label{legend2}
&&(x_2^2b_1(x_2)+s_2x_2)\vert_{x_2=x_2^c\x}\nonumber \\
&=&s_2^2B(s_2)+ \Big (\pa_2\phi_{\rm red}(x_1, s_2w_0(s_2)) \,g(x_1, s_2)\Big)^2 B_2(x_2^c\x, s_2).
\end{eqnarray}

	The first term $s_2^2B(s_2)$ is just what we expect. For the second term, observe
that $\pa_2\phi_{\rm red}(x_1, s_2w_0(s_2))g(x_1, s_2)=\{\pa_2p(x_1, s_2w_0(s_2))+\pa_2\phi_{\rm err}(x_1, s_2w_0(s_2))\}g(x_1, s_2)$ is a sum of monomials of $\ka^e $-degree $1-\ka^e_2$  and error terms of even higher $\ka^e $-degree. But then 
 $(x_2^c\x-s_2w_0(s_2))^2$ consists of terms  of $\ka^e $-degree at least $2-2\kappa^e_2>1$ since $\kappa^e_2<1/2$. In other words, the second term is a remainder term too.
 
 Our discussion also shows that $\breve\phi_1(0, cs_2)\equiv 0.$
\end{proof}

\noindent {\it Proof of Theorem \ref{equlind}.} 
We are now in a position to prove Theorem \ref{equlind}.  Again, we shall work in the category of analytic functions, but the proof works as well in the category of smooth functions.
\smallskip

Assume that the function  $\phi$ is of  type $\A^+_{n-1}.$ We begin by showing that its Legendre transform $\breve\phi$ is  of  type $\A^+_{n-1}$ too.

Suppose first that we are in  case a), i.e., that $\ka^e _2>0.$ This applies in particular if the coordinates $\x$ are  not line-adapted (recall  from Proposition \ref{nonadapt} that  the coordinates $\x$ are line-adapted if $\ka^e _2=0$).
 Here, our claims follow directly from  Lemma \ref{legendre}.
 
 Next consider case b) where  $\ka^e _2=0.$  Then clearly $n_e(\phi)=n,$ and $\phi$ is of  type $\A^e_{n-1}.$ 
  Given any sufficiently small $\ka'_2>0,$ we then decompose $\phi$ in the form \eqref{kadecomp}, but  now with respect to the weight 
$\ka':=(1/n,\ka'_2 ),$ i.e., we choose for the polynomial $p$ the $\ka'$-principal part $p_{\ka'}$ of $\pred,$ which is $\ka'$-homogeneous of degree $1$. Then clearly we have $p_{\ka'} \x=x_1^n b_n(0)=x_1^n \beta(0).$ We may thus apply Lemma \ref{legendre} with $\ka'$ in place of $\ka^e$ and find that $\breve\phi(x_1, s_2)=s_2^2B(s_2)+p_{\ka'}(x_1, cs_2)+\breve\phi_{\rm err}(x_1, s_2),$ which implies that $j/n+\ka_2' k\ge 1$ for every point $(j,k)$ in the Taylor support of $\breve\phi_{\rm red}.$ Letting $\ka'_2$ tend to $0,$ we see that $j/n\ge 1,$ so that $\N(\breve\phi_{\rm red})$ is contained in the region where $t_1\ge n.$ This implies that again  $\breve\phi$ is of type $\A^+_{n-1},$  and that $n_e(\breve\phi)=n=n_e(\phi),$ so that $\breve \phi$ is of  type $\A^e_{n-1}$ too.
\smallskip

In the second step of the proof, we show that $n_e(\breve\phi)=n_e(\phi).$ In the case where the coordinates $\x$ are line-adapted  to $\phi,$ this is again immediate from  Lemma \ref{legendre} when $\ka^e _2>0,$ and for the case $\ka^e _2=0$ we have already verified this in  the first step of the proof. 
 
  There remains the case  where the  coordinates $\x$ are not line-adapted to $\phi.$ 
As shown in  the proof of Proposition \ref{existadapt} then  there exists an analytic function $\alpha (x_2)$ with $\alpha(0)=0$ 
such that $\pred$ can be written as 
\begin{eqnarray}\label{normform}
\pred\x&=&B_n\x(x_1-\alpha (x_2))^n \nonumber\\
&+&B_{n-2}(x_2)(x_1-\alpha (x_2))^{n-2}+\dots+B_1(x_2)(x_1-\alpha (x_2)),
\end{eqnarray}
 where the functions $B_j$ are smooth and $B_n(0, 0)=\beta(0)\neq0, B_j(0)=0$ for $j=1,\dots, n-2,$ and so that the coordinates $y=(y_1,y_2):=(x_1-\alpha (x_2), x_2)$ are line-adapted.  In these coordinates, $\pred$ is given by 
\begin{equation}\label{legend3}
\tilde\phi_{\rm red}(y_1,y_2)=B_n(y_1+\al(y_2),y_2)y_1^n +B_{n-2}(y_2)y_1^{n-2}+\dots+B_1(y_2)y_1.
\end{equation}
\smallskip
 We denote by $\tilde \ga_1$ the edge of $\N(\tilde\phi_{\rm red})$ which has $(n,0)$ as right endpoint, and choose the weight $\tilde \ka^e =(\tilde \ka^e _1,\tilde\ka^e _2)=(1/n,\tilde\ka^e _2)$ so that $\tilde \ga_1$ is contained in the line $\tilde \ka^e _1t_1+\tilde \ka^e _2t_2=1.$ 
 Note that if we just work with smooth functions, then  $\tilde \ka^e _2=0$ if and only if  all  of the functions $B_j$  for $j=1, \dots, n-2,$ are flat, but in the category of analytic functions they even do vanish. In any case we see that  if $\tilde \ka^e _2=0,$ then   $n_e^y=n,$ so that $n_e(\phi)=n.$ 
\medskip

Let us first consider the case where $\tilde\ka^e _2>0.$ As before, we decompose 
\begin{equation}\label{tildedecomp}
\tilde\phi_{\rm red}(y_1,y_2)=\tilde p(y_1,y_2) +\tilde\phi_{\rm err}(y_1,y_2),
\end{equation}
where $\tilde p$ is a $\tilde\ka^e $-homogeneous polynomial of degree $1$, and $\tilde\phi_{\rm err}$ consists of terms of higher $\tilde\ka^e $-degree. Next, following the proof of Lemma \ref{legendre}, we apply   \eqref{legend1} to  decompose $\breve\phi(x_1,s_2)$ as a sum of two terms. We first look at the term 
$\pred(x_1,x_2^c(x_1,s_2)),$ which according to \eqref{x2crit} we can write as 
\begin{equation}\label{legend4}
\pred(x_1,x_2^c(x_1,s_2))=\pred\big(x_1,s_2w_0(s_2)+\pa_2\pred(x_1, s_2w_0(s_2))\, g(x_1, s_2)\big).
\end{equation}
Let us assume without loss of generality that $g(0,0)=1.$ Then it will suffice to understand $\pred\big(x_1,s_2w_0(s_2)+\pa_2\pred(x_1, s_2w_0(s_2))\big),$ since the passage from $g=1$ to our more general function $g$ will only add terms of higher $\tilde\ka^e $ - degree. 

Let us put $\tilde\al(s_2):= \al(s_2w_0(s_2)),$ and introduce the new coordinate $z_1:=x_1-\tilde\al(s_2).$  We will show that in the coordinates $(z_1,s_2)$, $\breve\phi$ is line-adapted.

We apply a Taylor expansion in the second variable  of the right-hand side in \eqref{legend4} (with $g=1$) around the point $s_2w_0(s_2).$  Then, according to \eqref{normform}, the leading term is given by
$$
\pred\big(x_1,s_2w_0(s_2))=\tilde B_n(z_1,s_2)z_1^n 
+\tilde B_{n-2}(s_2) z_1^{n-2}+\dots+\tilde B_1(s_2)z_1,
$$
where $\tilde B_n(z_1,s_2):= B_n(z_1+\tilde\al(s_2), s_2w_0(s_2))$ and $\tilde B_j(s_2):= B_j( s_2w_0(s_2))$ for $j=1,\dots, n-2.$
A comparison with \eqref{legend3} then shows  that $\pred\big(x_1,s_2w_0(s_2))=\tilde p(z_1,s_2w_0(0))$ plus terms of $\tilde\ka^e $ - degree $>1.$ 

The second term in this Taylor expansion will be given by
$$
\big(\pa_2\pred(x_1, s_2w_0(s_2)\big)^2.
$$
Looking at \eqref{normform}, $\pa_2\pred \x$ is a sum of terms $B_{j}(x_2)(x_1-\alpha (x_2))^{j}, j=1,\dots,n-2,$ and $B_n\x(x_1-\alpha (x_2))^n.$
And, for $j=1,\dots,n-2,$ 
$
\pa_2\big(B_{j}(x_2)(x_1-\alpha (x_2))^{j}\big)=\pa_2B_j(x_2)\, (x_1-\alpha (x_2))^{j}-j\al'(x_2)B_j(x_2)(x_1-\alpha (x_2))^{j-1},
$
so that 
$$
\pa_2\big(B_{j}(x_2)(x_1-\alpha (x_2))^{j}\big)\vert_{x_2=s_2w_0(s_2)}=\pa_2B_j(s_2w_0(s_2))\, z_1^{j}-j\al'(s_2w_0(s_2))B_j(s_2w_0(s_2)) z_1^{j-1};
$$
 a similar reasoning holds for $j=n.$ A comparison with \eqref{legend3} then shows  that all of these summands consist, in the coordinates $(z_1,s_2),$ of terms of $\tilde \ka^e $-degree $\ge 1-\tilde\ka^e _2,$ or  $\ge 1-\tilde \ka^e _1.$ Since $\tilde\ka^e _2<\ka^e _2<1/2$ and also $\tilde\ka^e _1=1/n<1/2,$ this shows that $\big(\pa_2\pred(x_1, s_2w_0(s_2)\big)^2$ consists of error terms of  $\tilde \ka^e $-degree $> 1.$ 
A similar reasoning applies also to the higher order terms in the afore mentioned Taylor expansion, so that we see that 
$\pred(x_1,x_2^c(x_1,s_2))=\tilde p(z_1,s_2w_0(0))$ plus terms of $\tilde\ka^e $ - degree $>1.$ 

\smallskip
For the first term in \eqref{legend1}, we use again \eqref{legend2}. The first term $s_2^2B(s_2)$ is exactly what we expect, and the second term consists of error terms  of  $\tilde \ka^e $-degree $> 1,$ as we have just seen.

Combining everything, we see that in the coordinates $(z_1,s_2),$ in which $\breve\phi(x_1,s_2)$ is represented  by the function $\breve\phi^{\rm la}(z_1,s_2),$ we have 

\begin{equation}\label{legendrephila}
\breve\phi^{\rm la}(z_1,s_2)=s_2^2B(s_2)+\tilde p(z_1,s_2w_0(0))\ +\text{ terms of $\tilde\ka^e $ - degree $>1.$ }
\end{equation}


This shows that the coordinates $(z_1,s_2)$ are line-adapted to $\breve\phi,$ and that $n_e(\breve\phi)=n_e(\phi).$ It also shows  that if  $\phi$ is of  type $\A^e_{n-1},$ then so is  $\breve\phi.$

\medskip

Assume finally that $\tilde\ka^e _2=0,$ so that $n_e(\phi)=n.$ We thus have to show that also $n_e(\breve\phi)=n.$ To this end, we can apply the same kind of trick that we used before to handle case b). For any sufficiently small $\tilde\ka'_2>0,$ we decompose $\breve\phi$ is an \eqref{tildedecomp}, but now with respect to the weight $\tilde\ka':=(1/n,\tilde\ka'_2).$ Then clearly $\tilde p(y_1,y_2)=B_n(0,0) y_1^n$ (note that the functions $B_j(y_2), j=1,\dots, n-2,$ in \eqref{legend3} are now all flat!). Thus, arguing as before, we see that $j/n+\tilde\ka_2' k\ge 1$ for every point $(j,k)$ in the Taylor support of $\breve\phi^{\rm la}_{\rm red}.$ Letting $\tilde\ka'_2$ tend to $0,$ we see that $j/n\ge 1,$ so that $\N(\breve\phi^{\rm la}_{\rm red})$ is contained in the region where $t_1\ge n.$ 
\qed

\setcounter{equation}{0}
\section{Preparatory steps in the proofs of Theorems \ref{thm-a-} and \ref{thm-a+}}\label{prepsteps}

\subsection{Auxiliary results}
Let us first recall a classical version of van der Corput's lemma  (cf. Corollary to Proposition 2  in \cite[VIII]{stein-book}).
\begin{lemma}\label{Corput}
Suppose $\phi$ is smooth and real-valued in $(u,v)$ such that $|\phi^{(m)}(x)|\ge 1$ for all $x\in (u,v),$ where $m\ge 2.$ Then, for every smooth complex amplitude $a$ on $[u,v]$ we have that
$$
\Big|\int_u^v e^{i\la\phi(x)}a(x)\, dx\Big|\le c_m\la^{-\frac 1m}\Big[ |a(v)|+\int_u^v |a'(x)|\, dx\Big].
$$
\end{lemma}

We shall also need the  following lemma, which is closely related to \cite[Lemma 5.6]{IMmon}.

\begin{lemma}\label{pinvers} Consider the oscillatory integral
$$
I(u,\zeta):=\iint e^{iN\big(\phi(x,\zeta)+\eta(u-x)\big)} a(x,\eta,u,\zeta)\, d\eta dx,
$$
where $a$ is a smooth, compactly supported amplitude function on $\RR\times \RR\times \RR\times \RR^p,$ $\phi$ is a real-valued smooth  amplitude function defined on a neighborhood of the projection of the support of $a$ to the $(x,\zeta)$-variables, and $N\gg1.$  We also assume that we have uniform bounds on the  partial derivatives of $a$ and $\phi$ of any order. 
Then 

$$
I(u,\zeta)=N^{-1} e^{iN\phi(u,\zeta)} b(u, \zeta, N^{-1}), 
$$
where $b$ is a smooth amplitude with compact support  whose partial derivatives can be controlled by means of the uniform bounds that we have postulated on $a$ and $\phi.$
\begin{remarks}\label{critvalue}
i) Note that  the phase function $\phi(u,\zeta)$ is just the value of the function $(x,\eta)\mapsto \phi(x,\zeta)+\eta(u-x)$ at its unique critical point $(x^c, \eta^c):=(u,\pa_x\phi(u,\zeta)).$ Thus, at least formally, the result is a consequence of the method of stationary phase, but our subsequent direct argument is simpler and  leads  faster to the claimed  estimates. 

ii) It will become important  later on to recall that the critical value, i.e., the value of a function at its critical point, is invariant under coordinate changes in the argument of the function. 
\end{remarks}

\end{lemma}

\begin{proof} 
Denoting by $\F_\eta$ the partial Fourier transform in $\eta,$ we have
\begin{eqnarray*}
I(u,\zeta)&=&\int e^{i N\phi(x,\zeta)} \F_\eta a(x,N(x-u),u,\zeta)\, dx\\
&=&N^{-1} \int e^{i N\phi(u+tN^{-1}, \zeta)} \F_\eta a(u+tN^{-1},t,u,\zeta)\, dt\\
&=&N^{-1} e^{iN\phi(u,\zeta)} b(u, \zeta, N^{-1}),
\end{eqnarray*}
where 
$$
 b(u, \zeta, N^{-1}):= \int e^{i \psi(u,tN^{-1}, \zeta)} \F_\eta a(u+tN^{-1},t,u,\zeta)\, dt,
$$
where by Taylor
$$
\psi(u,tN^{-1}, \zeta):=t\int\limits_0^1\pa_x\phi(u+st N^{-1},\zeta)\, ds.
$$
Our claim now follows easily  since $\F_\eta a(u+tN^{-1},t,u,\zeta)$ is rapidly decaying in $t.$

\end{proof}

Another important tool for us will be the following result (cf. \cite[Proposition 4.2]{bdim19}):
If $A$ is a bounded Lebesgue measurable subset of  $\RR^n,$  then  we denote by $\M_A$ the corresponding maximal operator 
$$
\M_A f(x):=\sup_{t>0} \int_A |f(x+ty)|\, dy, \qquad f\in L^1_{\rm loc}(\RR^n).
$$
In particular, if $A=B_1(0)$ is the Euclidean unit ball, then $\M_A$ is the Hardy-Littlewood maximal operator $\M_{\rm HL}.$

Denote further by $\pi:\RR^n\setminus\{0\}\to S^{n-1}$ the {\it spherical projection}  onto the unit sphere $S^{n-1},$  given by $\pi (x):= x/|x|,$     and by $|\pi(A)|$ the $n-1$-dimensional volume of this set with respect to the surface measure  on the sphere.

\begin{lemma}\label{maxproj}
Assume that  $A $ is an open subset of $\RR^n$   contained in an annulus where $|x|\sim 1.$
Then, for $1<p\le \infty,$   we have that 
\begin{equation*}
\|\M_A f\|_{L^p\to L^p}\le C_p |\pi(A)|.
\end{equation*}
\end{lemma}

\color{black}
\bigskip

\subsection{Dyadic frequency decomposition}

 To prove our main theorems, we shall follow in the first steps the discussion  of $\A_2$ - type singularities in Subsection 7.1 of \cite{bdim19}.

Recall that the maximal operator $\M_S$ is associated to  the  averaging
operators of  convolution with dilates  of a measure $\mu,$
whose Fourier transform at $\xi=(\xi_1,\xi_2,\xi_3)$  is given  by

\begin{equation}\label{fourtrans}
\hat\mu(\xi) :=\int e^{-i(\xi_1x_1+\xi_2x_2+\xi_3(1+\phi\x))}
\eta(x_1,x_2)\,dx_1 dx_2.
\end{equation}

 In order to estimate this maximal operator, we shall perform a non-homogeneous 
dyadic frequency decomposition with respect to the last variable $\xi_3.$ Since the low frequency part is easily controlled by the Hardy-Littlewood
maximal operator, we shall subsequently concentrate on the regions where   $\pm \xi_3\sim \la,$ with $\la\gg 1$ a sufficiently large dyadic number. 

To be more specific, let us choose suitable smooth cut-off functions $\chi_0$ and $\chi_1$ 
of sufficiently small compact supports, where $\chi_1$ vanishes near the origin and is identically one  near $1,$ whereas $\chi_0$ is identically $1$ on a small neighborhood of the origin.  Then,  for a fixed dyadic $\la\gg1,$  we shall concentrate on the contribution by  $\chi_1\left({\xi_3}/{\lambda} \right)\hat\mu(\xi).$  We shall later see that the estimates that we shall obtain for the corresponding maximal operators will sum over all such dyadic numbers $\la.$ 

As a further reduction, let  us decompose  these  contributions  into

\begin{equation}\nonumber
    \widehat{\mu^\la}(\xi) :=\chi_0\left(\frac{\xi_1}{\lambda},\,\frac{\xi_2}{\lambda} \right) \chi_1\left(\frac{\xi_3}{\lambda} \right)\hat\mu(\xi),
\end{equation}
and
$$
\left(1-\chi_0\left(\frac{\xi_1}{\lambda},\,\frac{\xi_2}{\lambda}
\right)\right)\chi_1\left(\frac{\xi_3}{\lambda} \right)\hat\mu(\xi).
$$

Note that if we choose the support of  the amplitude $\eta$ in \eqref{fourtrans}  sufficiently small, then
integrations by parts easily show that  the  latter term is of
order $O(\la^{-N})$ for every $N\in\NN$  as $\la\to+\infty, $ so
that the corresponding contribution to the maximal operators is
under control.

It therefore suffices to control the contribution by $\mu^\la.$
As in   \cite{IMmon} and \cite{bdim19}, we write
\begin{equation*}
\xi_3=\la s_3,\quad \xi_1=\la s_3s_1,\quad \xi_2=\la s_3s_2,
\end{equation*}
and put $s':=(s_1,s_2), s:=(s',s_3).$ Then, choosing the supports of $\chi_0$ and $\chi_1$ sufficiently small, we have
$$
|s_3|\sim 1\quad \mbox{and}\quad  |s'|\ll1
$$
on the support of $\widehat{\mu^\la}$. Hence
\begin{equation}\label{mulathat}
\widehat{\mu^\la}(\xi)
= e^{-i\la s_3} \chi_0(s_3s')\chi_1(s_3) \int_{\bR^2} e^{-i\la s_3(s_1x_1+s_2x_2+\phi(x))} \eta(x)\,dx.
\end{equation}
\medskip

\setcounter{equation}{0}

\section{ First steps in the proof of Theorem \ref{thm-a-}}\label{proof-}

In this section, we assume that $\phi$ is of type $\A_{n-1}^-,$ with $n\ge 4.$ Recall that then
\begin{equation*}
2m\le n \quad\text{and }\quad h=\frac {2n}{n+2}.
\end{equation*}

\subsection{ Stationary phase in $x_2$ and identities for $\mu^\la$}

We first apply the method of stationary phase to the $x_2$ - integration in \eqref{mulathat}.  
Recall that for  $\phi$ of type $\A_{n-1}^-$  we have $2m\le n,$   so that both $m$ and $n$ are finite.  We recall that the coordinates $(x_1,y_2)$ with $y_2:=x_2-\psi(x_1)$ are adapted to $\phi.$ 
In these coordinates, $\phi$ is of the form $\pad(x_1,y_2)=b^a(x_1,y_2) y_2^2+b_0(x_1),$ where we had written $b_0(x_1)=x_1^n\be(x_1)$ and $\psi(x_1)=x_1^m\om(x_1),$ with $\be(0)\ne 0$ and $\om(0)\ne0.$   Changing coordinates in \eqref{mulathat}, we see that
$$
\widehat{\mu^\la}(\xi) = e^{-i\la s_3} \chi_0(s_3s')\chi_1(s_3) \int_{\bR^2} e^{-i\la s_3(s_1x_1+ s_2\psi(x_1) + s_2y_2+\pad(x_1,y_2)} \eta(x_1,y_2+\psi(x_1))\,dy_2 dx_1.
$$
Applying the method of stationary phase to the $y_2$-integration, this leads to 
\begin{eqnarray*}
\widehat{\mu^\la}(\xi)
&= &e^{-is_3\la} \chi_0(s_3s')\chi_1(s_3) \big[ \la^{-1/2} \int_{\bR} e^{-i\la s_3(s_1x_1+ s_2\psi(x_1) +\pad(x_1, y_2^c(x_1,s_2)))}\tilde \eta(x_1,  s_2)\,dx_1\\
&&\hskip 8cm+r(\la,s) \big],
\end{eqnarray*}
with a slightly modified cut-off function $\chi_1,$  
where $\tilde \eta$ is another smooth bump function supported in a
sufficiently small neighborhood of  the origin,  $r(\la,s)$ is a
remainder term of order
$$
r(\la,s)=O(\la^{-\frac32})\quad \mbox{as} \quad  \la\to+\infty,
$$
and $y_2^c(x_1, s_2)$ denotes  the unique (non-degenerate)
critical point  of the  phase $\pad$ with respect to $y_2.$ Then  $\pad(x_1, y_2^c(x_1,s_2))=\breve\pad(x_1,s_2)$ is the Legendre transform of $\pad.$ 

Applying a similar scaling argument as in the proof of Lemma \ref{legendre},  we see that locally near the origin there exists a unique smooth function $w(x_1,s_2)$ with $w(0,0)=-1/b^a(0,0)$ such that $y_2^c(x_1, s_2)=s_2 w(x_1,s_2).$  This implies that $\breve\pad(x_1, s_2)=s_2^2\breve b^a(x_1,s_2)+x_1^n\be(x_1),$ where 
$\breve b^a$ is smooth, with $\breve b^a(0,0)=-1/4b^a(0,0)\ne 0.$ Taylor expansion around $x_1=0$ then shows that $\breve\pad(x_1, s_2)$ can be written as
\begin{equation}\label{brevepad1}
\breve\pad(x_1, s_2)=s_2^2B(s_2)+s_2^2x_1q(x_1, s_2)+x_1^n\be(x_1),
\end{equation}
with a smooth functions $B$ and $q,$ where $B(0)\neq 0$.

The contribution of the error term $r(\la,s)$ to $\widehat{\mu^\la},$ which we denote by $\mu^\la_{\rm error},$ 
and the corresponding maximal operator $\M^\la_{\rm error}$ are easily estimated, and we
shall henceforth ignore it.  

Indeed, let us briefly sketch the argument:   since $\widehat{\mu^\la_{\rm error}}=O(\la^{-\frac32}),$
by  applying the usual  arguments to control the  maximal operator on $L^2,$ which are  based on  almost orthogonality and a variant of Sobolev's  embedding theorem (as for instance in the proof of  \cite[Proposition 4.1]{bdim19}),   we obtain that 
$$
\|\M^\la_{\rm error} \|_{L^2\mapsto L^2} \lesssim   \la^{\frac 12}\la^{-\frac32}
 = C\la^{-1}.
$$
On the other hand, we find that $|\mu^\la_{\rm error}(y+\Gamma)|\lesssim \la^{3/2}$ (compare \eqref{mula}, where the exponent $5/2=3-1/2$ has to be replaced by $3-3/2=3/2$), and by Lemma \ref{maxproj} this implies that for every $\ve>0$,
$$
\|\M^\la_{\rm error} \|_{L^{1+\ve}\mapsto L^{1+\ve}} \lesssim  \la^{\frac 52} \la^{-1}.
$$
Interpolation between these estimates leads to 
$$
\|\M^\la_{\rm error} \|_{L^p\mapsto L^p} \lesssim   \la^{\frac 52 (\frac 2p -1)+\ve }\la^{-1}=\la^{\frac 5p-\frac 72+\ve}
$$
for every $\ve >0.$ However, we are assuming that $p>3/2,$ so that $ 5/p-7/2<-1/6,$ which shows that the estimates can be summed over all dyadic $\la$s.

\smallskip

Putting 
\begin{equation}\label{legendamins}
\breve\phi_1(x_1, s_2):=x_1^n\beta(x_1)+s_2x_1^m\omega(x_1)+s_2^2x_1 q(x_1, s_2),
\end{equation}
we thus see that  finally may assume that 
\begin{eqnarray}\nonumber\label{mula2}
\widehat{\mu^\la}(\xi)
&= &\la^{-\frac 12} \int_{\bR}e^{-i\la s_3\big(s_1x_1+ s_2x_1^n\om (x_1) +\breve \pad(x_1,s_2)+1\big)}\, \eta(x_1,s_2) dx_1 \, \chi_0(s_3s')\chi_1(s_3) \\
&= &\la^{-\frac 12} e^{-i\la s_3\big(s_2^2B(s_2)+1\big)}  J(\la, s) \, \chi_0(s_3s')\chi_1(s_3) ,
\end{eqnarray}
\medskip

where we have put 
\begin{equation}\label{Jla}
J(\la, s):=\int e^{-i\la s_3(\breve\phi_1(x_1, s_2)+s_1x_1)}  \eta(x_1,  s_2)\,dx_1.
\end{equation}

In order to defray the notation, we have simply dropped the tilde from $\eta.$
Note also that, by Fourier inversion and a change from the coordinates $\xi$ to $s,$ we have that (with slightly modified functions $\chi_0, \chi_1$)
\begin{eqnarray}\nonumber
\mu^\la(y+\Ga)&=&\la^{\frac52}
\int  J(\la, s) e^{-i\la s_3\big( s_2^2B(s_2)-s_1y_1-s_2 y_2-y_3\big)}\chi_0(s_3s')\chi_1(s_3)\,ds\\ \label{mula}
&=&\la^{\frac52}
\int \Big(\int F(\la,s_2,s_3,y_1) e^{-i\la s_3(s_2^2B(s_2)-s_2 y_2-y_3)} \chi_0(s_2) ds_2\Big) \chi_1(s_3) ds_3,
\end{eqnarray}
where we have put $\Gamma:=(0,0,1),$  and 
\begin{eqnarray}\nonumber
 F(\la,s_2,s_3,y_1) &:=&\int  J(\la, s) e^{i\la s_3s_1y_1} \chi_0(s_3s')\chi_0(s_1)\,ds_1\\  \label{defiF}
  &=&\iint e^{-i\la s_3(\breve\phi_1(x_1, s_2)+s_1x_1-s_1y_1)}  \eta(x_1,  s_2)  \chi_0(s_3s')\chi_0(s_1) \,dx_1 ds_1.
 \end{eqnarray}
Quite important for us will thus be the function
\begin{equation}\label{Phi0}
\Phi_0(x_1,s'):=\breve\phi_1(x_1, s_2)+s_1x_1, 
\end{equation}
which allows to re-write
\begin{eqnarray*}
J(\la, s)&=&\int e^{-i\la s_3\Phi_0(x_1,s')}  \eta(x_1,  s_2)\,dx_1,\\
 F(\la,s_2,s_3,y_1) &=& \iint e^{-i\la s_3(\Phi_0(x_1,s')-s_1y_1)}  \eta(x_1,  s_2)  \chi_0(s_3s')\chi_0(s_1) \,dx_1 ds_1.
\end{eqnarray*}
We should like to mention that the variable $s_3,$ which is of size $s_3\sim 1,$ will be irrelevant for our analysis, so that we hall basically consider it to be frozen. 

\medskip
In view our definition of the oscillatory integral $J(\la,s)$ and van der Corput's lemma, what will be quite important for our analysis will indeed be the second derivative of $\Phi_0$ with respect to the variable $x_1,$ i.e., the function 
\begin{equation}\label{Phi}
\Phi(x_1,s_2):=\pa_{x_1}^2\Phi_0(x_1,s') =\pa_{x_1}^2\breve\phi_1(x_1, s_2),
\end{equation}
which depends only on the integration  variable $x_1$ in $J(\la,s),$ and the ``parameter'' $s_2.$ In order to understand the oscillatory integrals $J(\la,s)$ and  $F(\la,s_2,s_3,y_1),$ it turns out that we need to understand the singularities of the function 
$\Phi(x_1,s_2)$ as a function of $x_1$ and $s_2,$ and to this end we shall follow ideas from \cite {phong-stein}, \cite{PSS}, \cite{IKM-max}, \cite{IM-ada}, \cite{IMmon} and related papers,  by looking at Newton polyhedra associated to $\Phi$ and the related Puiseux series expansions of the roots of $\Phi.$

\begin{remark}
 In its definition, as well as in the proofs of our main Theorems \ref{thm-a-} and  \ref{thm-a+}, we shall make use of certain local coordinate changes of the form \eqref{shear}, in which the roles of the variables $x_1$ and $x_2$ are interchanged, compared to the changes of coordinates that had been  used to pass to adapted coordinates in the monograph \cite{IMmon} and related  papers, such as \cite{IM-ada}. 
 
 Thus, when comparing with Newton polyhedra in these articles, one should imaging the coordinate axis to be interchanged compared to the present situation.
 \end{remark}
 This should be taken into account when comparing the discussions in the next subsection with the discussions in these papers.

\medskip
\subsection{Description of the Newton polyhedron $\N(\Phi)$ of $\Phi$  in terms of the roots.}\label{roots}
\medskip
 In the course of our arguments, by applying an algorithm for a resolution of singularities, we shall have to change coordinates from $x_1$ to $z_\iota$ in the step $\iota.$ Therefore let us begin by writing  $z_1:=x_1$ in Step 1 of this algorithm, where we work in our original coordinate $x_1,$ so that
$\Phi(z_1,s_2)=\Phi(x_1,s_2)=\pa_{z_1}^2\breve\phi_1(z_1, s_2).$

\medskip

Since $\Phi$ is  real-analytic and real valued,  the Weierstra\ss{}  preparation theorem allows to write
$$
\Phi(z_1,s_2)= U_0(z_1,s_2)z_1^{\nu_1} s_2^{\nu_2} \Psi(z_1,s_2)
$$
near the origin, where $\Psi(x_1,s_2)$ is a pseudo-polynomial of the form
$$
\Psi(z_1,s_2)=z_1^{N}+a_1(s_2)z_1^{N-1}+\dots+ a_{N}(s_2),\qquad (\nu_1+N=n-2),
$$
and $U_0,\, a_1,\dots , a_{N}$ are real-analytic functions satisfying 
 $U_0(0,0)\neq0, \, a_j(0)=0$. By \eqref{legendamins}, we here have $\nu_2=0.$ 
 
 Observe that the Newton polyhedron of $\Phi$ is the same as that of $z_1^{\nu_1}  \Psi(z_1,s_2).$ We shall also assume without loss of generality that $a_{N}$ is a non-trivial function, so that the roots $r(s_2)$ of the  equation $\Psi(z_1,s_2)=0,$ considered as a polynomial in $z_1,$ are all non-trivial. 
 
 It is well-known  that these roots can be expressed in a small neighborhood of $0$ as   Puiseux series 
$$
r(s_2)=c_{l_1}^{\al_1}s_2^{a_{l_1}}+c_{l_1l_2}^{\al_1\al_2}s_2^{a_{l_1l_2}^{\al_1}}+\cdots+
c_{l_1\cdots l_p}^{\al_1\cdots \al_p}s_2^{a_{l_1\cdots
l_p}^{\al_1\cdots \al_{p-1}}}+\cdots,
$$
where
$$
c_{l_1\cdots l_p}^{\al_1\cdots \al_{p-1}\be}\neq c_{l_1\cdots
l_p}^{\al_1\cdots \al_{p-1}\ga} \quad \mbox{for}\quad \be\neq \ga,
$$

$$
a_{l_1\cdots l_p}^{\al_1\cdots \al_{p-1}}>a_{l_1\cdots
l_{p-1}}^{\al_1\cdots \al_{p-2}},
$$
with strictly positive exponents $a_{l_1\cdots l_p}^{\al_1\cdots \al_{p-1}}>0$ and non-zero complex coefficients $c_{l_1\cdots l_p}^{\al_1\cdots \al_p}\ne 0,$ and where we have kept enough terms to distinguish between all the non-identical roots of $\Psi.$

By the {\it cluster } $\left[ \begin{matrix} \al_1&\cdots &\al_p\\
l_1&\dots &l_{p}
\end{matrix}\right ],$
we shall designate all the roots $r(s_2)$, counted with  their multiplicities,
which satisfy
\begin{equation*}
r(s_2)-\Big(c_{l_1}^{\al_1}s_2^{a_{l_1}}+c_{l_1l_2}^{\al_1\al_2}s_2^{a_{l_1l_2}^{\al_1}}+\cdots+
c_{l_1\cdots l_p}^{\al_1\cdots \al_p}s_2^{a_{l_1\cdots
l_p}^{\al_1\cdots \al_{p-1}}}\Big)=O(s_2^b) \
\end{equation*}
for some exponent $b>a_{l_1\cdots l_p}^{\al_1\cdots \al_{p-1}}$.  The corresponding function
$$
c_{l_1}^{\al_1}s_2^{a_{l_1}}+c_{l_1l_2}^{\al_1\al_2}s_2^{a_{l_1l_2}^{\al_1}}+\cdots+
c_{l_1\cdots l_p}^{\al_1\cdots \al_p}s_2^{a_{l_1\cdots
l_p}^{\al_1\cdots \al_{p-1}}}
$$
will be called the {\it leading jet} of the sub-cluster 
$\left[ \begin{matrix} \al_1&\cdots &\al_p\\
l_1&\dots &l_{p}
\end{matrix}\right ].$

We
also introduce the sub-clusters 
 $\left[ \begin{matrix}
  \al_1&\cdots &\al_{p-1}&\cdot\\
l_1&\dots &l_{p-1}& l_p
\end{matrix}\right ],$
by
$$
 \left[ \begin{matrix} 
 \al_1&\cdots &\al_{p-1}&\cdot\\
l_1&\dots &l_{p-1}& l_p
\end{matrix}\right ]
:=
\bigcup\limits_{\al_p}
\left[ \begin{matrix} 
\al_1&\cdots &\al_p\\
l_1&\dots &l_p
\end{matrix}\right ].
$$

Each index $\al_p$ or $l_p$ varies in some finite range which we shall not specify here. We finally put 
$$
N\left[ \begin{matrix} 
\al_1&\cdots &\al_p\\
l_1&\dots &l_p
\end{matrix}\right ]:=\mbox{number of roots in} \,
\left[ \begin{matrix} 
\al_1&\cdots &\al_p\\
l_1&\dots &l_p
\end{matrix}\right ],
$$
$$
N \left[ \begin{matrix} 
 \al_1&\cdots &\al_{p-1}&\cdot\\
l_1&\dots &l_{p-1}& l_p
\end{matrix}\right ]:=\mbox{number of roots in}
\,  \left[ \begin{matrix} 
 \al_1&\cdots &\al_{p-1}&\cdot\\
l_1&\dots &l_{p-1}& l_p
\end{matrix}\right ]
$$
The number  
$
N\left[ \begin{matrix} 
\al_1&\cdots &\al_p\\
l_1&\dots &l_p
\end{matrix}\right ]
$ 
will be called the {\it multiplicity of the sub-cluster}
$\left[ \begin{matrix} \al_1&\cdots &\al_p\\
l_1&\dots &l_{p}
\end{matrix}\right ].
$

\medskip

Let $a_1<\dots< a_l<\dots<a_L$ be the distinct leading exponents of all the 
roots of $\Psi.$ Each exponent $a_l$ corresponds to the cluster $\left[ \begin{matrix} 
\cdot\\
l
\end{matrix}\right ],$ so that  the set of all roots of $\Psi$ can be divided as $\bigcup\limits_{l=1}^L
\left[ \begin{matrix} 
\cdot\\
l
\end{matrix}\right ]$.
Then we may write
\begin{equation}\label{Phifac}
\Phi(z_1,s_2)=U_0(z_1,s_2)z_1^{\nu_1}  \prod_{l=1}^L
\Phi\left[ \begin{matrix} 
\cdot\\
l\end{matrix}\right ](z_1,s_2),
\end{equation}
where
$$
\Phi\left[ \begin{matrix} 
\cdot\\
l
\end{matrix}\right ](z_1,s_2):=\prod_{r\in \left[ \begin{matrix} 
\cdot\\
l
\end{matrix}\right ]}(z_1-r(s_2)).
$$

We introduce the following quantities:
\begin{equation}\label{BandA}
B_l=B\left[ \begin{matrix} 
\cdot\\
l
\end{matrix}\right ]:=\nu_1+\sum_{\mu\ge l+1}N\left[ \begin{matrix} 
\cdot\\
\mu
\end{matrix}\right ],\quad
A_l=A\left[ \begin{matrix} 
\cdot\\
l
\end{matrix}\right ]:=\sum_{\mu\le l}a_\mu N\left[ \begin{matrix} 
\cdot\\
\mu
\end{matrix}\right ].
\end{equation}
Notice that $B_l$ is just the number of all roots with leading exponent strictly  greater than $a_l$ (where we here interpret the trivial roots corresponding to the factor $(z_1-0)^{\nu_1}$ in our representation of $\Phi$ as roots with exponent $+\infty$),
 and that 
$$
B_0>B_1>\dots>B_L,\qquad A_0<A_1\dots <A_L.
$$

If $N_{\all}$ denotes the total number of all roots of $\Phi$ away from the axis $s_2=0,$ including the trivial ones and  counted with their multiplicities, then  we can also write 
\begin{equation}\label{Blall}
B_l=N_\all-\sum_{\mu\le l}N\left[ \begin{matrix} 
\cdot\\
\mu
\end{matrix}\right ].
\end{equation}

Then the vertices of the Newton diagram $\N_d(\Phi)$ of $\Phi$ are the  points
$(B_l, A_l)), \  l=0,\dots,L,$ and the Newton polyhedron $\N(\Phi)$ is the convex hull
of the set $\cup ((B_l,A_l)+\bR_+^2)$ (compare Observation 1 in \cite {phong-stein}).

\begin{figure}[!h]
\centering
\includegraphics[scale=0.5]{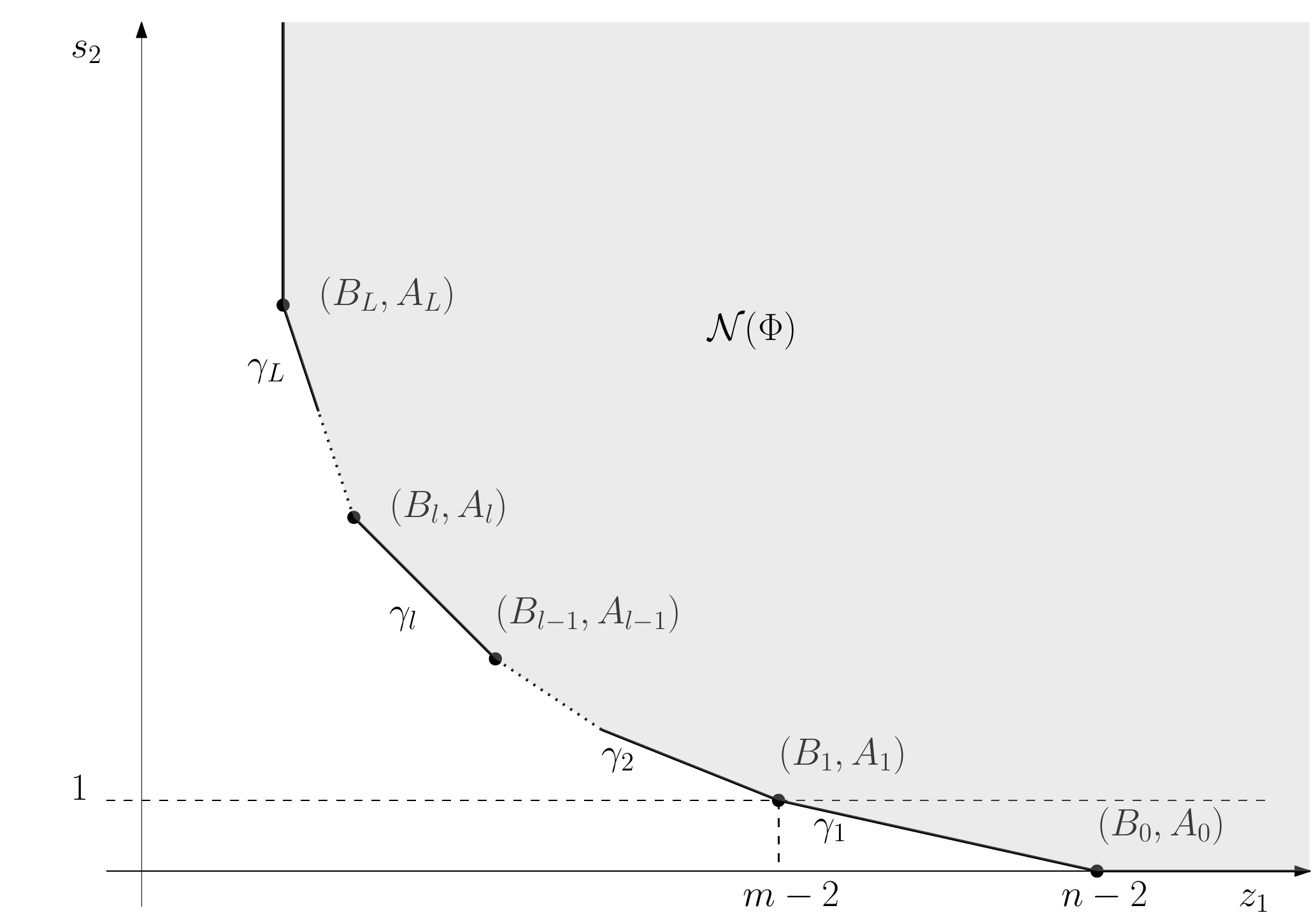}
  \caption{Edges}
  \label{edges}
\end{figure}

Notice also that 
$$A_l+a_lB_l=A_{l-1}+a_lB_{l-1},
$$
so that, if we put $\Delta A_l:=A_l-A_{l-1}, \Delta B_l:=B_l-B_{l-1},$ then 
$$
a_l=\frac {\Delta A_l}{-\Delta B_l}
$$
is just the  (modulus) of the slope of the edge $\ga_l:=[(B_{l},A_{l}),(B_{l-1},A_{l-1})]$ (compare Figure \ref{edges}).
\smallskip

Let $L_l:=\{(t_1,t_2)\in \bN^2:\ka^l_1t_1+\ka^l_2 t_2=1\}$ denote the line  supporting this edge $\ga_l.$
 It is easy to see that it is given by
\begin{eqnarray*}\nonumber
\ka^l_1&=&\frac {\Delta A_l}{A_l\Delta B_l-B_l\Delta A_l}=\frac {a_l}{A_l+a_lB_l},\nonumber \\
&&\\ 
\ka^l_2&=&\frac {\Delta B_l}{A_l\Delta B_l-B_l\Delta A_l}=\frac 1{A_l+a_lB_l}.\nonumber
\end{eqnarray*}
This defines a weight $\ka^l:=(\ka_1^l,\ka_2^l)$  associated to the edge $\ga_l,$ and we see that 
\begin{equation*}
\frac {\ka^l_1}{\ka^l_2} =a_l.
\end{equation*}
Note that
$$
a_1<a_2<\dots <a_l<\dots <a_L.
$$

Finally, fix $l,$ and let us determine the $\ka^l$-principal part $\Phi_{\ka^l }$ of $\Phi$ corresponding to the supporting line $L_l.$ To this end, observe that $\Phi$ has the same $\ka^l$-principal part  as the function 
$$
U_0(0,0)z_1^{\nu_1}\prod_{\al,\mu}\Big(z_1-c^\al_\mu s_2^{a_{\mu}}\Big)^{N \left[ \begin{matrix} 
\al\\
\mu
\end{matrix}\right ]}.
$$
Moreover, the $\ka^l$-principal part of $z_1-c^\al_\mu s_2^{a_{\mu}}$ is given by $-c^\al_\mu s_2^{a_{\mu}},$ if $\mu<l, $ and by $z_1$ if $\mu>l.$ This implies that 
\begin{equation}\label{Phiprinc}
\Phi_{\ka^l }(z_1,s_2)=c_l \,s_2^{A_{l-1}} z_1^{B_l}\prod_\al \Big(z_1-c^\al_l s_2^{a_{l}}\Big)^{N\aol}.
\end{equation}
In view of this identity, we shall say that the edge $\ga_l$ is {\it associated}  to the cluster of roots $\dotol.$
We collect these results in the following lemma.

\begin{lemma}
The vertices of the Newton polyhedron $\N(\Phi)$ of $\Phi$ are the  points
$(B_l,A_l),$  $ l=0,\dots,L,$ with $B_j,A_j$ given by \eqref{BandA}, and its  compact edges are the intervals 
$\ga_l:=[(B_{l},A_{l}),(B_{l-1},A_{l-1})],$ $l=1,\dots,L.$  Moreover the $\ka^l$-principal part  of $\Phi$ corresponding to the supporting line $L_l$ through the edge $\ga_l$ is given by \eqref{Phiprinc}.
\end{lemma}

Based on the information given in this subsection and  following ideas from \cite{phong-stein}, as well as from \cite{IKM-max} or \cite{IMmon}, we shall apply a {\it resolution (of singularity) algorithm}  which will allow us to reduce considerations step by step to suitable neighborhoods of smaller sub-clusters of roots until one ends up with a neighborhood a sub-cluster containing only one single root (with multiplicity)

\medskip

\setcounter{equation}{0}

\section{Analysis in  Step 1 of the resolution algorithm  }\label{step1}

We begin by working in our original coordinate $x_1,$ and recall that we want to  denote $x_1$ by  $z_1.$  Recall  also that we are restring attention to the upper half-plane  $s_2>0.$ 
\medskip

We shall decompose this half-plane into  $\ka^l$-{\it homogeneous  domains}  $D_l=D_l^{(1)}$ of the form
$$
D_l:=\{(z_1,s_2):   2^{-M} s_2^{a_l}\le |z_1|\le 2^{M} s_2^{a_l}\}, \quad l=1,\dots,L,
$$
 and  {\it transition domains}
\begin{eqnarray*}
E_{0}&:=&\{(z_1,s_2):   2^M s_2^{a_{1}}<|z_1|\}, \\
E_l&:=&\{(z_1,s_2):   2^M s_2^{a_{l+1}}<|z_1|<2^{-M} s_2^{a_l}\}, \quad l=1,\dots,L-1,\\
E_L&:=&\{(z_1,s_2):  |z_1|<2^{-M} s_2^{a_L}\} 
\end{eqnarray*}
between two such domains of different  type of homogeneity.
Here,  $M\gg 1$ is an integer which will  have to be chosen sufficiently large later on.

\begin{figure}[h]
\centering
\includegraphics[scale=0.4]{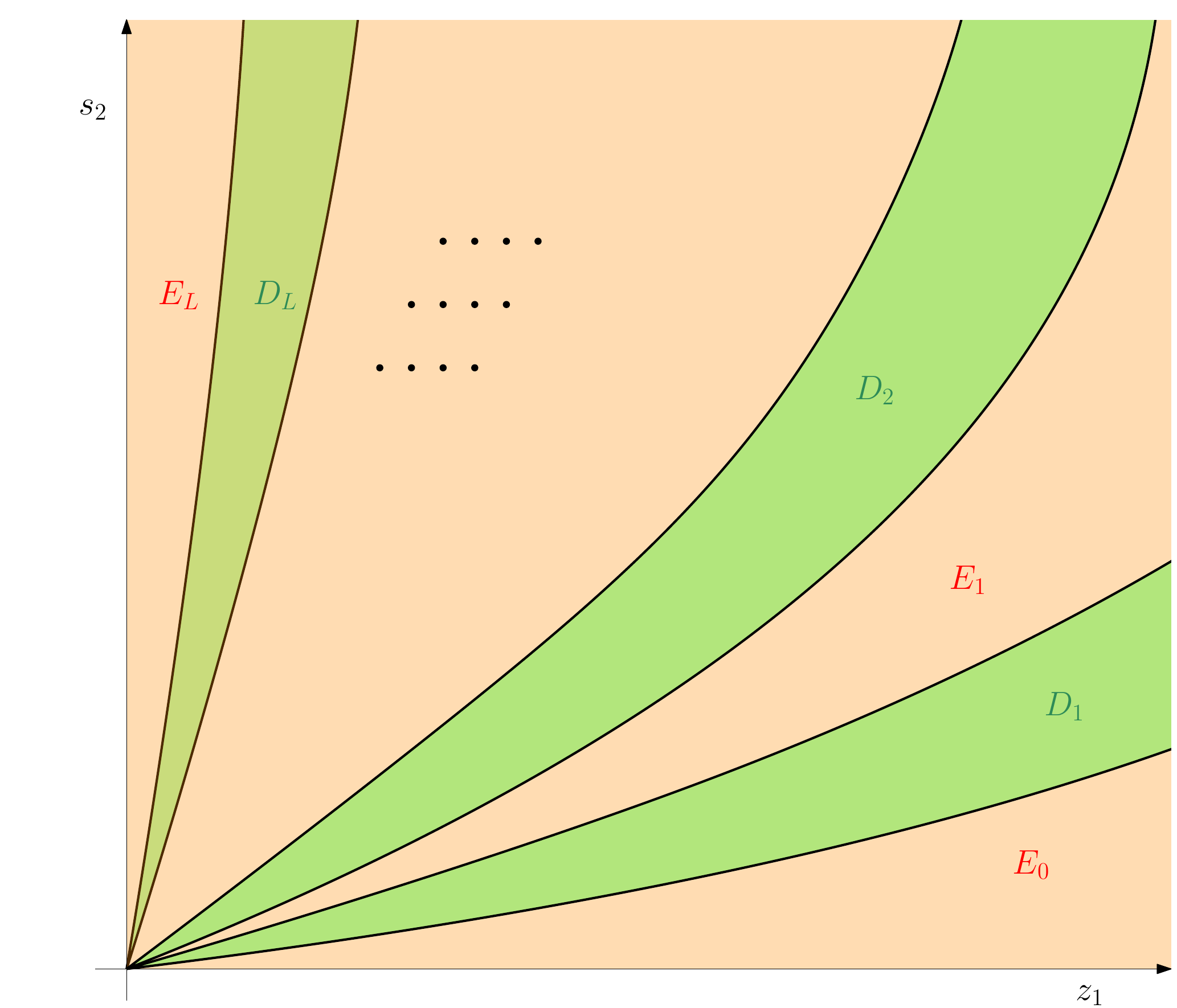}
  \caption{Domain Decomposition}
  \label{domains}
\end{figure}

\medskip
To localize to domains of type $D_l$ in a smooth way, we put more precisely
\begin{equation*}
\rho_l(z_1,s_2):=\chi_0\Big(\frac{ z_1}{2^M s_2^{a_l}}\Big) -\chi_0\Big(\frac{ z_1}{2^{-M} s_2^{a_l}}\Big), \quad l=1,\dots,L,
\end{equation*}
and replace $J(\la, s)$ in \eqref{Jla} by
$$
J^{\rho_l}(\la, s):=\int e^{-i\la s_3(\breve\phi_1(z_1, s_2)+s_1z_1)}  \eta(z_1,  s_2) \rho_l(z_1,s_2)\,dz_1.
$$
The corresponding functions such as $\mu^\la$ and $ F(\la,s_2,s_3,y_1),$ and the corresponding maximal operator  will be designated  by means an extra upper index $\rho_l,$ such as 
$\mu^{\rho_l,\la},$  $ F^{\rho_l}(\la,s_2,s_3,y_1)$ and $\M^{\rho_l}.$

  \medskip
Similarly, in order to localize to domains of type $E_l,$ we put, for $l=1,\dots, L-1,$ 
\begin{eqnarray*}
\tau_l(z_1,s_2)&:=&\chi_0\Big(\frac{ z_1}{2^{-M} s_2^{a_l}}\Big) \, -\chi_0\Big(\frac{ z_1}{2^{M} s_2^{a_{l+1}}}\Big),
\end{eqnarray*}
and define $\tau_0$ and $\tau_L$ by means of the obvious modifications of this definition. The corresponding functions  and the corresponding maximal operator will be denoted by $\mu^{\tau_l,\la}, F^{\tau_l}(\la,s_2,s_3,y_1),\M^{\tau_l},$ ... .

Recall that we may assume that $|s_2|<\ve,$ for any given $\ve>0.$ It is then easy to see that, once we have chosen $M,$ by assuming that $\ve>0$ is chosen sufficiently small, then $2^{-M} s_2^{a_l}>2^{M} s_2^{a_{l+1}}$ for every $l.$ 
Then  the functions  $\rho_l, l=1,\dots L,$ jointly with the functions $\tau_l, l=0, \dots L,$ form a partition of unity consisting of non-negative functions.
\smallskip

Finally note that it suffices to suitably control all the maximal operators  defined in this way in order to prove the first statement in Theorem \ref{thm-a-}.

\medskip
\subsection{Dyadic localization in $z_1$ and $s_2$}\label{dyadic}
\medskip

Let us localize to 
$$
|z_1|\sim 2^{-j}, \ s_2\sim 2^{-k},
$$
where we may assume that $j,k$ are sufficiently large integers $\gg 1.$ To this end, we shall  define scaled coordinates $z, \si_2$ by writing 
$z_1=2^{-j} z, \ s_2=2^{-k} \si_2,$ so that
$$
|z|\sim 1, \ s_2\sim 1.
$$
Accordingly, let 
\begin{eqnarray*}
J_{j,k}(\la, s)&:=&\chi_1(2^k s_2)\int e^{-i\la s_3(\breve\phi_1(z_1, s_2)+s_1z_1)}  \eta(z_1,  s_2)\chi_1(2^jz_1)\,dz_1\\
&=&\chi_1(\si_2)\, 2^{-j} \int e^{-i\la s_3(\breve\phi_1(2^{-j}z, s_2)+2^{-j}s_1z)}  \chi_1(z)\, \eta(2^{-j}z,  2^{-k}\si_2) \,dz,
\end{eqnarray*}
and 
$$
J^{\tau_l}_{j,k}(\la, s)=\chi_1(\si_2)\, 2^{-j} \int e^{-i\la s_3(\breve\phi_1(2^{-j}z, s_2)+2^{-j}s_1z)}\chi_1(z)\,   (\tau_l\eta)(2^{-j}z,  2^{-k}\si_2)\,dz.
$$
The corresponding functions such as $\mu^\la$ and $F(\la,s_2,s_3,y_1),$   and the corresponding maximal operator  will be designated  by means an extra lower index $j,k,$ such as 
$\mu^{\la}_{j,k}, F_{j,k}(\la,s_2,s_3,
y_1),\M_{j,k},$ and $\M^{\rho_l}_{j,k}.$ 
\medskip

 \noi {\bf Convention:} In order to defray  the notation, we shall  use the shorthand notation
$\eta dz= \eta(2^{-j}z,  2^{-k}\si_2) \,dz$ in these integrals, as well as  in many similar situations later on, where a smooth amplitude  factor $\eta$ depending on integration variables (here the variable $z$) and small parameters (here $2^{-j}$ and $2^{-k}$) in such a way that we have uniform estimates for  derivatives of these amplitude factors. The concrete meaning of $\eta$  will be allowed  to be different at every instance where this symbol is used. 

Thus, for instance we shall write
$$
J^{\tau_l}_{j,k}(\la, s)=\chi_1(\si_2)\, 2^{-j} \int e^{-i\la s_3(\breve\phi_1(2^{-j}z, s_2)+2^{-j}s_1z)}\chi_1(z)\,   \tau_l(2^{-j}z,  2^{-k}\si_2)\,\eta dz.
$$
Similarly, we write
$$
J^{\rho_l}_{j,k}(\la, s)=\chi_1(\si_2)\, 2^{-j} \int e^{-i\la s_3(\breve\phi_1(2^{-j}z, s_2)+2^{-j}s_1z)}\chi_1(z)\,   \rho_l(2^{-j}z,  2^{-k}\si_2)\,\eta dz.
$$

For the contribution by the transition domain $E_l,$ i.e., for $ J^{\tau_l},$ it is easily seen that we may assume that $J^{\tau_l}_{j,k}\equiv 0,$ unless 
\begin{eqnarray}\label{Ejk1}
&&j\ge ka_l+M/2, \qquad \text{if} \quad l=1,\dots L,\\
&&j\le ka_{l+1}-M/2,\qquad \text{if} \quad l=0,\dots L-1.\label{Ejk2}
\end{eqnarray}
Similarly,  for the contribution by the homogeneous domain $D_l, 1\le l\le L,$ i.e., for $ J^{\rho_l},$ we  may assume that $J^{\rho_l}_{j,k}\equiv 0,$ unless 
\begin{equation}\label{Djk1}
|j-ka_l|\le M/2.
\end{equation}
 We thus see that the factors $\tau_l(2^{-j}z,  2^{-k}\si_2)$ respectively  $\rho_l(2^{-j}z,  2^{-k}\si_2)$ can be absorbed into a modified factor  $\eta$ in the definitions of $J^{\tau_l}_{j,k},$ respectively $J^{\rho_l}_{j,k},$ so that, with a slight abuse of notation, we may and shall henceforth assume that 
 
 \begin{eqnarray*}
J^{\tau_l}_{j,k}(\la, s)&=&J^{\rho_l}_{j,k}(\la, s)=J_{j,k}(\la, s)=\chi_1(\si_2)\, 2^{-j} \int e^{-i\la s_3(\breve\phi_1(2^{-j}z, s_2)+2^{-j}s_1z)}  \chi_1(z)\, \eta dz
\end{eqnarray*}
for the range of $j$s and $k$s described by \eqref{Ejk1} and  \eqref{Ejk2}, respectively \eqref{Djk1}.

\medskip
\subsection{Resolution of singularity on the transition domain $E_l$}\label{Elresolution}
\medskip
 
It is well-known that on $E_l,$ we may write 
\begin{equation}\label{Elreso}
\Phi(z_1,s_2)=V(z_1,s_2)\, s_2^{A_l}z_1^{B_l},\qquad (z_1,s_2)\in E_l,
\end{equation}
where the function $V$ is analytic away from the coordinate axes and  $|V(z_1,s_2)|\sim 1$ (compare \cite{phong-stein}). Actually, we shall need more precise information on the function $V:$  by \eqref{Phifac}, we may write 
$$
\Phi(z_1,s_2)=U_0(z_1,s_2)z_1^{\nu_1}  \prod_{l'=1}^L\prod_{r\in \left[ \begin{matrix} 
\cdot\\
l'
\end{matrix}\right ]}(z_1-r(s_2)).
$$
Now, if $l'\ge l+1,$ then $a_{l'}\ge a_{l+1},$ and since $a_{l'}$ is the leading exponent of any root $r$  in the cluster  
$\left[ \begin{matrix} 
\cdot\\
l'
\end{matrix}\right ],
$
we may factor 
$$
z_1-r(s_2)=z_1\big(1-\frac {r(s_2)}{z_1}\big),\quad\text{where}\  \frac {r(s_2)}{z_1}=O(2^{-M/4}).
$$
Similarly, if $l'\le l,$ then we may factor 
$$
z_1-r(s_2)=s_2^{a_{l'}}\big(-\frac {r(s_2)}{s_2^{a_{l'}}}+\frac {z_1}{s_2^{a_{l'}}}\big),\quad\text{where}\  \frac{r(s_2)}{s_2^{a_{l'}}}=c^\al_{l'}\ne 0, \text{   and  }  \frac {z_1}{s_2^{a_{l'}}}=O(2^{-M/4}).
$$
In combination with \eqref{BandA}, this implies \eqref{Elreso}, with 
\begin{equation*}
V(z_1,s_2):=U_0(z_1,s_2) \prod_{l'=1}^l\prod_{r\in \left[ \begin{matrix} 
\cdot\\
l'
\end{matrix}\right ]}\big(-\frac {r(s_2)}{s_2^{a_{l'}}}+\frac {z_1}{s_2^{a_{l'}}}\big)
 \prod_{l'=l+1}^L\prod_{r\in \left[ \begin{matrix} 
\cdot\\
l'
\end{matrix}\right ]}\big(1-\frac {r(s_2)}{z_1}\big).
\end{equation*}
This implies in particular that, as a function of $z_1,$ $V(z_1,s_2)$ admits a convergent Laurent expansion of the form
\begin{equation}\label{VLaurent}
V(z_1,s_2)=\sum\limits_{k=-B_l}^\infty z_1^k \rho_k(s_2),
\end{equation}
where the $\rho_k(s_2)$ are fractionally analytic in $s_2.$  Moreover, on $E_l,$ we may assume that $|\rho_0(s_2)|\sim 1$ and
\begin{equation*}
\sum\limits_{0\ne k=-B_l}^\infty |z_1^k \rho_k(s_2)|\ll 1.
\end{equation*}

In order derive information from \eqref{Elreso} on the function $\Phi_0$ in \eqref{Phi0}, the following simple lemma will be useful:
\begin{lemma}\label{int}
Given $B\in\NN$ and a smooth function $V(z)$ on an interval $I$ not containing $0,$ then there is a smooth function $U(z)$ on this interval such hat 
$$
\frac d{dz} \big( U(z) z^{B+1}\big)=V(z) z^B.
$$
\end{lemma}
\begin{proof} 
We can choose 
$$
U(z):= z^{-(B+1)}\int z^B V(z)\, dz,
$$
where the integral can designate any primitive of the integrand.
\end{proof} 
Applying this to the term of order $k$ in \eqref{VLaurent}, with $B=B_l$ and $B_l+k\ge 0,$ we see that if we define 
$$
V_1(z_1,s_2):=\sum\limits_{k=-B_l}^\infty z_1^k \frac {\rho_k(s_2)}{B_l+k+1},
$$
then we find that 
$$
\pa_{z_1} \big( V_1(z_1,s_2) z_1^{B_l+1}\big)=V(z_1,s_2)z_1^{B_l},
$$
and iterating this once more, we obtain
$$
\pa_{z_1}^2 \big(U(z_1,s_2) z_1^{B_l+2}\big)=V(z_1,s_2)z_1^{B_l},
$$
if we set
$$
U(z_1,s_2):=\sum\limits_{k=-B_l}^\infty z_1^k \frac {\rho_k(s_2)}{(B_l+k+1)(B_l+k+2)}.
$$
Consequently, 
\begin{equation*}
\pa_{z_1}^2\big(U(z_1,s_2)s_2^{A_l} z_1^{B_l+2}\big)=\Phi(z_1,s_2)=\pa_{z_1}^2\breve\phi_1(z_1, s_2).
\end{equation*}
Clearly, $U(z_1,s_2)$ has similar properties as $V(z_1,s_2); $ in particular, $|U(z_1,s_2)|\sim 1.$ 
Note also that the function $U(z_1,s_2)s_2^{A_l} z_1^{B_l+2}$ is analytic in $z_1$ and vanishes of second order at $z_1=0.$ 
Consequently, we see that 
\begin{equation*}
\breve\phi_1(z_1, s_2)=U(z_1,s_2)s_2^{A_l} z_1^{B_l+2}+z_1g(s_2)+h(s_2),
\end{equation*}
where 
\begin{eqnarray*}
\begin{split}
h(s_2)&= \breve\phi_1(0, s_2)\equiv 0, \\
 g(s_2)&=\pa_{z_1}\breve\phi_1(0, s_2)=s_2^2q(0,s_2)=:s_2^2q_0(s_2).
 \end{split}
\end{eqnarray*}
Thus, we find that, on $E_l,$
\begin{equation}\label{Phi0onEl}
\Phi_0(z_1,s'):=\breve\phi_1(z_1, s_2)+s_1z_1=U(z_1,s_2)s_2^{A_l} z_1^{B_l+2}+z_1(s_1+s^2q_0(s_2)).
\end{equation}

Let us finally again apply our changes of coordinates $z_1=2^{-j} z, \ s_2=2^{-k} \si_2.$  Then, if we put $A:=A_l, B:= B_l$ and 
$$
d_{j,k}:=2^{-kA-j(B+2)},
$$
assuming that \eqref{Ejk1}, \eqref{Ejk2} hold true, we have that 
\begin{equation}\label{Phi0s}
\Phi_0^s(z,s_1,\si_2):=\Phi_0(z_1,s')= U^s(z,\si_2) \, d_{j,k}\si_2^{A} z^{B+2}+2^{-j}z\big(s_1+2^{-2k}\si_2^2g^s(\si_2)\big),
\end{equation}
with $U^s(z,\si_2):=U(2^{-j}z,2^{-k}\si_2)$ and $g^s(\si_2):=q_0(2^{-k}\si_2).$ 
Note that $|U^s(z,\si_2)|\sim 1,$ whereas derivatives of $ U^s(z,\si_2)$ can be assumed to be very small, and recall that
$$
 |z|\sim 1, \sigma_2\sim 1, |s_1|\ll 1.
 $$
 Note also that 
 \begin{equation}\label{jdjk}
d_{j,k}\le2^{-2j}\ll 2^{-j}.
\end{equation}

\subsection{$L^2$ - estimation of $\M_{j,k}^\la$ for the contribution by $E_l$}\label{L2onEl}
\medskip
We begin by observing that since $d_{j,k}\ll 2^{-j},$ by choosing $s_1$ and $s_2$ properly we can always find a critical point 
of $\Phi_0^s$ in \eqref{Phi0s} w.r. to $z.$ 
Thus, in view of Plancherel's theorem,  the best possible estimate of $\M_{j,k}^\la$ can be achieved by means of an application of van der Corput's estimate of order $2$ to 
$$
J_{j,k}(\la, s)=\chi_1(\si_2)\, 2^{-j} \int e^{-i\la s_3\Phi_0^s(z,s_1,\si_2)}  \chi_1(z)\, \eta dz.
$$ 
This leads to 
$
|J_{j,k}(\la, s)|\lesssim 2^{-j} (\la d_{j,k})^{-1/2},
$
if $\la d_{j,k}\ge 1.$ If $\la d_{j,k}\le 1,$ the trivial estimate
$
|J_{j,k}(\la, s)|\lesssim 2^{-j} 
$
holds true. Since, by \eqref{mula2}, $|\widehat{\mu^\la_{j,k}}|\lesssim \la^{-1/2} |J_{j,k}(\la, s)|,$ the usual arguments show that this implies the following $L^2$- estimate:
\begin{equation}\label{L2Eljk}
\|\M^\la_{j,k}\|_{2\to 2}\lesssim 2^{-j}(1+\la d_{j,k})^{-\frac 12}.
\end{equation}

\medskip
\begin{remark}
Throughout Section \ref{step1}, we shall  assume that $\la 2^{-j}\ge 1,$ i.e., $j\lesssim \log\la.$ The case where $\la 2^{-j}\ge 1$ will be discussed in   Subsection \ref{tLpElla<j} of Section \ref{step2}.
\end{remark}

\medskip
\subsection{$L^{1+\ve}$ - estimates of $\M_{j,k}^\la$ for the contribution by $E_l, l=1,\dots,L,$ when $\la\ge 2^j$}\label{L1onEl}
\medskip
In view of Lemma \ref{maxproj}, the $L^{1+\ve}$ - estimates  require pointwise bounds for $\mu_{j,k}^\la.$ By \eqref{mula}, we have
$$
\mu_{j,k}^\la(y+\Ga)=\la^{\frac52}2^{-k}
\int \int F_{j,k}(\la,2^{-k}\si_2,s_3,y_1) e^{-i\la s_3(s_2^2B(s_2)-2^{-k}\si_2 y_2-y_3)} \chi_1(\si_2) d\si_2 \chi_1(s_3) ds_3,
$$
where 
\begin{eqnarray*}
F_{j,k}(\la,2^{-k}\si_2,s_3,y_1) &=& 2^{-j}\iint e^{-i\la s_3(\Phi_0^s(z,s_1,\si_2)- s_1y_1)} \chi_0(s_1)\chi_1(z)  \,\eta dz ds_1.
\end{eqnarray*}

\medskip

Let us put $Y_1:=2^jy_1,$ and use \eqref{Phi0onEl} to write the phase function in the latter integral as
\begin{equation}\label{phaseF}
\Phi_0^s(z,s_1,\si_2)- s_1y_1=2^{-j}\big[U^s(z,\si_2) \,2^j d_{j,k}\si_2^{A} z^{B+2}+z2^{-2k}\si_2^2 g^s(\si_2) +s_1(z-Y_1)\big]
\end{equation}

Let us first consider the contribution by the region where  $|Y_1|\gg1 $ and $|Y_1|\ge \la^{\de-1} 2^j,$  where we assume that $0<\de \ll 1.$ Here, $\la 2^{-j} |Y_1|\gg \la^\de,$ and since $|z|\sim 1,$ integrations by parts in $s_1$ show that 
$F_{j,k}(\la,2^{-k}\si_2,s_3,y_1)=O(\la^{-N})$ for every $N\in\NN,$ and thus also $\mu_{j,k}^\la(y+\Ga)=O(\la^{-N})$ for every $N\in\NN,$ which shows that this leads just to a small error term.

\smallskip

Consider next the contribution by the region where  $1\ll |Y_1|\le \la^{\de-1} 2^j.$ Then, by \eqref{jdjk}, 
$$
\la d_{j,k}\le 2^{\frac j{1-\de}}2^{-2j} \ll 1,
$$
if we choose $\de$ sufficiently small. This shows that we can include the exponential factor corresponding to the term $U^s(z,\si_2) \, d_{j,k}\si_2^{A} z^{B+2}$ of the phase $\Phi_0^s$ into the amplitude and consequently  ignore it. Let us then collect all summands of the complete phase in  the oscillatory integral $\mu_{j,k}^\la(y+\Ga)$ which depend on $\si_2:$
\begin{eqnarray}\label{onlys2}
\begin{split}
&2^{-j}z2^{-2k}\si_2^2\tilde g(\si_2)) +2^{-2k}\si_2^2B(2^{-k}\si_2)-2^{-k}\si_2 y_2\\
&=2^{-2k}\big[\si_2^2B(2^{-k}\si_2)-\si_2 Y_2+2^{-j}\si_2^2z\tilde g(\si_2)\big],
\end{split}
\end{eqnarray}
where we have set $Y_2:=2^ky_2.$
\medskip

Let us now decompose the region where  $1\ll |Y_1|\le \la^{\de-1} 2^j$ into the regions 
$$
\Delta_{\iota_1}:= \{y:|Y_1|\sim 2^{\iota_1}\}, 
$$
where clearly we may assume that $1\ll\iota_1<j.$ We further decompose $\Delta_{\iota_1}$ into the subregions
\begin{eqnarray*}
\Delta_{\iota_1,\iota_2}&:=& \{y:|Y_1|\sim 2^{\iota_1}, |Y_2|\sim 2^{\iota_2}\}, \quad 1\ll \iota_2\le k,\\
\Delta_{\iota_1,0}&:=& \{y:|Y_1|\sim 2^{\iota_1}, |Y_2|\lesssim 1\}.
\end{eqnarray*}
Then 
$$
|\Delta_{\iota_1,\iota_2}|\lesssim 2^{\iota_1+\iota_2-j-k}.
$$
The corresponding measures which are given by restricting  $\mu_{j,k}^\la $ to the corresponding sets and the associated maximal operators will be denoted $\mu_{k,j,\iota}^\la$ and $\M_{j,k,\iota}^\la.$ 
\smallskip

Now, clearly, if $y\in \Delta_{\iota_1},$ then we may integrate by parts in $s_1.$ 

Moreover, if $\iota_2\gg1,$  and if $\la 2^{-2k+\iota_2}\ge 1,$ then \eqref{onlys2} shows that we can also integrate by parts in $\si_2,$ and altogether this leads to the following estimate for $\mu_{k,j,\iota}^\la:$
$$
|\mu_{j,k,\iota}^\la(y+\Ga)|\lesssim \la^{\frac 52} 2^{-j-k} (\la 2^{-j} 2^{\iota_1})^{-1}(\la 2^{-2k} 2^{\iota_2})^{-1}=\la^{\frac 12} 2^{k-\iota_1-\iota_2}.
$$
Thus, Lemma \ref{maxproj} shows that, for every $\ve >0,$ 
$$
\|\M_{j,k,\iota}^\la\|_{1+\ve\to 1+\ve}\lesssim \la^{\frac 12}2^{-j}.
$$
But, recall that we are assuming that $j\lesssim \log\la,$ and 
note that $\la 2^{-2k+\iota_2}\ge 1$ implies that $\la 2^{-k}\ge 1,$ hence $\iota_2\le k\lesssim \log\la,$ 
so that by summing over these $\iota${\tiny s} we will just pick up an extra factor $(\log\la)^2$ compared to previous estimate:
\begin{equation}\label{Mjkiota}
\sum\limits_{\iota_1\lesssim j,\iota_2\le k}\|\M_{j,k,\iota}^\la\|_{1+\ve\to 1+\ve}\lesssim \la^{\frac 12+\ve}2^{-j},
\end{equation}
for every $\ve>0.$
\smallskip

Assume next that $\iota_2\gg1$  and $\la 2^{-2k+\iota_2}<1.$ Then $1\ll2^{\iota_2}<2^{2k}\la^{-1},$ and so we cannot gain anything from the integration in $\si_2$ and arrive at the following estimate:
$$
|\mu_{j,k,\iota}^\la(y+\Ga)|\lesssim \la^{\frac 52} 2^{-j-k} (\la 2^{-j} 2^{\iota_1})^{-1}=\la^{\frac 32} 2^{-k-\iota_1},
$$
hence $\|\M_{j,k,\iota}^\la\|_{1+\ve\to 1+\ve}\lesssim \la^{\frac 32}2^{-j+\iota_2 -2k}.$ 
Thus, summing over the corresponding $\iota$s,  we get 
$$
\sum\limits_{\iota_1\lesssim j,\, 2^{\iota_2}<2^{2k}\la^{-1}}\|\M_{j,k,\iota}^\la\|_{1+\ve\to 1+\ve}
\lesssim \sum\limits_{\iota_1\lesssim j,\, 2^{\iota_2}<2^{2k}\la^{-1}}\la^{\frac 32}2^{-j+\iota_2 -2k}<\la^{\frac 12+\ve}2^{-j},
$$
which matches with \eqref{Mjkiota}.
\smallskip

Assume finally that $\iota_2=0.$ Then, in place of an integration by parts in $\si_2,$ we apply van der Corput's estimate of order 2  in $\si_2,$ which leads to 
$$
|\mu_{j,k,\iota}^\la(y+\Ga)|\lesssim \la^{\frac 52} 2^{-j-k} (\la 2^{-j} 2^{\iota_1})^{-1}(\la 2^{-2k})^{-\frac 12}
=\la 2^{-\iota_1},
$$
hence
\begin{equation}\label{mujkL1la}
\|\M_{j,k,(\iota_1,0)}^\la\|_{1+\ve\to 1+\ve}\lesssim \la^{1+\ve}2^{-j-k}.
\end{equation}
Again, we can sum in $\iota_1\lesssim j,$ and combining all our estimates we see that the contribution $\M_{j,k,|Y_1|\gg 1}^\la$
of the region where $|Y_1|\gg 1$ to  $\M_{j,k}^\la$ can be estimated by 

\begin{equation}\label{Mjk>1}
\|\M_{j,k,|Y_1|\gg 1}^\la\|_{1+\ve\to 1+\ve}\lesssim \la^{\frac 12+\ve}2^{-j}+ \la^{1+\ve}2^{-j-k},
\end{equation}
for every $\ve>0.$
\medskip

What remains is the contribution $\M_{j,k,|Y_1|\lesssim 1}^\la$ of the region where $|Y_1|\lesssim 1$ to  $\M_{j,k}^\la.$ 
A look at the phase \eqref{phaseF}  shows that we are now in the position to apply Lemma \ref{pinvers} to the oscillatory integral $F_{j,k}(\la,2^{-k}\si_2,s_3,y_1)$ in the variables $(s_1,z)$ in place of $(\eta,x),$ with $N:=\la 2^{-j}\ge 1.$ 
Moreover, observe that $s_1$ and $z$   result  from our original coordinates $s_1$ and $z_1=x_1$ by means of  a smooth change of coordinates.  Passing back to our original coordinate, according to Remarks \ref{critvalue}  we therefore must evaluate the original phase $\breve\phi_1(x_1, s_2)+s_1x_1-s_1y_1$ in \eqref{defiF} at $x_1=y_1,$ and thus get
\begin{equation}\label{applemma5.2}
 F_{j,k,|Y_1|\lesssim1}(\la,s_2,s_3,y_1)   = (\la2^{-j})^{-1} 2^{-j}e^{-i\la s_3(\breve\phi_1(y_1, s_2)}    \eta(y_1,2^{-k}\si_2).
  \end{equation}
  This implies that
  $$
\mu_{j,k,|Y_1|\lesssim1}^\la(y+\Ga)=\la^{\frac32}2^{-k}
\int  e^{-i\la s_3\big(\breve\phi_1(y_1, 2^{-k}\si_2)+2^{-2k}\si_2^2B(2^{-k}\si_2)-2^{-k}\si_2 y_2-y_3\big)} \chi_1(\si_2) \chi_1(s_3) \,\eta d\si_2ds_3.
$$
By \eqref{legendamins}, the complete phase is here given by
$$
2^{-2k}\big[\si_2^2 B(2^{-k}\si_2)+y_1\si_2^2 q(y_1, 2^{-k}\si_2) -2^{k}\si_2 (y_2 -y_1^m\omega(y_1))\big]   +y_1^n\beta(y_1)-y_3.
$$
In order to estimate the result of the integration in $\si_2,$ we can argue now in a similar way as before and decompose the region where $|Y_1|\lesssim 1$ into the subregions
\begin{eqnarray*}
\Delta_{\iota_2}&:=& \{y:|Y_1|\lesssim 1,2^k |y_2 -y_1^m\omega(y_1)|\sim 2^{\iota_2}\}, 1\ll \iota_2\le k,\\
\Delta_{0}&:=& \{y:|Y_1|\lesssim 1, 2^k|y_2 -y_1^m\omega(y_1)|\lesssim 1\}.
\end{eqnarray*}
Then 
$$
|\Delta_{\iota_2}|\lesssim 2^{-j} 2^{\iota_2-k}.
$$
If $\iota_2\gg 1,$ then we may integrate by parts in $\si_2$ and obtain 
$$
|\mu_{j,k,|Y_1|\lesssim1,\iota_2}^\la(y+\Ga)|\lesssim \la^{\frac32}2^{-k}(\la 2^{-2k+\iota_2})^{-1}=\la^{\frac 12}2^{k-\iota_2},
$$
hence 
$$
\|\M_{j,k,|Y_1|\lesssim1,\iota_2}^\la\|_{1+\ve\to 1+\ve}\lesssim \la^{\frac 12}2^{-j}.
$$
And, if $\iota_2=0,$ then we apply van der Corput's estimate of order 2 in $\si_2,$ which leads to 
$$
|\mu_{j,k,|Y_1|\lesssim1,\iota_2}^\la(y+\Ga)|\lesssim \la^{\frac32}2^{-k}(\la 2^{-2k})^{-\frac 12}=\la,
$$
hence 
$$
\|\M_{j,k,|Y_1|\lesssim1,\iota_2}^\la\|_{1+\ve\to 1+\ve}\lesssim \la 2^{-k} 2^{-j}.
$$
These estimates match with \eqref{Mjk>1}, and thus altogether we find that, for every $\ve>0,$
\begin{equation}\label{Mjk}
\|\M_{j,k}^\la\|_{1+\ve\to 1+\ve}\lesssim \la^{1+\ve}2^{-j-k}+  \la^{\frac 12+\ve}2^{-j}.
\end{equation}

\medskip
\subsection{$L^p$ - estimates of $\M_{j,k}^\la$  for the contribution by $E_l$ when $\la\ge 2^j$}\label{LpEl}
\medskip

Interpolation between the $L^2$ - estimate \eqref{L2Eljk} and the $L^{1+\ve}$ - estimate \eqref{Mjk} leads to 
\begin{eqnarray}\nonumber
\|\M_{j,k}^\la\|_{p\to p}&\lesssim&
 (\la^{1+\ve}2^{-k})^{\frac 2p -1}(1+\la d_{j,k})^{-\frac 12 (2-\frac 2p)} 2^{-j} 
+ (\la^{\frac 12+\ve})^{\frac 2p -1}(1+\la d_{j,k})^{-\frac 12 (2-\frac 2p)} 2^{-j}\\ \nonumber
&\lesssim&
\la^{\ve'} (\la d_{j,k})^{\frac 2p -1} (1+\la d_{j,k})^{\frac 1p-1} \, (d_{j,k}^{-1})^{\frac 2p-1}\, 2^{-k(\frac 2p -1)}\,2^{-j} \\
&&\quad +\, \la^{\ve'} 
       (\la d_{j,k})^{\frac 1p -\frac 12}(1+\la d_{j,k})^{\frac 1p-1} \, (d_{j,k}^{-1})^{\frac 1p-\frac 12}\, 2^{-j},\label{Mlajkp}
\end{eqnarray}
if $1<p< 2.$ 

\smallskip
\begin{itemize}
\item Since $1/p-1/2>0,$ we can first  sum both terms over all dyadic $\la$s such that $\la d_{j,k}\le 1,$ provided we choose $\ve'>0$ sufficiently small.
\item If $\la d_{j,k}> 1,$ then the total exponent of $\la d_{j,k}$ in the first summand is 
$$
\frac 2p-1+\frac 1p-1=\frac 3p-2<0,
$$
since we are assuming that $p>3/2.$ Similarly, the  total exponent of $\la d_{j,k}$ in the second summand is 
$$
\frac 1p-\frac 12+\frac 1p-1=\frac 2p-\frac 32<\frac 43-\frac 32<0.
$$
Thus, again both terms can be summed over all dyadic  $\la$s such that $\la d_{j,k}>1,$ provided we choose $\ve'>0$ sufficiently small.

\end{itemize}

This easily shows that, for every $\ve'>0,$ 
\begin{equation}\label{suminla1}
\sum\limits_{\la\gg 1}\|\M_{j,k}^\la\|_{p\to p}\lesssim (d_{j,k}^{-1})^{\frac 2p-1+\ve'}\, 2^{-k(\frac 2p -1)}\,2^{-j} 
+  (d_{j,k}^{-1})^{\frac 1p-\frac 12+\ve'}\, 2^{-j}.
\end{equation}

In order to also sum over the $j$s and $k$s associated to $E_l,$  it turns out that we shall have to make use of the condition 
\eqref{Ejk1}. To this end, we shall first derive a little lemma which  will help to relate any edge of $\N(\Phi)$ having the
vertex $(B,A):=(B_l,A_l)$ associated to the transition domain $E_l$ as an endpoint.

\medskip
\subsection{A geometric lemma and $L^p$-estimation of $\M^{\tau_l}, l=1,\dots,L$}\label{geolemma}
\medskip

\begin{lemma}\label{geo}
Given  a weight $\ka=(\ka_1,\ka_2)$ such that $\ka_1>0, \ka_2\ge 0,$ and the associated  line 
$$
L_\ka:=\{(t_1,t_2): \ka_1t_1+\ka_2 t_2=1\},
$$
assume that $(B,A)$ is any point on $L_\ka.$ We define the number $n_\ka-2$ as the $t_1$-coordinate of the intersection of the line $L_\ka$ with the horizontal line $t_2=1$ (compare Figure \ref{figure geo}). Then 
$$
n_\ka=\frac{A-1}a+B+2,
$$
where $a:=\ka_1/\ka_2$ is the modulus of the slope of the line $L_\ka.$
\end{lemma}
\begin{proof}   The condition
$\ka_1B+\ka_2 A=1$ is equivalent to $A/a+B=1/\ka_1,$ hence to 
$$
\frac {A-1}a+B+2=\frac 1{\ka_1}-\frac {\ka_2}{\ka_1}+2=\frac {1-\ka_2}{\ka_1}+2=(n_\ka-2)+2=n_\ka.
$$
\end{proof}

\begin{figure}[h]
\centering
\includegraphics[scale=0.4]{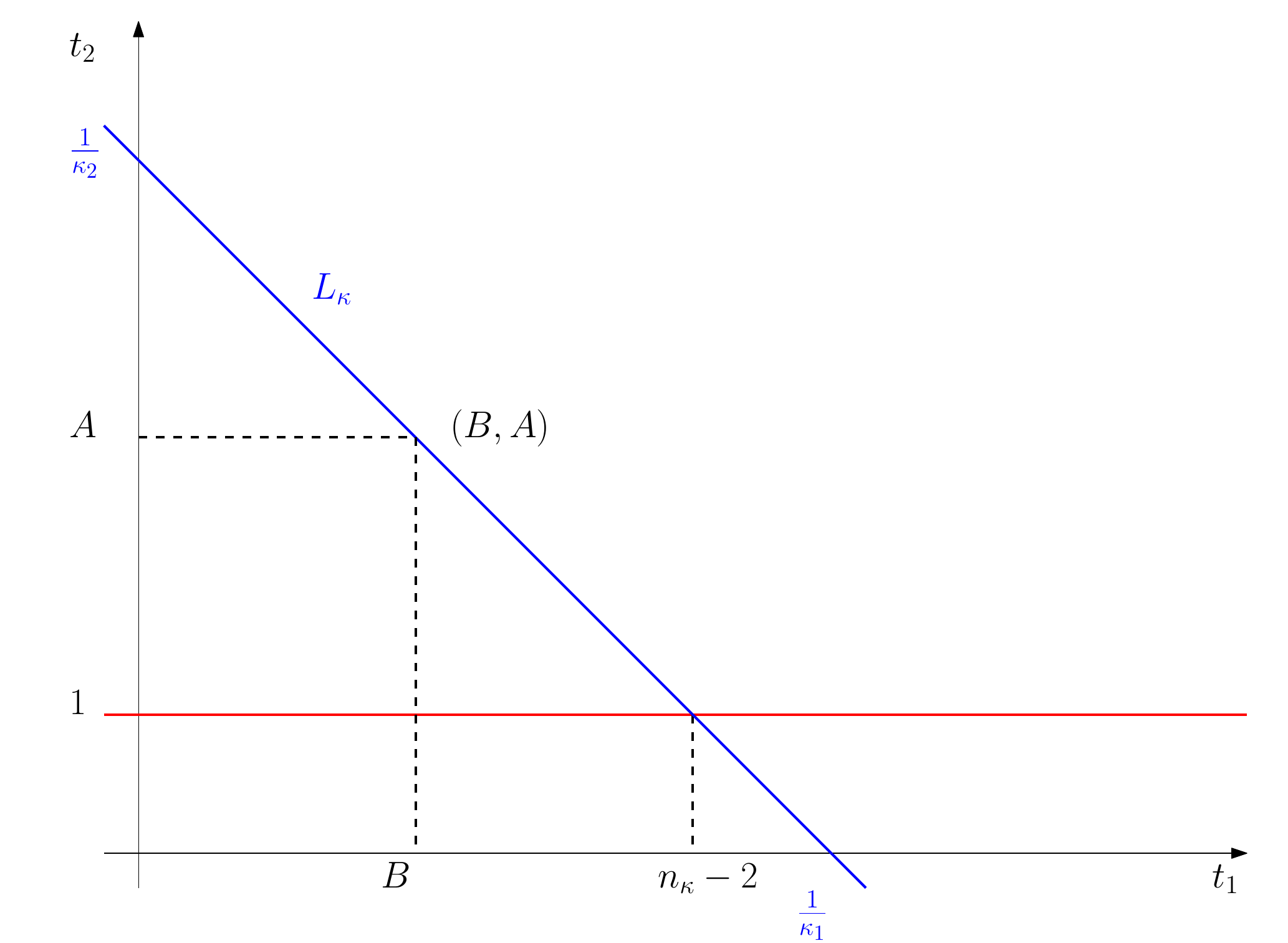}
  \caption{Meaning of $n_\ka$}
  \label{figure geo}
\end{figure}

The next lemma will allow us to show that the estimates in \eqref{suminla1} can be summed over the relevant $j$s and $k$s. Since in Step 2 and higher we shall have to deal with functions $\Phi(z,s_2)$ which are analytic in the first variable $z,$ but only fractionally analytic in the second variable, we shall state it more generally for Newton-Puiseux polyhedra.
\smallskip

So, assume that $\Phi$ is such a function, with Newton-Puiseux polyhedron $\N(\Phi),$ whose vertices are the points 
$(B_l,A_l), l=0,\dots L,$ where $(B_0,A_0)=(n-2,0),$ somewhat   as in Subsection \ref{roots} (compare Figure \ref{edges}). Note, however, that in contrast to the case of Newton polyhedra, where necessarily  $A_1\ge 1,$ we may possibly have $A_1<1.$ Otherwise, we adapt the same notation as in Subsection \ref{roots}.  

\begin{lemma}\label{key}
Assume that $A_1\ge 1,$ so that $A_l\ge 1$ for every $l\ge 1.$ Then  
\begin{equation*}
\frac{A_l-1}{a_l}+B_l+2\le \frac{A_1-1}{a_1}+B_1+2=\frac{A_0-1}{a_1}+B_0+2
\end{equation*}
for $l=1,\dots,L.$
\end{lemma}
The proof is geometrically evident in view of Lemma \ref{geo} (note that apparently the sequence $\{n_{\ka^l}\}_l$ is decreasing) and can easily by carried out by means of an  induction over $l\ge0.$ It will therefore be omitted.

\medskip
We are now in the position to show that the estimates in \eqref{suminla1} can be summed over all relevant $j$s and $k$s associated to a given transition domain $E_l, l=1,\dots, L,$ provided $p>\max\{3/2, h\}.$  Actually, we shall only need condition \eqref{Ejk1} here, i.e.,
$ j\ge k a_l+M/2.$   Indeed, we shall prove that under these assumptions 
\begin{equation}\label{suminjk1}
\sum\limits_{j,k\gg1:  j\ge k a_l}\Big((d_{j,k}^{-1}2^{-k})^{\frac 2p-1+\ve'}\, 2^{-j} 
+  (d_{j,k}^{-1})^{\frac 1p-\frac 12+\ve'}\, 2^{-j}\Big)<\infty,
\end{equation}
provided we choose $\ve'>0$ sufficiently small.

In combination with estimate \eqref{suminla1} this will clearly imply that for if $l=1,\dots,L,$
\begin{equation}\label{Lpest1}
\|\M^{\tau_l}\|_{p\to p} <\infty,
\end{equation}
provided $p>\max\{3/2, h\}.$

\medskip
Now, from \eqref{legendamins} and \eqref{Phi}, we find that the  first two vertices of $\N(\Phi)$ are given by $(B_0,A_0)=(n-2,0)$ and  $(B_1,A_1)=(m-2,1),$  and that $\ka_1^1=1/(n-2)>0$  and
\begin{equation*}
a_1=\frac {\ka_1^1}{\ka^1_2}=\frac 1{n-m}.
\end{equation*}
 (compare also Figure \ref{edges}). Thus, Lemma \ref{key} implies that
\begin{equation}\label{leqm}
\frac{A_l-1}{a_l}+B_l+2\le m
\end{equation}
for $l=1,\dots, L.$  Fix such an $l,$ and put $A:=A_l, B:= B_l$ and $a=a_l>a_1=1/(n-m)>0.$ Then 
$
d_{j,k}^{-1}= 2^{kA+j(B+2)}
$
and  $k\le j/a$ in \eqref{suminjk1}. Thus, we can re-write  \eqref{suminjk1} as 
\begin{equation}\label{crucsum}
\sum\limits_{j,k\gg1:  k\le j/a}\Big(( 2^{k(A-1)+j(B+2)})^{\frac 2p-1+\ve'}\, 2^{-j} 
+  ( 2^{kA+j(B+2)})^{\frac 1p-\frac 12+\ve'}\, 2^{-j}\Big)<\infty.
\end{equation}
Since the exponents $\frac 2p-1+\ve'$ and $\frac 1p-\frac 12+\ve'$ are strictly positive, we can first sum over all 
$k\le j/a$ and are left with showing that
\begin{equation*}
\sum\limits_{j\gg1}\Big( 2^{j\big[(\frac{A-1}a+B+2)(\frac 2p-1+\ve')-1\big]}
+  2^{j\big[\frac Aa+B+2)(\frac 1p-\frac 12+\ve')-1\big]}\Big)<\infty.
\end{equation*}

But, by \eqref{leqm}, 
$$
\frac{A-1}a+B+2\le m \quad \text{and} \quad \frac 2p-1<\frac 2n \le \frac 1m, 
$$ 
since $1/p<1/h=1/2+1/n$ and $n\ge 2m,$ so that we may assume that the exponent of the first summand in \eqref{crucsum} is strictly negative and we can sum in $j.$

Similarly, 
$$
\frac{A}a+B+2\le m +\frac 1a\le m+(n-m)=n\quad \text{and} \quad \frac 1p-\frac 12<\frac 1n, 
$$ 
so that we may assume that the exponent of the second summand in \eqref{crucsum} is strictly negative  too and we can again  sum in $j.$
This completes the proof of \eqref{suminjk1}.

\medskip
\subsection{$L^p$-estimation of $\M^{\tau_0}$}\label{Lpl0}
\medskip
Let us finally look at the case $l=0.$ Recall that here only the condition \eqref{Ejk2}, i.e., $j\le ka_1-M/2,$ is available, and that now
$$
d_{j,k}^{-1}= 2^{kA_0+j(B_0+2)}=2^{jn}.
$$
Thus, here we would like  to show that
\begin{equation}\label{suminjk2}
\sum\limits_{j,k\gg1:  k\ge j/a_1}\Big(( 2^{jn-k})^{\frac 2p-1+\ve'}\, 2^{-j} 
+  ( 2^{jn})^{\frac 1p-\frac 12+\ve'}\, 2^{-j}\Big)<\infty.
\end{equation}
In the first summand, we can first sum over all $k\ge j/a_1$ and arrive at the sum
$$
\sum\limits_{j\gg 1} 2^{j\big[(n-\frac 1{a_1})(\frac 2p-1+\ve')-1\big]}.
$$
But, 
$$
n-\frac 1{a_1}=m, \quad \text{and} \quad \frac 2p-1<\frac 2n \le \frac 1m,
$$
so that we may assume that the exponent of the first summand in \eqref{suminjk2} is strictly negative by choosing $\ve'$ sufficiently small,  and  thus we can sum in $j.$

The second summand is insufficient for a summation in $k,$ unless, say,  $k\lesssim j.$  In the latter case, we can sum in $k\lesssim j$ and  arrive at the sum 
\begin{equation}\label{suminj1}
\sum\limits_{j\gg 1}2^{jn(\frac 1p-\frac 12+\ve')}\, 2^{-j},
\end{equation}
with any slightly bigger $\ve'>0$ than before. This series is  convergent for $\ve'$ sufficiently small, since we are assuming that  $p>h.$ 

\medskip

So, let us assume henceforth that, say, $k> jn.$  
The crucial observation  is that  the second term in \eqref{Mjk}  can be improved when  $\la2^{-k}\le 1,$ and this will be needed for the case $l=0.$
\smallskip

To be more precise, assume first that  $\la2^{-k}>1,$ i.e., $k\lesssim\log \la.$ Then we can keep the second term in \eqref{Mlajkp}, which is here given by 
$$
 \la^{\ve'} (\la 2^{-jn})^{\frac 1p -\frac 12}(1+\la 2^{-jn})^{\frac 1p-1} \, (2^{jn})^{\frac 1p-\frac 12}\, 2^{-j}.
$$
Summing now first in $k\lesssim\log \la,$ we can estimate by the same kind of expression, but with a slightly bigger exponent $\ve'>0.$ 
After summing in $\la,$ we are again left with the series in \eqref{suminj1} and are done.

\medskip
 Assume next that  $\la2^{-k}\le 1.$ Then the exponential factor of the oscillatory integral $\mu_{j,k}^\la(y+\Gamma)$ which collects all terms depending  only on $\si_2$  (compare \eqref{onlys2}) is essentially non-oscillatory, and so we  shall rather estimate for $y\in \Delta_{\iota_1}$ by 
$$
|\mu_{j,k,\iota_1}^\la(y+\Ga)|\lesssim \la^{\frac 52} 2^{-j-k} (\la 2^{-j} 2^{\iota_1})^{-1}=\la^{\frac 32} 2^{-k-\iota_1}.
$$
Since $|\Delta_{\iota_1}|\lesssim 2^{\iota_1-j},$ we obtain the improved estimate
$$
\|\M_{j,k,\iota_1}^\la\|_{1+\ve\to 1+\ve}\lesssim \la^{\frac 32}2^{-j-k}.
$$
This shows that we can replace the second term in \eqref{Mjk} by $\la^{\frac 32+\ve}2^{-j-k},$ and finally, after interpolation with our $L^2$-estimate \eqref{L2Eljk}, we are led to summing 
$$
M(\la,j,k):=(\la^{\frac 32+\ve}2^{-k})^{\frac 2p -1}(1+\la 2^{-jn})^{\frac 1p-1}2^{-j}
$$
over our $\la$s, $j$s and $k$s.
\smallskip

Assume first that $2^k\ge \la>2^{jn}.$ Then 
$
M(\la,j,k)\sim(\la^{\frac 32+\ve}2^{-k})^{\frac 2p -1}(\la 2^{-jn})^{\frac 1p-1}2^{-j},
$
hence 
\begin{eqnarray*}
\sum\limits_{k:2^k\ge \la} M(\la,j,k)&\lesssim&  \la^{\frac 1p -\frac 12+\ve'}(\la 2^{-jn})^{\frac 1p-1}2^{-j}=(\la 2^{-jn})^{\frac 1p-\frac 12+\frac 1p-1+\ve'} 2^{jn(\frac 1p-\frac 12+\ve')} 2^{-j}\\
&=&(\la 2^{-jn})^{\frac 2p-\frac 32+\ve'} 2^{j\big[n(\frac 1p-\frac 12+\ve')-1\big]}.
\end{eqnarray*}
But, since we are assuming that $p>\max\{3/2, h\},$ we have $ 2/p- 3/2=-1/6<0,$ and again $n(1/p- 1/2)-1<0,$ 
so that we can first  sum in $\la>2^{jn},$ and then in $j\gg1,$ provided $\ve'>0$ is sufficiently small.
\medskip

We are thus left with the case where $ \la\le 2^{jn}$ and $k>jn.$ Then 
$
M(\la,j,k)\sim(\la^{\frac 32+\ve}2^{-k})^{\frac 2p -1}2^{-j},
$
and thus 
$$
\sum\limits_{\la:\la\le 2^{jn}} M(\la,j,k)\lesssim  (2^{jn(\frac 32+\ve)}2^{-k})^{\frac 2p -1}2^{-j}.
$$
This implies
$$
\sum\limits_{k: k\ge jn}\sum\limits_{\la:\la\le 2^{jn}} M(\la,j,k)\lesssim  (2^{jn(\frac 12+\ve)})^{\frac 2p -1}2^{-j}.
$$
and again we can finally also sum in $j,$ since $p>h.$
\smallskip

Consequently, estimate \eqref{Lpest1} holds true even for $l=0.$

\medskip
\subsection{On the contributions by the homogeneous domains $D_l, l=1,\dots,L$}\label{Dlcont}
\medskip
Let us fix $l_1\in \{1,\dots,L\}.$ We recall that 
$$
D_{l_1}:=\{(z_1,s_2):   2^{-M} s_2^{a_{l_1}}\le |z_1|\le 2^{M} s_2^{a_{l_1}}\}, $$
where we may assume that $\rho_{l_1}$ is supported in $D_{l_1}.$ We also recall from Subsection \ref{dyadic} that we may assume that
$$
J^{\rho_{l_1}}(\la,s)=\sum\limits_{j,k: |j-k a_{l_1}|\le M/2}J_{j,k}(\la,s).
$$

We shall decompose  $D_{l_1}$ into a finite number of very narrow subregions. To this end, fix any point $c$  in the compact interval  $I:=[2^{-M}, 2^M],$ and put
$$
D_{l_1}^c:=\{(z_1,s_2):   |z_1-cs_2^{a_{l_1}}|<  \ve s_2^{a_{l_1}}\}, 
$$
where $\ve>0$ is supposed to be sufficiently small. More precisely, recall from Subsection \ref{roots} that every root $r(s_2)$ in the cluster 
$\left[ \begin{matrix} 
\cdot\\
l_1
\end{matrix}\right ]$
associated to the compact edge $\ga_{l_1}$ of $\N(\Phi)$ admits a Puiseux series expansion
$$
r(s_2)=c_{l_1}^{\al_1}s_2^{a_{l_1}}+c_{l_1l_2}^{\al_1\al_2}s_2^{a_{l_1l_2}^{\al_1}}+\cdots+
c_{l_1\cdots l_p}^{\al_1\cdots \al_p}s_2^{a_{l_1\cdots
l_p}^{\al_1\cdots \al_{p-1}}}+\cdots,
$$
with leading term $c_{l_1}^{\al_1}s_2^{a_{l_1}},$ where the leading coefficient $c_{l_1}^{\al_1}$ may be real, or complex.
By having chosen $M$ to be sufficiently large, we are of course assuming that  all real coefficients  $c_{l_1}^{\al_1}$ will be contained in the interval $I.$ Denote by  $\mathcal C_{l_1}$ the set of those {\it real} coefficients $c_{l_1}^{\al_1}.$ 

If $c\notin \mathcal C_{l_1},$ we choose $\ve>0$ at least so small that $\ve\ll\dist(c,\mathcal C_{l_1}).$


\begin{figure}[!h]
\centering
\includegraphics[scale=0.45]{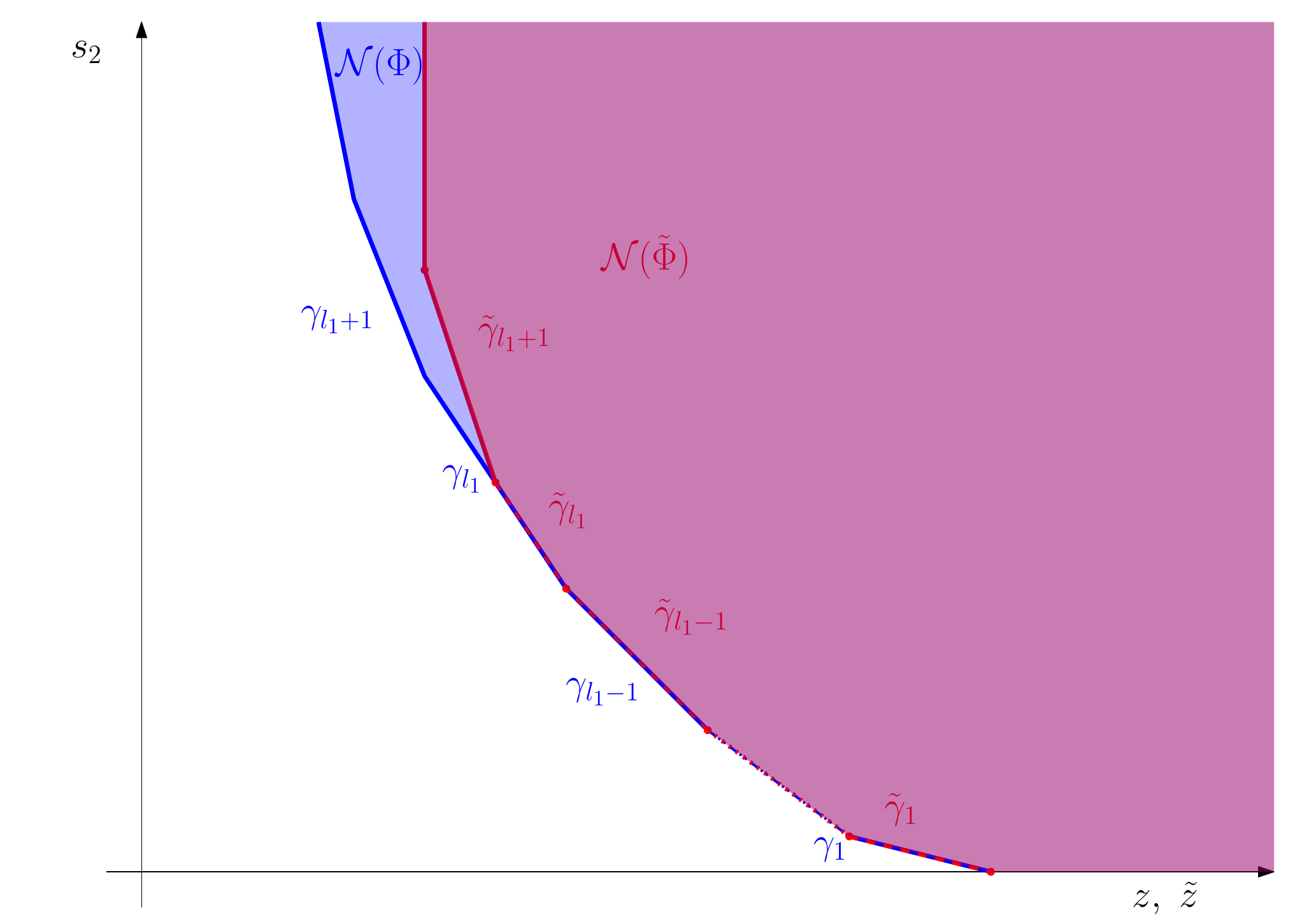}
  \caption{Comparison between $\N(\Phi)$ and $\N(\tilde\Phi)$}
  \label{Phit}
\end{figure}

Now, for any $c\in I,$ let us introduce the new variable $\tilde z:=z_1-cs_2^{a_{l_1}},$ and accordingly define 
\begin{equation*}
\tilde\Phi(\tilde z,s_2):= \Phi(\tilde z+cs_2^{a_{l_1}}, s_2).
\end{equation*}
Note that $\tilde \Phi$ is still analytic in $\tilde z,$ but may be only fractionally analytic in $s_2.$ The region $D_{l_1}^c$ then corresponds to 
$$
\tilde D_{l_1}^c:=\{(\tilde z, s_2):   |\tilde z|<  \ve s_2^{a_{l_1}}\}.
$$

In Step 2 of our resolution algorithm,  which will be discussed in the next section, we shall base our analysis on the Newton polyhedron of $\tilde\Phi$  in place of $\N(\Phi).$ The corresponding  vertices and compact edges of  $\N(\tilde\Phi)$ will  be denoted by $(\tilde B_l, \tilde A_l), l=0,\dots, \tilde L$ and  $\tilde \ga_l:=(\tilde B_{l},\tilde A_{l}),(\tilde B_{l-1},\tilde A_{l-1})], l=1,\dots, \tilde L.$
The  line supporting the edge $\tilde \ga_l$ will be denoted by $\tilde L_l,$ the associated weight by 
$\tilde \ka^l,$ and the corresponding modulus of slope  by $\tilde a_l= \tilde \ka_1^l/\tilde \ka_2^l.$ The corresponding homogeneous domains  and transition domains into which we shall decompose the upper half-plane where $\tilde z\in \RR, s_2>0,$ will be denoted by $\tilde D_l, l=1,\dots, \tilde L,$  and $\tilde E_l, l=0,\dots, \tilde L.$

\medskip

Coming back to $\tilde D_{l_1}^c,$  note that  by \eqref{Phifac},

 \begin{equation}\label{tildePhifac}
\tilde\Phi(\tilde z,s_2)=\tilde U_0(\tilde z,s_2)(\tilde z+cs_2^{a_{l_1}})^{\nu_1}  \prod_{l=1}^L
\Phi\left[ \begin{matrix} 
\cdot\\
l\end{matrix}\right ](\tilde z+cs_2^{a_{l_1}},s_2),
\end{equation}
where
$$
\Phi\left[ \begin{matrix} 
\cdot\\
l
\end{matrix}\right ](\tilde z+cs_2^{a_{l_1}},s_2)
=\prod_{r\in \left[ \begin{matrix} 
\cdot\\
l
\end{matrix}\right ]}(\tilde z+cs_2^{a_{l_1}}-r(s_2))=\prod_{r\in \left[ \begin{matrix} 
\cdot\\
l
\end{matrix}\right ]}(\tilde z-\tilde r(s_2)),
$$
if we define adjusted  roots $\tilde r$  by
$$
\tilde r(s_2):=r(s_2)-cs_2^{a_{l_1}}.
$$

If now $l<l_1$  and 
$r\in \left[ \begin{matrix} 
\cdot\\
l
\end{matrix}\right ], $
then $a_l<a_{l_1},$ so the roots $\tilde r(s_2)$ and $r(s_2)$ do have the same leading term and the same multiplicity. 

Thus the identities \eqref{BandA} and \eqref{Blall} show that we indeed have 
\begin{equation*}
(\tilde B_l, \tilde A_l)=( B_l, A_l), \quad 0\le l<l_1.
\end{equation*}
In particular,
$$
\tilde \ga_l=\ga_l \quad \text{and}  \quad \tilde a_l=a_l, \  l< l_1
$$
(compare Figure \ref{Phit}).
\medskip

\noi {\bf  Case 1:} $c\notin \mathcal C_{l_1}.$ Then, if $l=l_1$ and 
$r\in \left[ \begin{matrix} 
\cdot\\
l
\end{matrix}\right ], $
we can factor 
$$
\tilde z-\tilde r(s_2)=s_2^{a_{l_1}}\Big(c-c_{l_1}^{\al_1}+\frac {\tilde z}{s_2^{a_{l_1}}}\Big)+\text{higher degree terms},
$$
where $|{\tilde z}/{s_2^{a_{l_1}}}|<\ve.$ Moreover, if $c_{l_1}^{\al_1}$ is real, then $|c-c_{l_1}^{\al_1}|\gg\ve,$ and otherwise
$|c-c_{l_1}^{\al_1}|\ge |\Im  c_{l_1}^{\al_1}|>0,$ so that, by choosing $\ve $ small enough, we can assume that in both cases
$|c-c_{l_1}^{\al_1}|\gg\ve.$ Thus we may assume that
$
|c-c_{l_1}^{\al_1}+ {\tilde z}/{s_2^{a_{l_1}}}|\sim 1.
$

Arguing in a similar way also for $l\ne l_1$ (compare our discussion in Subsection \ref{Elresolution}), we find that the edge 
$\tilde\ga_{l_1}$  has the same slope  and right endpoint as $\ga_{l_1},$  but as left endpoint a point of the form 
$(0,\tilde A_{l_1}),$ and that we can factor 
\begin{equation*}
\tilde\Phi(\tilde z,s_2)=V(\tilde z,s_2)\, s_2^{\tilde A_{l_1}},\qquad (\tilde z,s_2)\in \tilde D_{l_1}^c,
\end{equation*}
where  the function $V(\tilde z,s_2)$ behaves in a similar way is in Subsection \ref{Elresolution}. In particular, here $\tilde L=l_1.$ 
\medskip

\begin{remark}\label{cnotinC}
This shows that, by changing from $z_1$ to $\tilde z,$ we can regard the  region $\tilde D_{l_1}^c$  now like a transition domain $\tilde E_{\tilde L},$ in the coordinates $(\tilde z, s_2).$  With view to Lemma \ref{key}, we also remark that here $\tilde A_{l_1}\ge A_{l_1}\ge A_1=1.$
These kind of  domains will be  handled in Step 2 of our resolution algorithm.
\end{remark}
\medskip

\noi {\bf Case 2:} $c\in \mathcal C_{l_1}.$ 
Then $c$ is of the form $c=c_{l_1}^{\al_1}\in\RR.$  This will fix $\al_1$ in the sequel. We then recall from \eqref{Phiprinc} that, up to a constant factor,  the $\ka^{l_1}$-principal part of $\Phi$ is given by
$$
\Phi_{\ka^{l_1} }(z_1,s_2)= s_2^{A_{l_1-1}} z_1^{B_{l_1}}\prod_\al \Big(z_1-c^\al_{l_1} s_2^{a_{l_1}}\Big)^{N\left[ \begin{matrix} 
\al\\
l_1
\end{matrix}\right ] }.
$$
Consequently, the   $\ka^{l_1}$-principal part of $\tilde \Phi$ is given by
$$
\tilde\Phi_{\ka^{l_1} }(\tilde z,s_2)= s_2^{A_{l_1-1}} (\tilde z+cs_2^{a_{l_1}})^{B_{l_1}}\prod_\al \Big(\tilde z-(c^\al_{l_1}-c_{l_1}^{\al_1})\ s_2^{a_{l_1}}\Big)^{N\left[ \begin{matrix} 
\al\\
l_1
\end{matrix}\right ] }.
$$
We must therefore distinguish two sub-cases:

\smallskip

\noi {\bf Subcase (i):} 
$B_{l_1}>0, $ or 
$\left[ \begin{matrix} 
\cdot\\
l_1
\end{matrix}\right ] 
\supsetneqq\left[ \begin{matrix} 
\al_1\\
l_1
\end{matrix}\right ]. $
 Then $\tilde\Phi_{\ka^{l_1}}$ is not a monomial, and thus the edge $\tilde \ga_{l_1}$  lies on the same line as $ \ga_{l_1}.$ It thus will have the same slope  $\tilde a_{l_1}= a_{l_1}$ and right endpoint $(B_{l_1-1},A_{l_1-1})$ as $\ga_{l_1},$ but the left endpoints of these two edges may be different.

\smallskip
Here
$$
\tilde E_{l_1}=\{(\tilde z,s_2):   2^M s_2^{\tilde a_{l_1+1}}<|\tilde z|<2^{-M} s_2^{a_{l_1}}\}.
$$
Thus, if we choose $\ve<2^{-M},$ and decompose 
$$
\tilde D_{l_1}^c=\tilde D_{l_1,1}^c \cup \tilde D_{l_1,2}^c,
$$
where
\begin{eqnarray*}
\tilde D_{l_1,1}^c&:=&\{(\tilde z,s_2):   2^M s_2^{\tilde a_{l_1+1}}\le |\tilde z|<\ve s_2^{a_{l_1}}\},\\
\tilde D_{l_1,2}^c&:=&\{(\tilde z,s_2):  |\tilde z|<2^Ms_2^{a_{l_1+1}}\},
\end{eqnarray*}
then $\tilde D_{l_1,1}^c\subset \tilde E_{l_1}.$ Note that for  $l_1=\tilde L,$ we even have $\tilde D_{l_1}^c\subset \tilde E_{l_1}.$

\smallskip
Moreover, clearly $\tilde D_{l_1,2}^c$ is contained in the union of the $\tilde E_l$s  and $\tilde D_l$s  with $l\ge l_1+1,$ and therefore
\begin{equation*}
\tilde D_{l_1}^c\subset \bigcup\limits_{l=l_1}^{\tilde L} \tilde E_l\, \cup \bigcup\limits_{l=l_1+1}^{\tilde L}\tilde D_l.
\end{equation*}

\begin{remarks}\label{cinCi}  (a) We thus see that in Subcase (i) the contribution by $D_{l_1}^c$ to our maximal operator can be covered by the contributions by the 
$\tilde E_l$s  with $l\ge l_1$ and the $\tilde D_l$s  with $l\ge l_1+1.$ Those by the $ \tilde E_l$s will be completely covered  by  Step 2  of our resolution algorithm in the next section, and to the $\tilde D_l$s we shall apply a similar kind of reasoning as in the present section,  which will lead to Step 3 of our resolution algorithm, and so forth. Note also that $l_1+1\ge 2,$ so that the first edge with right endpoint $(n-2,0)$  will be the same for all steps $\ge 2,$ namely $\tilde \ga_1.$  Thus, if $\tilde A_1\ge 1,$ then we can apply Lemma \ref{key} in Step 2, and as well in any higher step.

(b) Note also that  if $\tilde a_l$  with $l\ge l_1+1$ is the leading exponent of a root $\tilde r(s_2)$ of  $\tilde\Phi,$  then $\tilde a_l$ is of the form $a_{l_1l_2}^{\al_1},$ and the leading term of $\tilde r(s_2)$ is of the form
$
c_{l_1l_2}^{\al_1\al_2}s_2^{a_{l_1l_2}^{\al_1}}.
$

\end{remarks}
\medskip

\noi {\bf Subcase (ii):}  $B_{l_1}=0$ and 
$\left[ \begin{matrix} 
\cdot\\
l_1
\end{matrix}\right ] 
=\left[ \begin{matrix} 
\al_1\\
l_1
\end{matrix}\right ], $
i.e., all roots in $\left[ \begin{matrix} 
\cdot\\
l_1
\end{matrix}\right ]$  have the same leading term $cs_2^{a_{l_1}}.$  Then 
$\tilde\Phi_{\ka^{l_1} }(\tilde z,s_2)= s_2^{A_{l_1-1}} 
\tilde z^{N\left[ \begin{matrix} 
\cdot\\
l_1
\end{matrix}\right ] }
=s_2^{A_{l_1-1}} \tilde z^{B_{l_1-1}} 
$
is a monomial, so that the edge $\tilde \ga_{l_1}$ has ``shrunk''  to the vertex $(B_{l_1-1},A_{l_1-1})$ of $\N(\tilde\Phi)$ by the change of coordinates from $z_1$ to $\tilde z.$ Here,  the edge $\tilde \ga_{l_1}$ will thus  have a slope $\tilde a_{l_1}> a_{l_1}.$

\smallskip
Here
$$
\tilde E_{l_1-1}=\{(\tilde z,s_2):   2^M s_2^{\tilde a_{l_1}}<|\tilde z|<2^{-M} s_2^{a_{l_1-1}}\}, 
$$
since $\tilde a_{l_1-1}=a_{l_1-1}.$ And, we claim that we may assume that 
$$
2^M s_2^{\tilde a_{l_1}}< \ve s_2^{a_{l_1}}\le 2^{-M} s_2^{a_{l_1-1}}.
$$
Indeed, the first inequality is equivalent to $ s_2^{\tilde a_{l_1}-a_{l_1}}<\ve 2^{-M}.$ But, $\tilde a_{l_1}-a_{l_1}>0,$ so that the latter inequality  can be assumed to be true  because  we may assume  that  $s_2\ll1 $ is  sufficiently small.  The second inequality  follows in a similar way. Let us here decompose 
$
\tilde D_{l_1}^c=\tilde D_{l_1,1}^c \cup \tilde D_{l_1,2}^c,
$
where now
\begin{eqnarray*}
\tilde D_{l_1,1}^c&:=&\{(\tilde z,s_2):   2^M s_2^{\tilde a_{l_1}}\le |\tilde z|<\ve s_2^{a_{l_1}}\},\\
\tilde D_{l_1,2}^c&:=&\{(\tilde z,s_2):  |\tilde z|<2^Ms_2^{\tilde a_{l_1}}\}.
\end{eqnarray*}
Then $\tilde D_{l_1,1}^c\subset \tilde E_{l_1-1},$  in a similar way as in the previous subcase we see that here
\begin{equation}\label{coverDl122}
\tilde D_{l_1}^c\subset \bigcup\limits_{l=l_1-1}^{\tilde L} \tilde E_l\, \cup \bigcup\limits_{l=l_1}^{\tilde L}\tilde D_l.
\end{equation}

\medskip
\begin{remarks}\label{cinCii}
(a) If $l_1=1,$ then for our phase $\Phi$ in \eqref{Phi}, by \eqref{legendamins} we have
\begin{eqnarray}\label{l=1}
\begin{split}
\ka^1&=\big(\frac 1{n-2},\frac {n-m}{n-2}\big), \quad a_1=\frac 1{n-m},\\
\Phi_{\ka^{1} }(z_1,s_2)&= z_1^{n-2} n(n-1)\beta(0)+s_2 z_1^{m-2} m(m-1)\omega(0).
\end{split}
\end{eqnarray}
Note that here $B_{l_1}=m-2.$  Moreover, since we are assuming that $n\ge 2m,$ we have $n-m\ge m\ge 2,$ so that the polynomial $z_1^{n-m} n(n-1)\beta(0)+s_2  m(m-1)\omega(0)$ has at least 2 different (real or complex) roots, each of multiplicity 1, hence 
$\left[ \begin{matrix} 
\cdot\\
1
\end{matrix}\right ] 
\supsetneqq\left[ \begin{matrix} 
\al_1\\
1
\end{matrix}\right ]. $ 
This shows that here  only Subcase (i) can arise.  

In particular, in Step 2, Subcase (ii) can only arise when $l_1\ge 2,$ and the same applies then to all higher steps, so that $l_1-1\ge 1$ in \eqref{coverDl122} -- thus  no transition domains $\tilde E_0$ will appear in \eqref{coverDl122}.
\smallskip

(b)  By \eqref{coverDl122} we see that  in Subcase (ii) the contribution by $D_{l_1}^c$ to our maximal operator can be covered by the contributions by the 
$\tilde E_l$s  with $l\ge l_1-1\ge 1$ and the $\tilde D_l$s  with $l\ge l_1\ge 2.$ Those by the $ \tilde E_l$s will be completely covered  by  Step 2  of our resolution algorithm in the next section, and to the $\tilde D_l$s we shall apply a similar kind of reasoning as in the present section,  which will lead to Step 3 of our resolution algorithm, and so forth. Note again that since here $l_1\ge 2,$ the first edge with right endpoint $(n-2,0)$  will be the same for all steps $\ge 2,$ namely $\tilde \ga_1.$  Thus, if $\tilde A_1\ge 1,$ then we can apply Lemma \ref{key} in Step 2, and as well in any higher step.
\smallskip

(c)  Note also that  if $\tilde a_l$  with $l\ge l_1$ is the leading exponent of a root $\tilde r(s_2)$ of  $\tilde\Phi,$  then $\tilde a_l$ is of the form $a_{l_1l_2}^{\al_1},$ and the leading term of $\tilde r(s_2)$ is of the form
$
c_{l_1l_2}^{\al_1\al_2}s_2^{a_{l_1l_2}^{\al_1}}.
$
\end{remarks}

\medskip

We finally remark that, by compactness of the interval $I,$ we can cover  the homogeneous domain $D_{l_1}$ by the sub-domains $D_{l_1}^c,c\in \mathcal C_{l_1},$ and a finite number of subdomains $D_{l_1}^{c^j},$ with $c^j\notin \mathcal C_{l_1}.$

Thus we may reduce our studies to considering such sub-domains $D_{l_1}^c,$ for which we shall then, as explained before,  pass to the coordinate $\tilde z$ in place of $z_1.$  The study of the contributions of the corresponding homogeneous domains $\tilde D_{l_1}^c$ will be carried out in Step 2  (and higher) of our algorithm.

\medskip
\subsection{Stopping domains $D^c_{l_1}, c\in  \mathcal C_{l_1},$ for the resolution algorithm at Step 1}\label{stopstep1}
\medskip
We shall stop on the domain $D^c_{l_1}, c=c_{l_1}^{\al_1}\in  \mathcal C_{l_1},$ when the multiplicity 
$N\left[ \begin{matrix} 
\al_1\\
l_1
\end{matrix}\right ]
$
of the cluster 
$\left[ \begin{matrix} 
\al_1\\
l_1
\end{matrix}\right ]
$
agrees  with the multiplicity of every sub-cluster of it.

Observe that this is equivalent to the condition that the cluster 
$\left[ \begin{matrix} 
\al_1\\
l_1
\end{matrix}\right ]
$
contains only one, real root $r$ of multiplicity 
$N\left[ \begin{matrix} 
\al_1\\
l_1
\end{matrix}\right ].
$
The fact that this root must be real stems from the observation that with any root in 
$\left[ \begin{matrix} 
\al_1\\
l_1
\end{matrix}\right ]
$,
 its complex conjugate must also be in the same cluster, since $\Phi$ is real.
 
 The contribution by such a domain  $D^{c}_{l_1}$ will be covered in a complete way in  Subsection \ref{D1c} of the next section.

\begin{remark}[Stopping when $l_1=1$]\label{stopl1}
If $l_1=1,$ we have seen in Remark \ref{cinCii} (a) that any real root in 
$\left[ \begin{matrix} 
\al_1\\
1
\end{matrix}\right ]
$
has multiplicity $1,$ so that by \eqref{Phiprinc} we do have 
$N\left[ \begin{matrix} 
\al_1\\
1
\end{matrix}\right ]=1.
$
This shows that for $l_1=1$ we stop our resolution algorithm in Step 1.
\end{remark}

\medskip

\setcounter{equation}{0}

\section{ Analysis in Step 2 of the resolution algorithm  }\label{step2}

In this section, we shall proceed in a similar way as in the previous Section \ref{step1}, but our analysis will now be based on the coordinates 
$\tilde z$  and the function $\tilde \Phi$ in place of $z_1$ and the function $\Phi.$  Adapting the notation from the previous subsection, recall that 
\begin{equation*}
z_1=\tilde z+ w(s_2),
\end{equation*}
where here $w(s_2)=cs_2^{a_{l_1}}.$  The non-trivial extra term $w(s_2),$ however, will create additional challenges compared to the analysis in the previous section. 

Of course, we shall have to apply the same change of coordinates also to our original  phase functions $\breve\phi_1(x_1, s_2)$ and $\Phi_0(x_1,s'):=\breve\phi_1(x_1, s_2)+s_1x_1$ in \eqref{Phi0}, and therefore put
\begin{eqnarray*}
\tilde\Phi_{0,0}(\tilde z, s_2)&:=& \breve\phi_1(\tilde z+ w(s_2), s_2),\\
\tilde\Phi_{0}(\tilde z, s')&:=& \Phi_{0}(\tilde z+ w(s_2),s_2)=\tilde\Phi_{0,0}(\tilde z, s_2)+s_1\tilde z+s_1w(s_2).
\end{eqnarray*}
Note that then 
\begin{equation}\label{tPhi}
\tilde\Phi(\tilde z,s_2):=\pa_{\tilde z}^2\tilde\Phi_{0}(\tilde z, s') =\pa_{\tilde z}^2\tilde\Phi_{0,0}(\tilde z, s_2)
\end{equation}
(compare with \eqref{Phi}).

A crucial tool will again be Lemma \ref{key}. Its application will here require that $\tilde A_1\ge 1,$ so let us begin by discussing this condition.

\subsection{On the condition $\tilde A_1\ge 1$}\label{A1ge1}
\medskip
By Remark  \ref{cnotinC}, we see that  $\tilde A_1\ge 1$ in Case 1, i.e., when $c\notin \mathcal C_{l_1}.$ Assume next that we are in  Case 2 where $c=c_{l_1}^{\al_1}\in \mathcal C_{l_1}.$ Then, if $l_1\ge 2,$ we have that $\tilde A_1=A_1\ge 1.$

But, in Remark \ref{stopl1} we have seen that for $l_1=1$ we can stop our resolution algorithm already at Step 1, so that henceforth we may and shall assume that $l_1\ge 2,$   hence $\tilde A_1=A_1\ge 1.$

\medskip
For the sake of completeness, let us finally observe that when $l_1=1,$ then by   Remark \ref{cinCii} (a)   we must be in Subcase (i), and by \eqref{l=1} we have 
$$
\Phi_{\ka^{1}}(z_1,s_2)= C_1z_1^{m-2} \big(z_1^{n-m}+ C_2 s_2 \big)
=C_1z_1^{m-2}\prod_\al \Big(z_1-c^\al_{l_1} s_2^{a_{1}}\Big),
$$
with  nontrivial real constants  $C_1,C_2.$ This implies that
$$
\tilde\Phi_{\ka^{1} }(\tilde z,s_2)=C_1\tilde z\, (\tilde z+c_{l_1}^{\al_1}s_2^{a_1})^{m-2} \prod_{\al\ne \al_1} \Big(\tilde z -(c^\al_{l_1}-c_{l_1}^{\al_1})\ s_2^{a_{1}}\Big),
$$
where the product over the $\al$s such that $\al\ne \al_1$ will consist of $n-m-1\ge 1$ different factors, since all non-trivial roots  of 
$\Phi_{\ka^{1}}$  have multiplicity 1. 
By \eqref{BandA} this implies that 
$\tilde A_1=(n-3)a_1=(n-3)/(n-m),$  and $\tilde B_1=1,$ so that $\tilde \ga_1=[(1,(n-3)/(n-m)),(n-2,0)].$ 

Thus, $\tilde A_1<1$ if and only if  $m=2 .$

\smallskip

\begin{remarks}\label{onstop}
(a) Notice that  by Remark \ref{stopl1}  and our discussion in Subsection \ref{D1c} we see that   we shall have dealt in a complete way with the contribution by the homogeneous domain $D_1$ and therefore will be allowed to assume that $l_1\ge 2$ in Step 3 and higher.
\smallskip

(b) We shall see later (see Section \ref{endproof1})  that our resolution algorithm will stop whenever we have arrived at a narrow subdomain  containing only one  sub-cluster of roots consisting of just  one  root   (possibly with multiplicity) whose leading jet $w(s_2)$ is real,  i.e., when no further ``fine splitting'' of roots can occur. In such a  case 
 we shall be able to argue in a very similar way as we are doing here. 
\end{remarks}

Recall also that since the edges $\tilde\ga_1$ and $\ga_1$ are lying on the same line $L_1,$ we have $\tilde a_1=a_1,$ and 
\begin{equation}\label{keyestl=1}
\frac{\tilde A_1-1}{\tilde a_1}+\tilde B_1+2=m.
\end{equation}
For later use, note that this identity remains valid even  when $\tilde A_1<1.$ And,  if $m\ge 3,$  so that  $\tilde A_1\ge 1,$ then  by Lemma \ref{key} we also have
\begin{equation}\label{keyestlg1}
\frac{\tilde A_l-1}{\tilde a_l}+\tilde B_l+2\le m,\quad  l=1,\dots, \tilde L.
\end{equation}
Clearly, this holds true also for $l_1\ge 2.$

\medskip

\medskip
\subsection{Domain decomposition and dyadic localization in $\tilde z$ and $s_2$}\label{dyadictildez}
\medskip
After having fixed $l_1\in\{1,\dots,L\},$  according to our discussion in the previous section we can cover the domain $\tilde D_{l_1}^c$ by $\tilde \ka^l$- homogeneous  domains  $\tilde D_l$ of the form
$$
 \tilde D_l:=\{(\tilde z,s_2):   2^{-M} s_2^{\tilde a_l}\le |\tilde z|\le 2^{M} s_2^{\tilde a_l}\}, \quad l\le \tilde L,
$$
 and  transition domains of the form
\begin{eqnarray*}
\tilde E_l&:=&\{(\tilde z,s_2):   2^M s_2^{\tilde a_{l+1}}<|\tilde z|<2^{-M} s_2^{\tilde a_l}\}, \quad l\le \tilde L-1,\\
\tilde E_{\tilde L}&:=&\{(\tilde z,s_2):  |\tilde z|<2^{-M} s_2^{\tilde a_{\tilde L}}\} .
\end{eqnarray*}

More precisely, in Case 1 where $c\notin \mathcal C_{l_1}, $ according to Remark \ref{cnotinC} we just need the transition domain $\tilde E_{\tilde L},$ where $\tilde L=l_1.$  

In Case 2, where $c\in \mathcal C_{l_1},$ we need to distinguish two subcases: if we are in  Subcase (i), then   
according to  Remark \ref{cinCi}   (a) we shall need the transition domains $\tilde E_l$ with $l\ge l_1\ge 1$ and the homogeneous domains  $\tilde D_l$ with $l\ge l_1+1\ge 2.$ Moreover, according to Remark \ref{onstop} (a), we shall treat  the case $l_1=1$ 
separately (in Subsection \ref{D1c}), so that we shall here even assume that $l_1\ge 2.$

In Subcase (ii) of Case 2, according to Remark \ref{cinCii} we may assume that $l_1\ge 2$ and shall then need the transition domains $\tilde E_l$ with $l\ge l_1-1\ge 1$ and the homogeneous domains  $\tilde D_l$ with $l\ge l_1\ge 2.$

\smallskip

Moreover, our discussion in Subsection \ref{A1ge1} shows that under these assumptions we always have that  $\tilde A_1\ge 1.$ 

Recall also that in all cases that had appeared we had seen that the edges $\tilde \ga_1$ and $\ga_1$ were supported on the same line $\tilde L_1=L_1$ of ``slope'' $\tilde a_1=a_1,$ so that  the identity \eqref{keyestl=1} and  the inequalities \eqref{keyestlg1} will hold true. 


Having fixed $l_1$ and $l$  according to these assumptions, we shall put 
$$
a:= a_{l_1}, \quad \tilde a:= \tilde a_l.
$$
Note that if we are in Case 1, or Subcase (i) of Case 2, then $\tilde a\ge \tilde a_{l_1}=a,$ but in Subcase (ii) it may also happen that $\tilde a<a$ (for instance if $l=l_1-1,$ where $\tilde a=a_{l_1-1}<a_{l_1}=a$).

Note also that, in analogy to  \eqref{keyestlg1}, by Lemma \ref{key} we also have
\begin{equation}\label{tkeyestlg1}
\frac{\tilde A_l-1}{\tilde a_l}+\tilde B_l+2\le m,\quad  l=1,\dots, \tilde L,
\end{equation}
since $\tilde A_1\ge 1.$

\medskip
As in Section \ref{step1}, we shall  again  localize to  these domains $\tilde D_l$ and $\tilde E_l$ in a smooth way by means of the corresponding  functions $\tilde\rho_l$ respectively $\tilde\tau_l$  given by
$$
\tilde\rho_l(\tilde z,s_2):=\chi_0\Big(\frac{ \tilde z}{2^M s_2^{\tilde a_l}}\Big) -\chi_0\Big(\frac{ \tilde z}{2^{-M} s_2^{\tilde a_l}}\Big),
$$
respectively 
$$
\tilde\tau_l(\tilde z,s_2):=\chi_0\Big(\frac{ \tilde z}{2^{-M} s_2^{\tilde a_l}}\Big) \, -\chi_0\Big(\frac{ \tilde z}{2^{M} s_2^{\tilde a_{l+1}}}\Big).
$$
The corresponding functions such as $\mu^\la, J(\la,s)$ and $F(\la,s_2,s_3,y_1),$  will be designated  by means of an extra upper index $\tilde \rho_l$ respectively $\tilde \tau_l,$ such as 
$\mu^{\tilde\rho_l,\la},J^{\tilde \rho}(\la,s),$ or $F^{\tilde \tau_l}(\la,s_2,s_3,y_1).$

\smallskip

Similarly as before, let us here  localize to 
$$
|\tilde z|\sim 2^{-j}, \ s_2\sim 2^{-k},
$$
where we may assume that $j,k$ are sufficiently large integers $\gg 1.$ To this end, we define scaled coordinates $z, \si_2$ by writing 
$\tilde z=2^{-j} z, \ s_2=2^{-k} \si_2,$ so that
$$
|z|\sim 1, \ \sigma_2\sim 1.
$$
Accordingly, let us here put 
\begin{eqnarray*}
\tilde J_{j,k}(\la, s)&:=&\chi_1(2^k s_2)\int e^{-i\la s_3\tilde\Phi_0(\tilde z, s')} \chi_1(2^j\tilde z)\,\eta d\tilde z\\
&=&\chi_1(\si_2)\, 2^{-j} \int  e^{-i\la s_3\tilde\Phi_0^s(z, s_1,\si_2)}  \chi_1(z)\, \eta dz,
\end{eqnarray*}
where we have put
$$
\tilde\Phi_0^s(z, s_1,\si_2):=\tilde\Phi_0(2^{-j}z, s_1,2^{-k}\si_2),
$$
and 
\begin{eqnarray*}
\tilde F_{j,k}(\la,2^{-k}\si_2,s_3,y_1) &:=& 2^{-j}\iint e^{-i\la s_3(\tilde\Phi_0^s(z,s_1,\si_2)- s_1y_1)} \chi_0(s_1)\chi_1(z)  \,\eta dz ds_1.
\end{eqnarray*}
The corresponding contribution to $\mu^\la$ is here given by
\begin{eqnarray}\nonumber
\mu_{j,k}^\la(y+\Ga)&:=&\la^{\frac52}2^{-k}\times \\
&&\iint  \tilde F_{j,k}(\la,2^{-k}\si_2,s_3,y_1) e^{-i\la s_3(s_2^2B(s_2)-2^{-k}\si_2 y_2-y_3)} \chi_1(\si_2) d\si_2 \chi_1(s_3) ds_3. \label{tmujk}
\end{eqnarray}

Of course, we are following here the same conventions as in Subsection \ref{dyadic}. Applying the same reasoning as in that subsection, in analogy to \eqref{Ejk1}, \eqref{Ejk2} and \eqref{Djk1} we see here that 
for the contribution by the transition domain $\tilde E_l,$ i.e., for $ J^{\tilde \tau_l},$ we may assume that $J^{\tilde\tau_l}_{j,k}\equiv 0,$ unless 
\begin{equation}\label{tEjk1}
j\ge k\tilde a_l+M/2, \qquad \text{if} \quad l=1,\dots \tilde L.
\end{equation}
Similarly,  for the contribution by the homogeneous domain $\tilde D_l, 2\le l\le\tilde  L,$ i.e., for $ J^{\tilde \rho_l},$ we  may assume that $J^{\tilde \rho_l}_{j,k}\equiv 0,$ unless 
\begin{equation}\label{tDjk1}
|j-k\tilde a_l|\le M/2.
\end{equation}
For the range of $j$s and $k$s described by \eqref{tEjk1}, respectively \eqref{tDjk1}, we are then also allowed to assume for simplicity that  
 \begin{eqnarray*}
J^{\tilde \tau_l}_{j,k}(\la, s)&=&J^{\tilde \rho_l}_{j,k}(\la, s)=\tilde J_{j,k}(\la, s).
\end{eqnarray*}

The corresponding maximal operators  will be designated  by means of an extra lower index $j,k,$ such as 
$\M_{j,k}^\la,$ or $\M^{\rho_l}_{j,k}.$ 
\medskip

\medskip
\subsection{Resolution of singularity on the transition domain $\tilde E_l$}\label{tElresolution}

Fix $\tilde E_l$ with $l\ge l_1,$ respectively $l\ge l_1-1\ge 1$  if we are in the Subcase (ii) of Case 2.
Arguing as in Subsection \ref{Elresolution}, we may write 
\begin{equation}\label{tElreso}
\tilde\Phi(\tilde z,s_2)=\tilde V(\tilde z,s_2)\, s_2^{\tilde A_l}\tilde z^{\tilde B_l},\qquad (\tilde z,s_2)\in \tilde E_l,
\end{equation}
where $\tilde V$ is analytic in $\tilde z$ and fractional analytic in $s_2$ away from the coordinate axes and $|\tilde V|\sim 1.$
Moreover, by means of Lemma \ref{int} we see again that there is a function $\tilde U(\tilde z,s_2)$  having similar properties as $\tilde V(\tilde z,s_2)$ such that
\begin{equation*}
\pa_{\tilde z}^2\big(\tilde U(\tilde z,s_2)s_2^{\tilde A_l} \tilde z^{\tilde B_l+2}\big)=\tilde\Phi(\tilde z,s_2)=\pa_{\tilde z}^2\tilde\Phi_{0,0}(\tilde z,s_2).
\end{equation*}
 In particular, $|\tilde U|\sim 1.$ 
Again the function $\tilde U(\tilde z,s_2)s_2^{\tilde A_l} \tilde z^{\tilde B_l+2}$ is analytic in $\tilde z$ and vanishes of second order at $\tilde z=0.$ Consequently, we see that (compare also \eqref{tPhi})
\begin{equation*}
\tilde\Phi_{0,0}(\tilde z,s_2)=\tilde U(\tilde z,s_2)s_2^{\tilde A_l} \tilde z^{\tilde B_l+2}+\tilde z \tilde g(s_2)+\tilde h(s_2),
\end{equation*}
with fractional analytic functions $\tilde g(s_2),\tilde h(s_2).$ Since $\tilde z=0$ iff $z_1=w(s_2)$  (recall that 
$\tilde z=z_1-w(s_2)$), these  are explicitly given by 
\begin{eqnarray}\label{tgh}
\begin{split}
\tilde h(s_2)&= \breve\phi_1(w(s_2), s_2), \\
 \tilde g(s_2)&=\pa_{z_1}\breve\phi_1(w(s_2), s_2).
 \end{split}
\end{eqnarray}
Consequently,
\begin{equation*}
\tilde\Phi_{0}(\tilde z, s')=\tilde U(\tilde z,s_2)s_2^{\tilde A_l} \tilde z^{\tilde B_l+2}+ s_1(\tilde z+w(s_2))+\tilde z \tilde g(s_2)+\tilde h(s_2).
\end{equation*}
Let us finally again apply our changes of coordinates $\tilde z=2^{-j} z, \ s_2=2^{-k} \si_2.$  Then, if we put $A:=\tilde A_l, B:= \tilde B_l$ and 
$$
d_{j,k}:=2^{-kA-j(B+2)},
$$
assuming that \eqref{tEjk1} holds true, we obtain
$$
\tilde\Phi_0^s(z,s_1,\si_2)- s_1y_1=d_{j,k} U^s(z,\si_2) \si_2^{A} z^{B+2}+ s_1(2^{-j} z+w(s_2)-y_1)+2^{-j} z \tilde g(s_2)
+\tilde h(s_2),
$$
where $U$ has similar properties as in \eqref{Phi0s}.

\begin{remark}\label{onw}
Note that we have picked up in particular the new term $s_1w(s_2),$ where
$$
|w(s_2)|\sim s_2^{a}\sim 2^{-ka}
$$
(recall that  $a=a_{l_1}$). The same will in fact apply in all subsequent steps of our resolution algorithm, where in Step $p$  we will have to choose for $w(s_2)$ the leading jet 
$$
w(s_2):=c_{l_1}^{\al_1}s_2^{a_{l_1}}+c_{l_1l_2}^{\al_1\al_2}s_2^{a_{l_1l_2}^{\al_1}}+\cdots+
c_{l_1\cdots l_p}^{\al_1\cdots \al_p}s_2^{a_{l_1\cdots
l_p}^{\al_1\cdots \al_{p-1}}}.
$$
\end{remark}
Let us then write
$$
w(s_2)=2^{-ka} \si_2^ a\tilde w(\si_2),
$$
so that $|\tilde w|\sim 1,$ whereas all derivatives of $\tilde w$ can be assumed to be very small. We shall use the following  abbreviations:
\begin{eqnarray*}
Y_1:&=& 2^{ka} y_1,\\
\gamma(\si_2,Y_1)&:=&\si_2^ a\tilde w(\si_2)-Y_1,\\
K&:=&2^{j-ka}.
\end{eqnarray*}
These allow us  to re-write
\begin{eqnarray}\nonumber
&&\tilde\Phi_0^s(z,s_1,\si_2)- s_1y_1\\
&=&d_{j,k} U^s(z,\si_2) \si_2^{A} z^{B+2}+ 2^{-j} s_1\big(K\ga(\si_2,Y_1)+z\big)+2^{-j} z \tilde g(s_2)+\tilde h(s_2),\label{tPhi0s}
\end{eqnarray}
where $z\sim 1\sim \si_2.$ We can then re-write
\begin{eqnarray}\nonumber
&&\tilde F_{j,k}(\la,2^{-k}\si_2,s_3,y_1) = 2^{-j}\iint e^{-i\la s_3(\tilde\Phi_0^s(z,s_1,\si_2)- s_1y_1)} \chi_0(s_1)\chi_1(z)  \,\eta dz ds_1\\ \label{tFjk}
&&=2^{-j} e^{-i\la s_3\tilde h(s_2)} \,\int \vp\big(s_3\la2^{-j} [K\ga(\si_2,Y_1)+z]\big)\\ \nonumber
&&\hskip5cm \times e^{-i\la s_3\big(d_{j,k} U^s(z,\si_2) \si_2^{A} z^{B+2} +2^{-j}z \tilde g(2^{-k}\si_2)\big)}\chi_1(z)\eta dz,
\end{eqnarray}
where $\vp:=\widehat {\chi_0}\in \S(\RR)$ is rapidly decaying.

\subsection{$L^2$ - estimation of $\M_{j,k}^\la$ for the contribution by $\tilde E_l$}\label{L2ontEl}
\medskip
In the same way as we had  argued in Subsection \ref {L2ontEl}, we obtain the same kind of $L^2$-estimate:
\begin{equation}\label{tL2Eljk}
\|\M^\la_{j,k}\|_{2\to 2}\lesssim 2^{-j}(1+\la d_{j,k})^{-\frac 12}.
\end{equation}

\subsection{$L^{1+\ve}$ - estimates of $\M_{j,k}^\la$ for the contribution by $\tilde E_l$ when $\la\ge 2^j$}
\medskip


\subsubsection{Reduction to $K|\gamma(\si_2,Y_1)|\lesssim 1$}\label{Kgagg1}

a) Let us first  consider the $\si_2$ -region where 
\begin{equation}\label{Kgag1}
K|\gamma(\si_2,Y_1)|\gg 1.
\end{equation}
Then $K\ga(\si_2,Y_1)+z\sim K\ga(\si_2,Y_1).$

\medskip
\noi {\bf Case 1: $\la d_{j,k}\ge 1$.} Note that this implies that
$$
\la\ge 2^{k A+j(B+2)}\ge 2^{2j},\quad \text{hence} \quad 2^j\le \la^{\frac 12}.
$$

Given any $0<\de\ll 1,$ we may then assume that $K|\gamma(\si_2,Y_1)|\ge \la^{\de-1} 2^j.$
For, if 
$$
1\ll K|\gamma(\si_2,Y_1)|< \la^{\de-1} 2^j,
$$
then we would have $1\ll \la^{\de-1/2},$ contradicting our assumption that $\la\gg1,$ if we choose $\de<1/2.$ 

\medskip
But, if $K|\gamma(\si_2,Y_1)|\ge \max\{1, \la^{\de-1} 2^j\},$ then 
$
|\la 2^{-j} [K\ga(\si_2,Y_1)+z]|\ge \la^{\de},
$
so that by \eqref{tFjk} and \eqref{tmujk} we see that the contribution to $\mu_{j,k}$ by this $\si_2$-region is a small error term of order $O(\la^{-N})$  for every $N\in \NN.$ 

\medskip
\noi {\bf Case 2: $\la d_{j,k}< 1$.}  Here, we may ignore the term   $d_{j,k} U^s(z,\si_2) \si_2^{A} z^{B+2}$ of the phase in \eqref{tFjk}, since we can include the factor $e^{-i\la s_3(d_{j,k} U^s(z,\si_2) \si_2^{A} z^{B+2})} $ in the amplitude. Let us thus assume that 
\begin{eqnarray*}
\tilde F_{j,k}(\la,2^{-k}\si_2,s_3,y_1) 
&=&2^{-j} e^{-i\la s_3\tilde h(s_2)} \,\int \vp\big(s_3\la2^{-j} [K\ga(\si_2,Y_1)+z]\big)\\
&&\hskip 3cm \times e^{-i\la s_3 \big(2^{-j}z \tilde g^s(\si_2)+\tilde h^s(\si_2)\big)}\chi_1(z)\eta dz,
\end{eqnarray*}
where we have put
$$
\tilde g^s(\si_2):=\tilde g(2^{-k}\si_2),\quad  \tilde h^s(\si_2):=\tilde h(2^{-k}\si_2).
$$
This case will require a more careful analysis. Observe first that by  \eqref{tgh} and \eqref{legendamins}, we have 
\begin{eqnarray*}
\tilde h(s_2)&=& w(s_2)^n\beta(w(s_2))+s_2w(s_2)^m\omega(w(s_2))+s_2^2w(s_2) q(w(s_2), s_2),\\
\tilde g(s_2)&=& w(s_2)^{n-1}\tilde\beta(w(s_2))+s_2w(s_2)^{m-1}\tilde\omega(w(s_2))+s_2^2\tilde q(s_2)
=:\tilde g_0(s_2)+s_2^2\tilde q(s_2).
\end{eqnarray*}
Note that the term $s_2^2w(s_2) q(w(s_2), s_2)$ can be absorbed  by the term $s_2^2B(s_2)$ in \eqref{tmujk}, by slightly modifying the definition of $B(s_2), $ still keeping the condition $|B(s_2)|\sim 1.$  Let us therefore henceforth simply assume that 
$$
\tilde h(s_2)= w(s_2)^n\beta(w(s_2))+s_2w(s_2)^m\omega(w(s_2)).
$$

\begin{lemma}\label{thest}
For every $\al\in\NN,$ the following estimate holds true:
$$
|\pa_{\si_2}^\al \tilde h^s(\si_2)|\lesssim 2^{-k(1+ma)}, 
\qquad |\pa_{\si_2}^\al \tilde g_0^s(\si_2)|\lesssim 2^{-k(1+(m-1)a)},
\quad \text {if} \quad \si_2\sim 1.
$$
\end{lemma}
\begin{proof} 
Note that 
$$
na-(1+ma)=(n-m)a-1\ge (n-m)a_1-1=0,
$$
so that $na\ge 1+ma$ and also $(n-1)a\ge 1+(m-1)a.$
Since $|w(s_2)|\sim s_2^a\sim 2^{-ka},$ the claimed estimates now follow easily.
\end{proof}

\smallskip
\noi {\bf Subcase 2.a: $\la 2^{-ka}\lesssim 1.$} Here,  by Lemma \ref{thest}, $\la \tilde h^s(\si_2)=O(2^{-k})$ and 
$\la \tilde g_0^s(\si_2)=O(2^{-k}),$ since $m\ge 2.$ We may thus ignore the terms $\tilde h^s(\si_2)$ and  $2^{-j}z \tilde g_0(\si_2)$ in the amplitude of $\tilde F_{j,k}.$

Moreover, since  we shall here not use any oscillation from the $z$-integration, we may also  assume that the term $2^{-j} zs_2^2 \tilde q(s_2)$ stemming from our decomposition of $\tilde g(s_2)$ is absorbed by the term $s_2^2B(s_2)$ in \eqref{tmujk}, so that 
$$
\tilde F_{j,k}(\la,2^{-k}\si_2,s_3,y_1) = 2^{-j}\iint e^{-i\la s_32^{-j} s_1[K\ga(\si_2,Y_1)+z]} \chi_0(s_1)\chi_1(z)  \,\eta dz ds_1.
$$

Note  also that
$$
\la2^{-j}[K\ga(\si_2,Y_1)+z]=\la2^{-ka}\si_2^a \tilde w(\si_2) -\la 2^{-ka}Y_1 +\la 2^{-j}z.
$$
The first  term $\la 2^{-ka}\si_2^a \tilde w(\si_2)$ can later on be included into the amplitude in  \eqref{tmujk}, since $\la 2^{-ka}\lesssim 1,$ so that we can ignore this term.
\smallskip

{\bf Assume first that $|Y_1|\gg 1.$} Then 
$1\ll |K\ga(\si_2,Y_1)|\sim K|Y_1|,$ hence 
\begin{equation}\label{Kgasize}
\la 2^{-j} |K\ga(\si_2,Y_1)+z|\sim \la 2^{-ka} |Y_1|.
\end{equation}
We can then proceed in a similar way as we did in  the first part of  Subsection \ref{L1onEl}. 

If  $|Y_1|\gg1 $ and $|Y_1|\ge \la^{\de-1} 2^{ka},$  integrations by parts in $s_1$ show that 
$F_{j,k}=O(\la^{-N})$ for every $N\in\NN,$ and thus also $\mu_{j,k}^\la(y+\Ga)=O(\la^{-N})$ for every $N\in\NN,$ which shows that this leads just to a small error term.

So, assume next that
$$
1\ll |Y_1|\lesssim \la^{\de-1} 2^{ka}.
$$
We then decompose this region  into the regions 
$$
\Delta_{\iota_1}:= \{y:|Y_1|\sim 2^{\iota_1}\}, 
$$
where clearly we may assume that $1\ll\iota_1\lesssim k.$ We further decompose $\Delta_{\iota_1}$ into the subregions
\begin{eqnarray*}
\Delta_{\iota_1,\iota_2}&:=& \{y:|Y_1|\sim 2^{\iota_1}, |Y_2|\sim 2^{\iota_2}\}, \quad 1\ll \iota_2\le k,\\
\Delta_{\iota_1,0}&:=& \{y:|Y_1|\sim 2^{\iota_1}, |Y_2|\lesssim 1\}.
\end{eqnarray*}
Then 
$$
|\Delta_{\iota_1,\iota_2}|\lesssim 2^{\iota_1+\iota_2-ka-k}.
$$
For $y\in \Delta_{\iota_1},$ then we may integrate by parts in $s_1.$ 
Moreover, if $\iota_2\gg1,$  and if $\la 2^{-2k+\iota_2}\ge 1,$ then  our preceding discussion and \eqref{onlys2} show that we can also integrate by parts in $\si_2,$ and altogether this leads to the following estimate for $\mu_{k,j,\iota}^\la:$
$$
|\mu_{j,k,\iota}^\la(y+\Ga)|\lesssim \la^{\frac 52} 2^{-j-k} (\la 2^{-ka} 2^{\iota_1})^{-1}(\la 2^{-2k} 2^{\iota_2})^{-1}=\la^{\frac 12} 2^{-j+k+ka-\iota_1-\iota_2}.
$$
Thus, Lemma \ref{maxproj} shows that, for every $\ve >0,$ 
$$
\|\M_{j,k,\iota}^\la\|_{1+\ve\to 1+\ve}\lesssim \la^{\frac 12}2^{-j}.
$$
And, if $\iota_2=0,$ we can again apply van der Corput's estimate of order 2  in $\si_2,$ which leads to 
$$
|\mu_{j,k,\iota}^\la(y+\Ga)|\lesssim \la^{\frac 52} 2^{-j-k} (\la 2^{-ka} 2^{\iota_1})^{-1}(\la 2^{-2k})^{-\frac 12}
=\la 2^{-j+ka-\iota_1},
$$
hence
$$
\|\M_{j,k,(\iota_1,0)}^\la\|_{1+\ve\to 1+\ve}\lesssim \la^{1+\ve}2^{-j-k}.
$$
These estimate have to be summed over the  $\iota_2$s, which will produce an extra factor of order $O(k).$ However, later on we have to sum over the $j$s and $k$s such that $k\lesssim j/\tilde a,$ so effectively by summing we shall only pick up a factor of order $O(2^{\ve j}),$ where we may choose $\ve>0$ as small as we need;  this will cause no harm.

\smallskip

{\bf Assume next that $|Y_1|\lesssim 1.$} Here we decompose into  the regions
\begin{eqnarray*}
\Delta_{\iota_2}&:=& \{y:|Y_1|\lesssim 1, |Y_2|\sim 2^{\iota_2}\}, \quad 1\ll \iota_2\le k,\\
\Delta_{0}&:=& \{y:|Y_1|, |Y_2|\lesssim 1\}.
\end{eqnarray*}
Then 
$$
|\Delta_{\iota_2}|\lesssim 2^{\iota_2-ka-k}.
$$
We can now only gain from integration by parts (respectively an application of van der Corput's estimate) in $\si_2,$ and thus obtain 
$$
|\mu_{j,k,\iota_2}^\la(y+\Ga)|\lesssim \la^{\frac 52} 2^{-j-k} (\la 2^{-2k} 2^{\iota_2})^{-1}=\la^{\frac 32} 2^{-j+k-\iota_2}\lesssim 
\la^{\frac 12} 2^{-j+k+ka-\iota_2},
$$
where we have used our assumption $\la 2^{-ka}\lesssim 1.$
Similarly, 
$$
|\mu_{j,k,0}^\la(y+\Ga)|\lesssim \la^{\frac 52} 2^{-j-k} (\la 2^{-2k})^{-\frac 12}
=\la^2 2^{-j}\lesssim \la 2^{-j+ka}.
$$
This shows that we arrive at the same estimates as  before, if we choose $\iota_1=0$ therein.

\medskip
\noi {\bf Subcase 2.b: $\la 2^{-ka}\gg1.$} {\bf Assume again first that $|Y_1|\gg 1,$} so that \eqref{Kgasize} holds true.

If  $|Y_1|\ge\la^\de$ for some $\de $ such that $0<\de\ll 1,$ then $\la 2^{-j} |K\ga(\si_2,Y_1)+z|\gtrsim \la^{\de},$ hence
\begin{equation}\label{rdecay}
\mu_{j,k}^\la(y+\Gamma)=O(\la^{-N})\qquad\text{for every}\quad N\in \NN,
\end{equation}
since $\vp$ is rapidly decaying, so this case contributes again only a small error term.
\medskip

 So, assume next w.l.o.g.  that $1\ll |Y_1|\le \la^\de.$ Then 
\begin{equation}\label{y1size}
|y_1|\lesssim \la^\de2^{-ka},
\end{equation}
and, by \eqref{Kgasize},
\begin{equation}\label{phisize}
|\vp\big(s_3\la2^{-j} [K\ga(\si_2,Y_1)+z]\big)|\lesssim (\la2^{-ka} |Y_1|)^{-N}\le  (\la2^{-ka})^{-N}
\end{equation}
for every $N\in\NN.$ 
\smallskip

Again, if $\la 2^{-ka}\ge \la^\de,$ this implies \eqref{rdecay}, so let us assume that $\la 2^{-ka}\le \la^\de.$
Then, by Lemma \ref{thest}, since $m\ge 2,$ 
$$
|\la \tilde h^s(\si_2)|\lesssim \la 2^{-k(1+ma)}=\la^{1-m} (\la 2^{-ka})^m 2^{-k}\ll \la^{1-m} \la^{\de m}\le\la^{\de m-1}\ll1,
$$
if we choose $\de>0$ sufficiently small. Similarly,
$$
|\la 2^{-j} z\tilde g_0^s(\si_2)|\lesssim 2^{-j}\la 2^{-k(1+(m-1)a)}\le 2^{-j}(\la 2^{-ka})2^{-k}.
$$
Let us next write $\tilde a=da,$ where $d>0.$  Since we are on $\tilde E_l,$ by \eqref{tEjk1} we then have that  $j\ge k\tilde a=kda.$  Thus, for every $\ve >0,$ 
\begin{eqnarray*}
|\la 2^{-j} z\tilde g_0^s(\si_2)|&\lesssim&  2^{-kda}(\la 2^{-ka})2^{-k}= \la^{-\ve} (\la 2^{-ka})^{1+\ve} 2^{-k(1+da-a\ve)}\\
&\le& \la^{-\ve +\de(1+\ve)} 2^{-k(1+da-a\ve)}.
\end{eqnarray*}
Choosing first $\ve>0$ sufficiently small, and then $\de>0$ sufficiently small, we see that also $|\la 2^{-j} z\tilde g_0^s(\si_2)|\ll 1.$ 

We may thus again ignore the terms $\tilde h^s(\si_2)$ and  $2^{-j}z \tilde g_0(\si_2)$ in the amplitude of $\tilde F_{j,k},$ so that we may assume that 

\begin{eqnarray}\nonumber
\mu_{j,k}^\la(y+\Ga)&=&\la^{\frac52}2^{-j-k}\iint \vp\big(s_3\la2^{-j} [K\ga(\si_2,Y_1)+z]\big)  \,e^{i\la s_3 y_3} (1-\chi_0(K\ga(\si_2,Y_1)))\\
&&\hskip 0.2cm \times  e^{-i\la s_32^{-2k}\Big(\si_2^2\big(B(2^{-k}\si_2)+2^{-j}z \tilde q^s(\si_2)\big)-\si_2 Y_2\Big)}\chi_1(z)  \chi_1(\si_2) \chi_1(s_3) \eta dz d\si_2  ds_3.\label{tmujk2}
\end{eqnarray}
Here, we assume that $\chi_0\equiv 1$ on a sufficiently large neighborhood of the origin, and we have put 
$
Y_2:= 2^k y_2.
$
\medskip

\noi {\bf Reduction to $|Y_2|\lesssim 1.$} Note that here
\begin{eqnarray*}
\Big|\pa_{\si_2} \big(\vp\big(s_3\la2^{-j} [K\ga(\si_2,Y_1)+z]\big)\big)\Big|&\lesssim& \la 2^{-ka}\Big| \vp'\big(s_3\la2^{-j} [K\ga(\si_2,Y_1)+z]\big)\Big|\\
&\lesssim& \Big| \tilde \vp\big(s_3\la2^{-j} [K\ga(\si_2,Y_1)+z]\big)\Big|,
\end{eqnarray*}
where $\tilde \vp\in \S(\RR)$, and where we have used \eqref{Kgasize} and $1\le |Y_1|$  for the last inequality. Note that also the $\si_2$-derivative of  $\chi_0(K\ga(\si_2,Y_1))$ can be controlled in a similar way, since  $ K=2^j2^{-ka}\le \la2^{-ka}.$

Thus, the $\si_2$-derivative  of the amplitude in \eqref{tmujk2} behaves in the same way as the factor itself.
In a similar way as in Subsection \ref{L1onEl}, let us consider now the subregions 
\begin{eqnarray*}
\Delta_{\iota}&:=& \{y: 1\ll |Y_1|\le \la^\de, |Y_2|\sim 2^{\iota}\}, \quad 1\ll \iota_2\le k.
\end{eqnarray*}
Since for $y\in \Delta_{\iota}$
$$
\Big|\pa_{\si_2}\Big(\si_2^2\big(B(2^{-k}\si_2)+2^{-j}z \tilde q^s(\si_2)\big)-\si_2 Y_2\Big)\Big|\sim 2^\iota,
$$
an integration by parts in $\si_2,$  using also \eqref{phisize} with $N=1,$  shows that
$$
|\mu_{j,k}^\la(y+\Ga)|=\la^{\frac52}2^{-j-k} (\la2^{-ka})^{-1}\, (\la2^{-2k}2^\iota)^{-1}=\la^{\frac 12}2^{-j}2^{k(a+1)} 2^{-\iota}.
$$
And, by \eqref{y1size}, we find that
$
|\Delta_\iota|\lesssim \la^\de2^{-ka} 2^{-k} 2^{\iota},
$
so that we conclude that 
$$
\|\M_{j,k,\iota}^\la\|_{1+\ve\to 1+\ve}\lesssim \la^{\frac 12+\ve}2^{-j},
$$
provided we choose $\de>0$ sufficiently small. 
Here, $\M_{j,k,\iota}^\la$ denotes the contribution  by $\Delta_\iota$  to $\M_{j,k}^\la.$ Thus, we obtain the following analogue of 
\eqref{Mjkiota}:
$$
\sum\limits_{1\ll \iota\le k}\|\M_{j,k,\iota}^\la\|_{1+\ve\to 1+\ve}\lesssim \la^{\frac 12+\ve}2^{-j}.
$$

We have thus reduced ourselves to the region 
$$
\Delta_{0}:= \{y: 1\ll |Y_1|\le \la^\de, |Y_2|\lesssim 1\}.
$$
Here,  we apply van der Corput's estimate of order $2$ (Lemma \ref{Corput}) in $\si_2$ and obtain
$$
|\mu_{j,k}^\la(y+\Ga)|=\la^{\frac52}2^{-j-k} (\la2^{-ka})^{-1}\, (\la2^{-2k})^{-\frac 12}=\la2^{-j}2^{ka},
$$
whereas $|\Delta_0|\lesssim \la^\de2^{-ka} 2^{-k},$ so that we shall here obtain  the estimate
$$
\|\M_{j,k,0}^\la\|_{1+\ve\to 1+\ve}\lesssim \la^{1+\ve}2^{-j-k}
$$
as in \eqref{mujkL1la} of Section \ref{step1}, and altogether we arrive also here at  the estimate \eqref{Mjk>1} for the contribution by the region where $|Y_1|\gg 1.$ 
\medskip

We have thus reduced ourselves to estimating {\bf the contribution by the region where  $|Y_1|\lesssim1.$} 
Recall also that we are still assuming \eqref{Kgag1} and $\la2^{-j}\ge 1.$ 
\medskip

We next observe that we can further reduce to the case where $1\ll |K\ga(\si_2,Y_1)|\lesssim 2^j\la^{\de-1},$ again with $0<\de\ll1,$ for if $|K\ga(\si_2,Y_1)|\gg 2^j\la^{\de-1},$ we see from \eqref{tmujk2} that the contribution to $\mu_{j,k}$ by this $\si_2$-region is a small error term of order $O(\la^{-N})$  for every $N\in \NN.$ 

So, let us assume that 
$$
\si_2\in \Sigma:=\{\si_2\sim 1: |\ga(\si_2,Y_1)|\lesssim \la^{\de-1}2^{ka}\}.
$$
Then $|\Sigma|\lesssim \la^{\de-1}2^{ka},$ and thus
$$
\iint \limits_\Sigma\big|\vp\big(s_3\la2^{-j} [K\ga(\si_2,Y_1)+z]\big)\big|\, \chi_1(\si_2) |\eta| dz d\si_2 
\lesssim |\Sigma|(\la2^{-j})^{-1}=\la^{\de-2}2^{j+ka},
$$
hence, by \eqref{tmujk2},
$$
|\mu_{j,k}^\la(y+\Ga)|=\la^{\frac52}2^{-j-k} \la^{\de-2}2^{j+ka}=\la^{\frac12+\de}2^{k(a-1)}.
$$
Moreover, since here $y\in \Delta:=\{y: 1\ll |Y_1|\lesssim 1\},$ where $|\Delta|\lesssim 2^{-ka},$ by choosing $\de>0$ sufficiently small  we  may conclude that
\begin{equation*}
\|\M_{j,k,|Y_1|\lesssim 1}^\la\|_{1+\ve\to 1+\ve}\lesssim \la^{\frac 12+\ve}2^{-k}.
\end{equation*}
{\it  Note that this type of estimate had not yet appeared in Section \ref{step1}.}

\medskip

b) We have thus reduced ourselves to estimating the contribution by the $\si_2$ -region where 
\begin{equation}\label{Kgal1}
K|\gamma(\si_2,Y_1)|\lesssim 1.
\end{equation}
A look at the phase $\tilde\Phi_0^s(z,s_1,\si_2)- s_1y_1$ in \eqref{tPhi0s} then shows that we can now again apply Lemma \ref{pinvers} in the variables $(s_1,z)$ in place of $(\eta,x)$ to the oscillatory integral $\tilde F_{j,k}(\la,2^{-k}\si_2,s_3,y_1)$ in \eqref{tFjk}, with $N:=\la 2^{-j}\ge 1.$ 
Moreover, observe that $s_1$ and $z$   result  from our original coordinates $s_1$ and $z_1=x_1$ by means of  a smooth change of coordinates.  Passing back to our original coordinate, according to Remarks \ref{critvalue}  we therefore must evaluate the original phase $\breve\phi_1(x_1, s_2)+s_1x_1-s_1y_1$ in \eqref{defiF} at $x_1=y_1,$ and thus get
\begin{equation*}
 \tilde F_{j,k}(\la,2^{-k}\si_2,s_3,y_1)   = (\la2^{-j})^{-1} 2^{-j}e^{-i\la s_3(\breve\phi_1(y_1, s_2)}    \eta(y_1,2^{-k}\si_2).
  \end{equation*}
  This implies that
 \begin{eqnarray*}
\mu_{j,k,K|\ga|\lesssim1}^\la(y+\Ga)&=&\la^{\frac32}2^{-k}
\int  e^{-i\la s_3\big(\breve\phi_1(y_1, 2^{-k}\si_2)+2^{-2k}\si_2^2B(2^{-k}\si_2)-2^{-k}\si_2 y_2-y_3\big)} \\
&&\hskip2cm \times \chi_1(\si_2) \chi_1(s_3) \, \chi_0(K\ga(\si_2,Y_1))\, \eta d\si_2ds_3
\end{eqnarray*}
for the contribution by the region where $K|\ga(\si_2,Y_1)|\lesssim 1.$
By \eqref{legendamins}, we have 
$$
\breve\phi_1(y_1, 2^{-k}\si_2):=y_1^n\beta(y_1)+2^{-k}\si_2y_1^m\omega(y_1)+2^{-2k}\si_2^2 y_1 q(y_1, 2^{-k}\si_2).
$$
The  last term  can be absorbed by the term $2^{-2k}\si_2^2B(2^{-k}\si_2),$ by slightly modifying $B(s_2)$ to $B(s_2,y_1),$ which allows us  to assume  for simplicity that $q\equiv 0.$  Then we may re-write 
 \begin{eqnarray*}
\mu_{j,k,K|\ga|\lesssim1}^\la(y+\Ga)&=&\la^{\frac32}2^{-k}
\int  e^{-is_3\la \big(2^{-k}\Phi_1(\si_2, y)+y_1^n\beta(y_1)-y_3\big)} \\
&&\hskip2cm \times \chi_1(\si_2) \chi_1(s_3) \, \chi_0(K\ga(\si_2,Y_1))\, \eta d\si_2ds_3,
\end{eqnarray*}
where the phase $\Phi_1(\si_2, y)$ is given by 
$$
\Phi_1(\si_2, y):=-\si_2[y_2-y_1^m\omega(y_1)]+2^{-k}\si_2^2B(2^{-k}\si_2,y_1).
$$
We shall now distinguish two subcases:

\medskip

\noi{\bf The subcase where $K\lesssim 1.$} We claim that then our assumption \eqref{Kgal1} implies that  $|y_1|\lesssim 2^{-j}.$

To this end, note first that the condition $K\lesssim 1$ is equivalent to $2^{-ka}\lesssim 2^{-j}.$ Thus, if $|Y_1|\lesssim 1,$ then $|y_1|\lesssim 2^{-ka}\lesssim 2^{-j}.$ And, if $|Y_1|\gg 1,$  then 
$1\gtrsim K|\ga(\si_2,Y_1))|\sim K|Y_1|=2^j|y_1|,$ so that again  $|y_1|\lesssim 2^{-j}.$

\smallskip Observe next that here all $\si_2$-derivatives of $\chi_0(K\ga(\si_2,Y_1))$ are uniformly bounded. This shows that we can now just follow the arguments from  the last part of Subsection \ref{L1onEl} (after identity \eqref{applemma5.2}), by either integrating by parts in $\si_2,$ or applying van der Corput's estimates, in order to arrive again at  estimates of the form \eqref{Mjk} for the corresponding contributions to $\M_{j,k}^\la$ (notice that in Subsection \ref{L1onEl} the quantity $Y_1$ had been  defined in a different way, so that there 
$|Y_1|\lesssim 1$ had been equivalent to $|y_1|\lesssim 2^{-j}$).
 
\medskip

\noi{\bf The subcase where $K\gg 1.$} Here we see that  necessarily $|Y_1|\lesssim 1.$  We shall again try to follow our reasoning from  Subsection \ref{L1onEl}, but shall see that the case $\iota_2=0$ will still require a more refined reasoning.

Let us again decompose the region where  $|Y_1|\lesssim 1$ into the subregions
\begin{eqnarray*}
\Delta_{\iota_2}&:=& \{y:|Y_1|\lesssim 1,2^k |y_2 -y_1^m\omega(y_1)|\sim 2^{\iota_2}\}, 1\ll \iota_2\le k,\\
\Delta_{0}&:=& \{y:|Y_1|\lesssim 1, 2^k|y_2 -y_1^m\omega(y_1)|\lesssim 1\}.
\end{eqnarray*}
Then 
$$
|\Delta_{\iota_2}|\lesssim 2^{-ka} 2^{\iota_2-k}.
$$

{\bf Assume first that $\iota_2\gg 1.$} Observe first that 
$$
\int \Big|\pa_{\si_2}\big( \chi_0(K\ga(\si_2,Y_1)) \chi_1(\si_2)  \eta \big)\Big|\,d\si_2
\lesssim K\int \big| \chi'_0(K\ga(\si_2,Y_1))\big| \chi_1(\si_2)\,d\si_2 \lesssim 1,
$$
since here $\si_2\sim 1,$ so that we may use $\tau_2:= \si_2^a\tilde w(\si_2)$ as a new variable of integration in place of $\si_2.$
Thus, we may  again integrate by parts in $\si_2$ and obtain that 
$$
|\mu_{j,k,K|\ga|\lesssim1,\iota_2}^\la(y+\Ga)|\lesssim \la^{\frac32}2^{-k}(\la 2^{-2k+\iota_2})^{-1}=\la^{\frac 12}2^{k-\iota_2},
$$
hence 
\begin{equation}\label{L1K>1<iota}
\|\M_{j,k,K|\ga|\lesssim1,\iota_2}^\la\|_{1+\ve\to 1+\ve}\lesssim \la^{\frac 12+\ve}2^{-ka}.
\end{equation}
{\it  Note that this type of estimate had not yet appeared in Section \ref{step1}.}

{\bf Assume finally  that $\iota_2=0.$}  Then a direct application of  van der Corput's estimate of order 2 in $\si_2$ would lead in a similar way to the estimate
$$
\|\M_{j,k,K|\ga|\lesssim1,0}^\la\|_{1+\ve\to 1+\ve}\lesssim \la 2^{-k} 2^{-ka},
$$
which turns out to be insufficient. Indeed, we shall need to gain a factor $K^{-1}=2^{ka} 2^{-j}$ over  this estimate, which will require a more refined argument.
\medskip

Note first that 
$|\pa_{\si_2}\ga(\si_2,Y_1)|\sim 1.$  This allows us to  write, for $\si_2\sim 1,$ 
$$
\ga(\si_2,Y_1)	= (\sigma_2-\sigma_2^c) W(\sigma_2,Y_1),
$$
where $\si_2^c=\si_2(Y_1)$ and $ W$ are analytic functions with $|\si_2^c|\sim 1\sim |W|.$
It is then natural to change coordinates from $\si_2$ to 
$$
\tau_2:=K(\si_2-\si_2^c),\quad \text{i.e.,}\quad \si_2=\si_2^c+K^{-1}\tau_2.
$$

Then, for $y\in \Delta_0,$ we may assume that, with respect to our new variables of integration, 
\begin{eqnarray}\nonumber
\mu_{j,k,K|\ga|\lesssim1,0}^\la(y+\Ga)&=&\la^{\frac32}2^{-k}K^{-1}
\int  e^{-is_3\la \big(2^{-k}\Phi_2(\tau_2, y)+y_1^n\beta(y_1)-y_3\big)} \\
&&\hskip2cm \times \chi_1(\si_2^c+K^{-1}\tau_2) \chi_1(s_3) \, \chi_0(\tau_2)\, \eta d\tau_2ds_3,\label{tmutau2}
\end{eqnarray}
where
\begin{eqnarray*}
\Phi_2(\tau_2,y):=\Phi_1(\si_2^c+K^{-1}\tau_2,y)=-\big(\si_2[y_2-y_1^m\omega(y_1)]+2^{-k}\si_2^2B(2^{-k}\si_2,y_1)\big)
\Big|_{\si_2:=\si_2^c+K^{-1}\tau_2}.
\end{eqnarray*}
Since $y\in \Delta_0,$ we may further write 
$$
y_2-y_1^m\omega(y_1)=:2^{-k}Y_2, \quad \text{where}\quad |Y_2|\lesssim 1,
$$
so that
$$
\Phi_2(\tau_2,y)=\Phi_1(\si_2^c+K^{-1}\tau_2,y)=2^{-k}\big(-\si_2Y_2+\si_2^2B(2^{-k}\si_2,y_1)\big)
\Big|_{\si_2:=\si_2^c(Y_1)+K^{-1}\tau_2}.
$$

By performing a  Taylor expansion of the phase  function $\Phi_1(\si_2^c+K^{-1}\tau_2,y)$ with respect to $K^{-1}\tau_2,$ we see that we may write
$$
\la2^{-k}\Phi_2(\tau_2,y)
= \la2^{-2k}\Big(H_0(Y)+ K^{-1} \tau_2\,H_1(Y)+
K^{-2} \tau_2^2\, H_2(Y,K^{-1}\tau_2)\Big),
$$
where the $H_j$ are analytic functions of their very variables, and where
\begin{equation}\label{onH}
|H_2|\sim 1, \qquad H_1(Y)=-(Y_2-A(y_1))=-2^k\big(y_2-y_1^m\om(y_1)-2^{-k}A(y_1)\big).
\end{equation}
Here, 
$A(y_1)=2\si_2^c(Y_1)B(0,y_1)+D(y_1),$ with an error  term $D(y_1)$ of order $O(2^{-k}).$ Note also that
$$
\big|\pa_{\tau_2}^2\big[\tau_2^2H_2(Y,K^{-1}\tau_2)\big]\big|\sim 1, \qquad \big|\pa_{\tau_2}\big[\tau_2^2H_2(Y,K^{-1}\tau_2)\big]\big|\lesssim 1,
$$
since $K\gg 1.$

We are thus led to decomposing the $y$-region $\Delta_0$  into the following subsets:
\begin{eqnarray*}
\Delta_{0,\nu}&:=&\{y\in \Delta_0: |H_1(Y)|\sim  2^\nu K^{-1}\}, \qquad 1\ll \nu \lesssim  \log K\le \log\la;\\
\Delta_{0,0}&:=&\{y\in \Delta_0: |H_1(Y)|\lesssim K^{-1}\}.
\end{eqnarray*}
As usually, the corresponding measures will be denoted by $\mu_{j,k,\nu}^\la,$ and so forth. 
Since $2^{ka}|y_1|=|Y_1|\lesssim 1$ on $\Delta_0,$ in combination with \eqref{onH} we see that 
\begin{equation}\label{tsizeDeltanu}
|\Delta_{0,\nu}|\lesssim 2^{-ka} 2^{-k} 2^\nu K^{-1}.
\end{equation}

\smallskip
{\bf Assume first that $\nu\gg 1.$} Then an integration by parts with respect  to  $\tau_2$ in \eqref{tmutau2} leads to 
$$
|\mu_{j,k,\nu}^\la(y+\Gamma)|\lesssim 
\la^{\frac32}2^{-k}K^{-1}\big[\la 2^{-2k} \, 2^\nu K^{-2}\big]^{-1}=\la^{\frac 12} 2^{k-\nu} K,
$$
hence, in combination with \eqref{tsizeDeltanu},
$$
\|\M_{j,k,\nu}^\la\|_{1+\ve\to 1+\ve}\lesssim \la^{\frac 12+\ve}2^{-ka},
$$
i.e., the same type of estimate as in \eqref{L1K>1<iota}.
\medskip

{\bf Assume next  that $\nu=0.$} Here, we can again apply van der Corput's estimate of order 2 in $\tau_2$ and obtain that
$$
|\mu_{j,k,0}^\la(y+\Gamma)|\lesssim 
\la^{\frac32}2^{-k}K^{-1}\big[\la 2^{-2k} \,K^{-2}\big]^{-\frac 12}=\la,
$$
hence 
$$
\|\M_{j,k,0}^\la\|_{1+\ve\to 1+\ve}\lesssim \la 2^{-ka} 2^{-k} K^{-1}=\la^{1+\ve} 2^{-j-k},
$$
which is  again of the form \eqref{mujkL1la}.
\bigskip

Combining all the estimates from this subsection, we see that in analogy with the estimate \eqref{Mjk} which we had established for    Step 1, for  Step 2 we obtain that for every $\ve>0$ we have the following estimate:
\begin{equation}\label{tMjk}
\|\M_{j,k}^\la\|_{1+\ve\to 1+\ve}\lesssim \la^{1+\ve}2^{-j-k}+  \la^{\frac 12+\ve}2^{-j} +\la^{\frac 12+\ve}2^{-k}+\la^{\frac 12+\ve}2^{-ka}.
\end{equation}

\medskip
\subsection{$L^p$ - estimates of $\M_{j,k}^\la$  for the contribution by $E_l$ when $\la\ge 2^{j}$}
\medskip

Interpolation between the $L^2$ - estimate \eqref{tL2Eljk} and the $L^{1+\ve}$ - estimate \eqref{tMjk} leads to 
\begin{eqnarray}\nonumber
\|\M_{j,k}^\la\|_{p\to p}&\lesssim&
 (\la^{1+\ve}2^{-k})^{\frac 2p -1}(1+\la d_{j,k})^{-\frac 12 (2-\frac 2p)} 2^{-j} 
+ (\la^{\frac 12+\ve})^{\frac 2p -1}(1+\la d_{j,k})^{-\frac 12 (2-\frac 2p)} 2^{-j}\\ \nonumber
&+&(\la^{\frac 12+\ve}2^{-k})^{\frac 2p -1}\big(2^{-j}(1+\la d_{j,k})^{-\frac 12}\big)^{2-\frac 2p}\\ \nonumber
&+& (\la^{\frac 12+\ve}2^{-ka})^{\frac 2p -1}\big(2^{-j}(1+\la d_{j,k})^{-\frac 12}\big)^{2-\frac 2p}\\ \nonumber
&\lesssim&
\la^{\ve'} (\la d_{j,k})^{\frac 2p -1} (1+\la d_{j,k})^{\frac 1p-1} \, (d_{j,k}^{-1})^{\frac 2p-1}\, 2^{-k(\frac 2p -1)}\,2^{-j} \\ \nonumber
&&+\, \la^{\ve'}  (\la d_{j,k})^{\frac 1p -\frac 12}(1+\la d_{j,k})^{\frac 1p-1} \, (d_{j,k}^{-1})^{\frac 1p-\frac 12}  2^{-j}\\
&&+ \la^{\ve'} 
(\la d_{j,k})^{\frac 1p -\frac 12}(1+\la d_{j,k})^{\frac 1p-1} \, (d_{j,k}^{-1})^{\frac 1p-\frac 12}\, 2^{-k(\frac 2p-1)}2^{-j(2-\frac 2p)}\nonumber\\
      &&+ \la^{\ve'} 
(\la d_{j,k})^{\frac 1p -\frac 12}(1+\la d_{j,k})^{\frac 1p-1} \, (d_{j,k}^{-1})^{\frac 1p-\frac 12}\, 2^{-ka(\frac 2p-1)}2^{-j(2-\frac 2p)}, \nonumber
\end{eqnarray}
if $1<p< 2.$ 
Arguing as in Subsection \ref{LpEl}, we  see that we can sum over all dyadic $\la$s and obtain that, for every $\ve'>0,$ 
\begin{eqnarray}\nonumber
\sum\limits_{\la\gg 1}\|\M_{j,k}^\la\|_{p\to p}&\lesssim& (d_{j,k}^{-1}2^{-k})^{\frac 2p-1+\ve'}\,2^{-j} 
+  (d_{j,k}^{-1})^{\frac 1p-\frac 12+\ve'}\, 2^{-j}\\  \label{tsuminla1}
&+&(d_{j,k}^{-1})^{\frac 1p-\frac 12+\ve'}\, 2^{-k(\frac 2p-1)}2^{-j(2-\frac 2p)}\\ 
&+&(d_{j,k}^{-1})^{\frac 1p-\frac 12+\ve'}\, 2^{-ka(\frac 2p-1)}2^{-j(2-\frac 2p)}. \nonumber
\end{eqnarray}
Recall next that $A=\tilde A_l, B=\tilde B_l,$ so that by \eqref{tkeyestlg1}
$$
\frac{A-1}{\tilde a}+ B+2\le m, \quad  \text{and} \quad d_{j,k}^{-1}=2^{kA+j(B+2)}.
$$
Moreover, since we are on $\tilde E_l,$ by \eqref{tEjk1} we have $j\ge k\tilde a.$ 
\smallskip
This shows that we can argue in exactly the same way as in Subsection \ref{geolemma}  in order to show that we can sum the first two terms on the right-hand side of \eqref{tsuminla1} over all $k$s and $j$s such that $j\ge k\tilde a.$

\smallskip 
This leaves us with showing that 
\begin{equation}\label{tsuminjk1}
\sum\limits_{j,k\gg1:  j\ge k \tilde a}\Big((d_{j,k}^{-1})^{\frac 1p-\frac 12+\ve'}\, 2^{-k(\frac 2p-1)}2^{-j(2-\frac 2p)}
+  (d_{j,k}^{-1})^{\frac 1p-\frac 12+\ve'}\, 2^{-ka(\frac 2p-1)}2^{-j(2-\frac 2p)}\Big)<\infty,
\end{equation}
provided $\ve'>0$ is sufficiently small.

\smallskip
But, if $j\ge k \tilde a,$  since $A\ge 1,$ we see that 
\begin{eqnarray*}
(d_{j,k}^{-1})^{\frac 1p-\frac 12+\ve'}\, 2^{-k(\frac 2p-1)}2^{-j(2-\frac 2p)}
&=& ( 2^{kA+j(B+2)})^{\frac 1p-\frac 12+\ve'}\, 2^{-k(\frac 2p-1)}2^{-j(2-\frac 2p)}\\
&=& ( 2^{k(A-1)+j(B+2)})^{\frac 1p-\frac 12+\ve'}\, 2^{-k(\frac 1p-\frac 12 -\ve')}2^{-j(2-\frac 2p)}\\
&\le& 2^{j\big(\frac{A-1}{\tilde a}+B+2\big)(\frac 1p-\frac 12+\ve')}\, 2^{-k(\frac 1p-\frac 12 -\ve')}2^{-j(2-\frac 2p)}\\
&\le& 2^{jm (\frac 1p-\frac 12+\ve')}\, 2^{-k(\frac 1p-\frac 12 -\ve')}2^{-j(2-\frac 2p)}\\
&=& 2^{j\big[m (\frac 1p-\frac 12+\ve')+\frac 2p -2\big]}\, 2^{-k(\frac 1p-\frac 12 -\ve')}.
\end{eqnarray*}
But, 
$$
m (\frac 1p-\frac 12)+\frac 2p -2<m \frac 1{2m} +\frac 2p -2=\frac 2p-\frac 32<0,
$$  
since $p>\max\{3/2, h\}.$  Thus we see that if we choose  $\ve'$ sufficiently small, then both the exponent of $2^{j}$ as well as that of $2^k$ in the last estimate are negative, so that we can sum these estimates over all $j,k\ge 1.$  
\smallskip

We are thus left with the second sum in \eqref{tsuminjk1}. But, since $j\ge k\tilde a,$
\begin{eqnarray*}
(d_{j,k}^{-1})^{\frac 1p-\frac 12+\ve'}\, 2^{-ka(\frac 2p-1)}2^{-j(2-\frac 2p)}
&=& ( 2^{kA+j(B+2)})^{\frac 1p-\frac 12+\ve'}\, 2^{-ka(\frac 2p-1)}2^{-j(2-\frac 2p)}\\
&=& ( 2^{k(A-1)+j(B+2)})^{\frac 1p-\frac 12+\ve'}\, 2^{k(\frac 1p-\frac 12 +\ve')-ka(\frac 2p-1)}2^{-j(2-\frac 2p)}\\
&\le &  2^{j\big(\frac {A-1}{\tilde a}+B+2\big)(\frac 1p-\frac 12+\ve')}\, 2^{k(\frac 1p-\frac 12 +\ve'-\frac {2a}p+a)}2^{-j(2-\frac 2p)}\\
&\le &  2^{j\big[m(\frac 1p-\frac 12+\ve')+\frac 2p-2\big]}\, 2^{k(\frac 1p-\frac 12 +\ve'-\frac {2a}p+a)}.
\end{eqnarray*}
Again, we may assume that the exponent of $2^j$ is strictly negative. 
\smallskip

{\bf Assume first that $\tilde a\ge a.$ } Then $j\ge ka,$ hence  after summing first in $j\ge ka,$ we are left with the sum
\begin{eqnarray*}
\sum\limits_{k\ge 1} 2^{ka\big[m(\frac 1p-\frac 12+\ve')+\frac 2p-2\big]}\, 2^{k(\frac 1p-\frac 12 +\ve'-\frac {2a}p+a)}.
\end{eqnarray*}
But, the exponent of $2^k$ can be estimated as follows:
\begin{eqnarray*}
a\big[m(\frac 1p-\frac 12)+\frac 2p-2\big]+\frac 1p-\frac 12 -\frac {2a}p+a&=&(am+1)(\frac 1p-\frac 12)-a\\
<(am+1)\frac 1n-a=\frac{1-(n-m)a}n\le 0,
\end{eqnarray*}
since $p>h$ and $a\ge a_1=1/(n-m).$ 
By choosing $\ve'$ sufficiently small, we can thus  finally also sum over all $k\ge 1.$

\smallskip

{\bf Assume finally  that $\tilde a< a.$ } Then $-2a/p+a=-2a(1/p-1/2)<-2\tilde a(1/p-1/2)=-2\tilde a/p+\tilde a,$ so that we can here estimate
\begin{eqnarray*}
(d_{j,k}^{-1})^{\frac 1p-\frac 12+\ve'}\, 2^{-ka(\frac 2p-1)}2^{-j(2-\frac 2p)}
&\le &  2^{j\big[m(\frac 1p-\frac 12+\ve')+\frac 2p-2\big]}\, 2^{k(\frac 1p-\frac 12 +\ve'-\frac {2\tilde a}p+\tilde a)}.
\end{eqnarray*}
Thus, after summing first in $j\ge k\tilde a,$ we are left with the sum
\begin{eqnarray*}
\sum\limits_{k\ge 1} 2^{k\tilde a\big[m(\frac 1p-\frac 12+\ve')+\frac 2p-2\big]}\, 2^{k(\frac 1p-\frac 12 +\ve'-\frac {2\tilde a}p+\tilde a)}.
\end{eqnarray*}
From here on, we can argue as before, only with $a$ replaced by $\tilde a;$  note here that also  $\tilde a\ge a_1.$

\smallskip
This concludes the proof of estimate \eqref{tsuminjk1}.

\medskip
\subsection{$L^p$ - estimates of $\M_{j,k}^\la$  for the contribution by $E_l$ and $\tilde E_l$  when $\la<2^{j}$}\label{tLpElla<j}
\medskip
We shall study here both the situations in Step 1 as well as in Step 2. 
Let us begin by observing that the condition $\la<2^{j}$ implies that 
\begin{equation}\label{la<j}
\la d_{j,k} <2^{-j}\ll1.
\end{equation}

This implies by \eqref{L2Eljk}, respectively \eqref{tL2Eljk}, that here we have the $L^2$-estimate
\begin{equation}\label{L2Eljk2}
\|\M^\la_{j,k}\|_{2\to 2}\lesssim 2^{-j}.
\end{equation}

\bigskip

As for the $L^{1+\ve}$-estimates, let us first look at the {\bf Step 1-situation and the contribution by $E_l$:}
By \eqref{phaseF} and \eqref{la<j}, we may here assume that
$$
 \Phi_0^s(z,s_1,\si_2)- s_1y_1=-s_1y_1;
$$ 
the exponentials of all other terms of the phase can be included into the amplitude. Thus, 

\begin{eqnarray*}
F_{j,k}(\la,2^{-k}\si_2,s_3,y_1) &=& 2^{-j}\iint e^{-i\la s_3(\Phi_0^s(z,s_1,\si_2)- s_1y_1)} \chi_0(s_1)\chi_1(z)  \,\eta dz ds_1\\
&\sim&2^{-j}\vp(-\la s_3y_1),
\end{eqnarray*}
where $\vp:=\widehat {\chi_0}\in \S(\RR)$ is rapidly decaying.
This shows that, up to a small error term, we can reduce to the region where 
$$
|y_1|\le \la^{\de-1},
$$
where again we may later on choose  $\de>0$ as small as we like. But then we obtain
\begin{eqnarray}\nonumber
\mu_{j,k}^\la(y+\Ga)&=&\la^{\frac52}2^{-j-k}\\
&&\times \int e^{-i\la s_3(2^{-2k}\si_2^2B(2^{-k}\si_2)-2^{-k}\si_2 y_2-y_3)}\chi_1(\si_2) d\si_2 \chi_1(s_3)\,  \vp(-\la s_3y_1)ds_3.
\label{mujklastep22}
\end{eqnarray}

\medskip

{\bf Let us first assume that $\la\ge 2^k.$}
In  a similar way as in Subsection \ref{L1onEl} we  may then  decompose the remaining $y$-region into the subregions
\begin{eqnarray*}
\Delta_{\iota}&:=& \{y: |y_1|\le \la^{\de-1}, |y_2|\sim 2^{\iota-k}\}, \quad 1\ll \iota\le k,\\
\Delta_{0}&:=& \{y:|y_1|\le \la^{\de-1}, |y_2|\lesssim 2^{-k}\}.
\end{eqnarray*}
Then
$$
|\Delta_{\iota}|\lesssim \la^{\de-1} 2^{\iota -k}.
$$
And, as before,  if $\iota\gg1,$  an integration by parts in $\si_2$ shows that
$$
|\mu_{j,k,\iota}^\la(y+\Ga)|\lesssim \la^{\frac52}2^{-j-k} (\la 2^{\iota-2k})^{-1}=\la^{\frac 32} 2^{-j}2^{k-\iota}.
$$
Hence, if we choose $\de$ sufficiently small, 
$$
\|\M_{j,k,\iota}^\la\|_{1+\ve\to 1+\ve}\lesssim \la^{\frac 32+\ve/2}2^{-j} 2^{k-\iota}  \la^{\de-1} 2^{\iota -k}\lesssim\la^{\frac 12+\ve}2^{-j},
$$
which leads to the following analogue of \eqref{Mjkiota}:
\begin{equation*}\nonumber
\sum\limits_{\iota\le k}\|\M_{j,k,\iota}^\la\|_{1+\ve\to 1+\ve}\lesssim \la^{\frac 12+\ve}2^{-j}.
\end{equation*}

If $\iota=0,$ we can again apply van der Corput's estimate and obtain in a similar way that 
$
\|\M_{j,k,0}^\la\|_{1+\ve\to 1+\ve}\lesssim \la^{1+\ve}2^{-j-k}.
$
Altogether, we thus obtain the following analogue of \eqref{Mjk}:
\begin{equation*}
\|\M_{j,k}^\la\|_{1+\ve\to 1+\ve}\lesssim \la^{1+\ve}2^{-j-k}+ \la^{\frac 12+\ve}2^{-j}.
\end{equation*}

Interpolating the estimates \eqref{L2Eljk2} and \eqref{Mjk}, we find that for $1<p< 2,$
$$
\|\M_{j,k}^\la\|_{p\to p}\lesssim
 (\la^{1+\ve}2^{-k})^{\frac 2p -1} 2^{-j} 
+ (\la^{\frac 12+\ve})^{\frac 2p -1} 2^{-j}.
$$
As for the first term, summing over all  $k\ge 0$ and all dyadic $\la$s such that $\la<2^j,$ we are left with the sum
$$
\sum\limits_{j\ge 0} 2^{j\big((1+\ve)(\frac 2p-1)-1\big)},
$$
which is easily seen to be convergent if we choose $\ve>0$ sufficiently small.

As for the second term, summing first over all  all $k\le \log\la$  and then all dyadic $\la$s   that $\la<2^j,$ we are left with the sum
$$
\sum\limits_{j\ge  0} 2^{j\big((\frac 12+\ve)(\frac 2p-1)-1\big)},
$$
with an $\ve>0$ which may have to be chosen slightly bigger than before. Again, it is easily seen that the latter series is convergent if we choose $\ve>0$ sufficiently small, since $p>3/2.$ 

\medskip

{\bf Let us finally assume that $\la< 2^k.$}  Then the ``oscillatory integral'' $\mu_{j,k}^\la(y+\Ga)$ has indeed essentially no oscillation, so that we only get the estimates
$$
|\mu_{j,k}^\la(y+\Ga)|\lesssim \la^{\frac52}2^{-j-k}\quad\text{and}\quad |\Delta|\lesssim\la^{\de-1},  
$$
where $\Delta:=\{y:|y_1|\le \la^{\de-1}, |y_2|\ll 1\}.$ Thus,
\begin{equation*}\nonumber
\|\M_{j,k}^\la\|_{1+\ve\to 1+\ve}\lesssim \la^{\frac 32+\ve}2^{-j-k},
\end{equation*}
hence
$$
\|\M_{j,k}^\la\|_{p\to p}\lesssim (\la^{\frac 32+\ve}2^{-k})^{\frac 2p -1} 2^{-j}.
$$
Summing first over all $j$s such that $2^j>\la$ and all $k$s such that  $2^k>\la,$ we are left with the sum 
$$
\sum\limits_{\la\gg 1} \la^{\big( (\frac 32+\ve-1)(\frac 2p-1)-1\big)}.
$$
But, $(3/2-1)(2/p-1)-1=1/p-3/2<-5/6,$ since $p>3/2,$ so that the sum  is again convergent if we choose $\ve>0$ sufficiently small.

\bigskip
Let us next look at the {\bf Step 2-situation  and the contribution by $\tilde E_l$:}
By \eqref{tPhi0s} and \eqref{la<j}, we may here assume that
$$
\tilde\Phi_0^s(z,s_1,\si_2)- s_1y_1= s_12^{-ka}\ga(\si_2,Y_1)+\tilde h(s_2),
$$
where $z\sim 1\sim \si_2.$ Putting again $\tilde h^s(\si_2):=\tilde h(2^{-k}\si_2),$ we find that
\begin{equation}\nonumber
\tilde F_{j,k}(\la,2^{-k}\si_2,s_3,y_1) 
=2^{-j} e^{-i\la s_3\tilde h^s(\si_2)} \,\int \vp\big(s_3\la2^{-ka} \ga(\si_2,Y_1)\big)\chi_1(z)\eta dz,
\end{equation}
where $\vp:=\widehat {\chi_0}\in \S(\RR)$ is rapidly decaying and, by \eqref{tmujk},
\begin{eqnarray*}
\mu_{j,k}^\la(y+\Ga)&=&\la^{\frac52}2^{-k}\times \\
&&\iint  \tilde F_{j,k}(\la,2^{-k}\si_2,s_3,y_1) e^{-i\la s_3(s_2^2B(s_2)-2^{-k}\si_2 y_2-y_3)} \chi_1(\si_2) d\si_2 \chi_1(s_3) ds_3. 
\end{eqnarray*}
\medskip

{\bf Assume first that $\la2^{-ka}\ge 1.$} We then decompose in $y$ into the regions
\begin{eqnarray*}
\Delta_{\iota}&:=& \{y: |Y_1|\sim 2^{\iota}\}, \quad 1\ll \iota\le ka,\\
\Delta_{0}&:=& \{y:|Y_1|\lesssim 1\}, \end{eqnarray*}
and denote the corresponding contributions by $\mu_{j,k,\iota}^\la.$ 
Then 
$$
|\Delta_{\iota}|\lesssim 2^\iota 2^{-ka}.
$$
And, if $\iota\gg1,$ then clearly
$$
|\tilde F_{j,k,\iota}(\la,2^{-k}\si_2,s_3,y_1)|\lesssim 2^{-j}( \la2^{-ka}2^\iota)^{-1},
$$
hence 
$$
|\mu_{j,k,\iota}^\la(y+\Ga)|\lesssim \la^{\frac52}2^{-j-k}( \la2^{-ka}2^\iota)^{-1}.
$$
This implies that 
\begin{equation}\label{Mjkiota2}
\|\M_{j,k,\iota}^\la\|_{1+\ve\to 1+\ve}\lesssim \la^{\frac 32+\ve}2^{-j-k}.
\end{equation}

Finally, if $\iota=0,$ we can  estimate 
\begin{eqnarray*}
|\mu_{j,k,0}^\la(y+\Ga)|\lesssim\la^{\frac52}2^{-j-k} \int  \big|\vp\big(s_3\la2^{-ka} \ga(\si_2,Y_1)\big)\big|\chi_1(\si_2) d\si_2 \,\chi_1(s_3)\chi_1(z)dz ds_3. 
\end{eqnarray*}
But, 
$$
\int \big|\vp\big(s_3\la2^{-ka} (\si_2^ a\tilde w(\si_2)-Y_1)\big)\big|\chi_1(\si_2) d\si_2
\lesssim ( \la 2^{-ka})^{-1},
$$
since  $\si_2\sim 1,$ so that we may use $\tau_2:= \si_2^a\tilde w(\si_2)$ as a new variable of integration in place of $\si_2.$ Thus,
$$
|\mu_{j,k,0}^\la(y+\Ga)|\lesssim \la^{\frac52}2^{-j-k}( \la2^{-ka})^{-1},
$$
so that \eqref{Mjkiota2} remains valid also for $\iota=0.$ Interpolation with \eqref{L2Eljk2} leads to 
$$
\|\M_{j,k,\iota}^\la\|_{p\to p}\lesssim \big(\la^{\frac 32+\ve}2^{-k}\big)^{\frac 2p-1} 2^{-j}.
$$

Summing the estimate in  first in $\iota\le ka,$ then in $k\ge 0$ and next over all  $j$s such that $2^j>\la,$ we are left with the sum
$$
\sum\limits_{\la\gg1} \la^{(\frac 32+\ve)(\frac 2p-1)}\la^{-1}=\la^{\frac 3p-\frac 52+\ve'},
$$
where $\ve'>0$ can be chosen as small as we like. Since $p>3/2,$ we have that  $3/p-5/2<2-5/2<0,$ so that the remaining series in $\la$ is convergent.

\medskip

{\bf Assume next that $\la2^{-ka}<1.$} Here, we must again invoke Lemma \ref{thest}. According to this lemma, 
we have that for $\si_2\sim 1,$ 
$$
|\pa_{\si_2}^\al \tilde h^s(\si_2)|\lesssim 2^{-k(1+ma)}.
$$  
This implies that $\la 2^{-k(1+ma)}\ll \la 2^{-ka}<1,$ which shows that we can now include the factor $e^{-i\la s_3\tilde h^s(\si_2)}$ of $\tilde F_{j,k}(\la,2^{-k}\si_2,s_3,y_1)$ in the oscillatory integral defining  $\mu_{j,k}^\la(y+\Ga)$ into the amplitude. I.e., we may  simply assume that  
$\tilde h^s\equiv 0,$ and thus write
\begin{eqnarray*}
\mu_{j,k}^\la(y+\Ga)&=&\la^{\frac52}2^{-j-k}\times \\
&&\hskip-2cm \iint  e^{-i\la s_3(2^{-2k}\si_2^2B(2^{-k}\si_2)-2^{-k}\si_2 y_2-y_3)} \chi_1(\si_2)  \vp\big(s_3\la2^{-ka} \ga(\si_2,Y_1)\big)\eta d\si_2\, \chi_1(s_3) ds_3. 
\end{eqnarray*}
But note  that $\la 2^{-ka} \ga(\si_2,Y_1)=\la 2^{-ka}\si_2^ a\tilde w(\si_2)- \la y_1,$ where the first term is harmless, since  $\la 2^{-ka}\le 1.$
A comparison with the identity \eqref{mujklastep22} thus shows that we may now argue exactly as we did for the Step 1- situation.
\smallskip

This concludes our discussion of the case where $\la<2^j.$

\medskip
\subsection{$L^{p}$ - estimates for the contributions by domains  $D_{l_1}^c, c\in \mathcal C_{l_1},$ when the algorithm stops after step 1}\label{D1c}
\medskip
Assume now that $c=c^{\al_1}_{l_1}\in \mathcal C_{l_1}$ is such that  the cluster 
$\left[ \begin{matrix} 
\al_1\\
l_1
\end{matrix}\right ]
$
contains only one, real root $r_0$ of multiplicity 
$N_0:=N\left[ \begin{matrix} 
\al_1\\
l_1
\end{matrix}\right ].
$
The leading term of $r_0(s_2)$ is then given by $c^{\al_1}_{l_1}s_2^a,$ where again $a:=a_{l_1}.$ 
\smallskip

We can then {\it modify our change of coordinates by re-defining}
$$
\tilde z:=z_1-r_0(s_2),\qquad \text{and} \qquad \tilde\Phi(\tilde z,s_2):=\Phi(\tilde z+r_0(s_2),s_2).
$$
The domain corresponding to $D_{l_1}^c$ is then essentially again of the form
$$
\tilde D_{l_1}^c=\{(\tilde z, s_2):   |\tilde z|<  \ve s_2^{a_{l_1}}\},
$$
with $0<\ve\ll 1.$ We claim that  then also on $\tilde D_{l_1}^c$ we have a good resolution of singularities of $\tilde \Phi,$ namely \begin{equation}\label{Dlreso}
\tilde\Phi(\tilde z,s_2)=\tilde V(\tilde z,s_2)\, s_2^{\tilde A}\tilde z^{N_0},\qquad (\tilde z,s_2)\in \tilde D_{l_1}^c,
\end{equation}
for some $\tilde A\ge 1,$ where $\tilde V$ has similar properties as in \eqref{tElreso}.

\smallskip
Indeed, in analogy to \eqref{tildePhifac}, we have that 
\begin{equation*}
\tilde\Phi(\tilde z,s_2)=\tilde U_0(\tilde z,s_2)(\tilde z+r_0(s_2))^{\nu_1}  \prod_{l=1}^L
\Phi\left[ \begin{matrix} 
\cdot\\
l\end{matrix}\right ](\tilde z+r_0(s_2),s_2).
\end{equation*}
Since $|\tilde z|\ll |r_0(s_2)|,$ by modifying $\tilde U_0(\tilde z,s_2)$ appropriately, we may replace here 
$(\tilde z+r_0(s_2))^{\nu_1}$ by $s_2^{a\nu_1},$ as well as 
$\Phi\left[ \begin{matrix} 
\cdot\\
l\end{matrix}\right ](\tilde z+r_0(s_2),s_2)
$
by the following:

--  if  $l<l_1,$  then by
$
s_2^{a_l N\left[ \begin{matrix} 
\cdot\\
l\end{matrix}\right ]};
$

--  if  $l=l_1,$  then by
$
\prod\limits_{\al\ne \al_1} s_2^{a N\left[ \begin{matrix} 
\al\\
l_1\end{matrix}\right ]}\, \tilde z^{N_0};
$

--  if  $l>l_1,$  then by
$
 s_2^{a N\left[ \begin{matrix} 
\cdot\\
l\end{matrix}\right ]}.
$
\medskip

This shows that indeed \eqref{Dlreso} holds true.  Moreover, a quite similar reasoning shows that the $\ka^{l_1}$-principal part of 
$\tilde \Phi$  is given by 
\begin{eqnarray*}
\tilde \Phi_{\ka^{l_1}}={\rm constant} \, (\tilde z+cs_2^a)^{\nu_1}\prod\limits_{l<l_1} s_2^{a_l N\left[ \begin{matrix} 
\cdot\\
l\end{matrix}\right ]} \prod\limits_{\al\ne \al_1} \big(\tilde z-(c_{l_1}^{\al}-c_{l_1}^{\al_1})s_2^a\big)^{N\left[ \begin{matrix} 
\al\\
l_1\end{matrix}\right ]}\, \tilde z^{N_0}
 \prod\limits_{l>l_1}  (\tilde z+cs_2^a)^{ N\left[ \begin{matrix} 
\cdot\\
l\end{matrix}\right ]}.
\end{eqnarray*}
The term with the smallest power of $\tilde z$ in this $\ka^{l_1}$ - homogenous function agrees with  a constant multiple of $s_2^{\tilde A}\tilde z^{N_0},$ which shows that the Newton-Puiseux polyhedron of $\tilde\Phi$ must be of the type sketched in Figure \ref{stop}.
\smallskip

\begin{figure}[!h]
\centering
\includegraphics[scale=0.6]{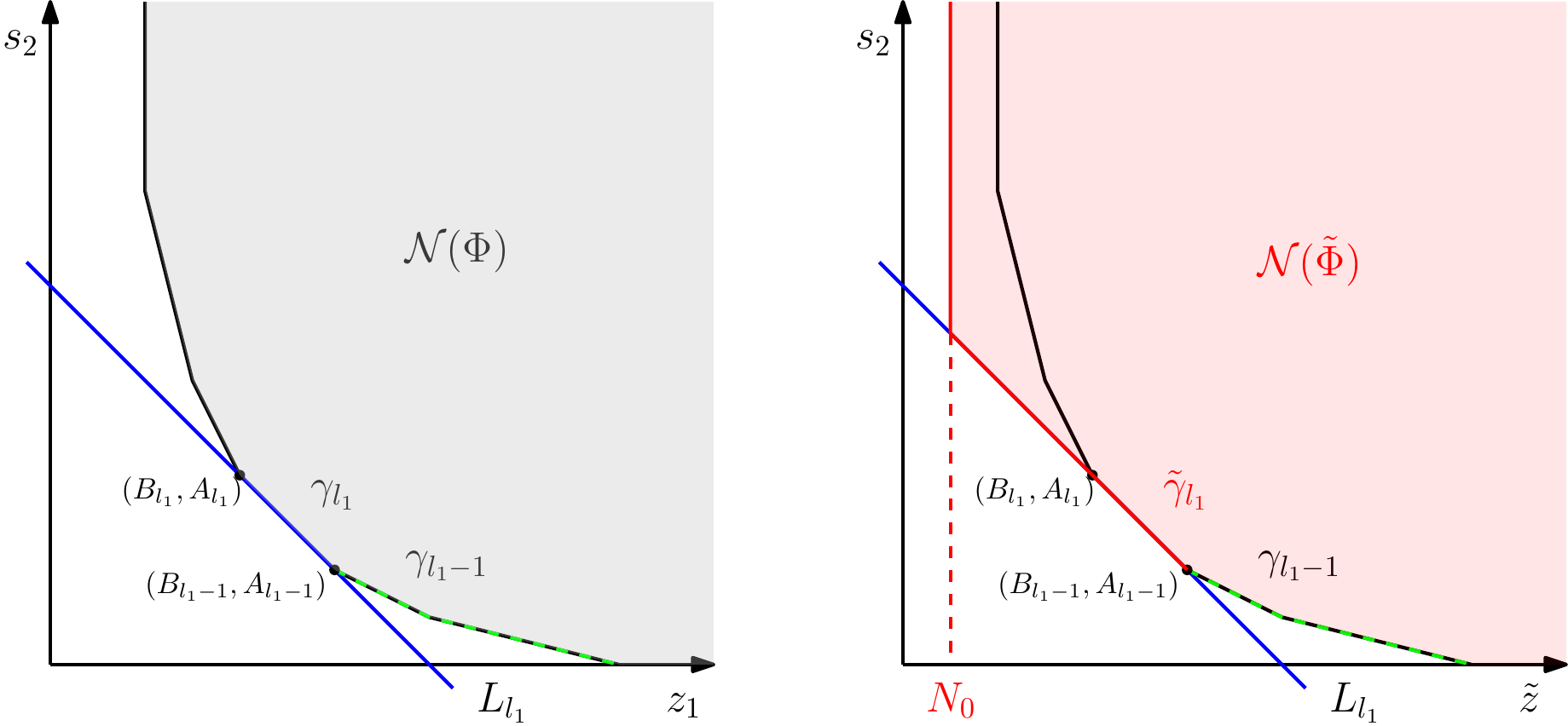}
  \caption{Stopping: comparison between $\N(\Phi)$ and $\N(\tilde\Phi)$}
  \label{stop}
\end{figure}

Consequently, we see that $\tilde L=l_1,$ $(\tilde B_{\tilde L},\tilde A_{\tilde L})=(N_0,\tilde A)$  and $\tilde a_{\tilde L}=a_{l_1}=a,$ if $N_0<B_{l_1-1},$ 
whereas $\tilde L=l_1-1,$ $(\tilde B_{\tilde L},\tilde A_{\tilde L})=(N_0,\tilde A)=(B_{l_1-1}, A_{l_1-1})$  and 
$\tilde a_{\tilde L}=a_{l_1-1},$ if $N_0=B_{l_1-1}.$
This shows that the region $ \tilde D_{l_1}^c$ is contained in the transition domain
$$
\tilde E_{\tilde L}:=\{(\tilde z,s_2):  |\tilde z|<2^{-M} s_2^{\tilde a_{\tilde L}}\}
$$
defined as in Subsection \ref{dyadictildez}, provided we choose $\ve\ll 2^{-M},$ so that the desired $L^p$-estimates for the contribution by $ \tilde D_{l_1}^c$ follow immediately from our discussions in the previous subsections.

\medskip
\subsection{On the contributions by the homogeneous domains $\tilde D_l$ }
\medskip
Recall that we are here in Case 2, i.e.,  $c\in \mathcal C_{l_1},$ and that more precisely  in Subcase (i)  we shall assume that $l\ge l_1+1\ge 2,$ where $l_1\ge 1,$ and in  Subcase (ii)  we assume that $l\ge l_1\ge 2.$ 

The following notation, which lends for a generalisation to higher steps of the resolution algorithm, will be come useful:
We shall write  $\Phi=:\Phi^{[1]}, D_l=:D_l^{[1]},$ and so forth, and 
$$
\tilde \Phi=:\Phi^{[2]},  \quad\tilde z=:z_2=z_1-c^{\al_1}_{l_1}s_2^a,  \quad\tilde D_l=:D_l^{[2]},
$$
and so forth, where as before $a=a_{l_1}.$

Recall next from Remark \ref{cinCii} (c)  that if we are on $D_l^{[2]}$ and 
if $\tilde a_l$ is the leading exponent of a root $\tilde r(s_2)$ of  $\Phi^{[2]}$ lying in  $D_l^{[2]},$ then $\tilde a_l$ is of the form $a_{l_1l_2}^{\al_1},$ and the leading term of $\tilde r(s_2)$ is of the form
$
c_{l_1l_2}^{\al_1\al_2}s_2^{a_{l_1l_2}^{\al_1}}.
$
\smallskip

Fixing from here on $l_2,$ by replacing $\Phi=:\Phi^{[1]}$ by $ \Phi^{[2]},$ and the set $\mathcal C_{l_1}$ by the set
$
\mathcal C^{\al_1}_{l_1,l_2}
$
of all {\it real}  coefficients $c_{l_1l_2}^{\al_1\al_2},$ 
we  can thus from here on proceed  as in Subsection \ref{Dlcont}, changing first 
from the variable $z_2$ to 
$$
\tilde z_2=z_3:=z_2-cs_2^{a_{l_1l_2}^{\al_1}}=z_1-\big(c^{\al_1}_{l_1}s_2^a+c s_2^{a_{l_1l_2}^{\al_1}}\big),
$$
in order to pass from homogeneous regions 
$$
D_l^{[2],c}:=\{(z_2,s_2):   |z_2-cs_2^{a_{l_1l_2}^{\al_1}}|<  \ve s_2^{a_{l_1l_2}^{\al_1}}\}, 
$$
to
$$
\tilde D^{[2],c}:=\{(z_3,s_2):   |z_3|<  \ve s_2^{a_{l_1l_2}^{\al_1}}\}, 
$$
and from $\Phi^{[2]}$ to $\tilde \Phi^{[2]}=: \Phi^{[3]},$ 
and so forth. 
Note also that if $c=c_{l_1l_2}^{\al_1\al_2}\in \mathcal C^{\al_1}_{l_1,l_2},$ then we can write 
$$
z_3=z_1-w(s_2),
$$
if we choose for $w(s_2)$ here 
$$
w(s_2):=c^{\al_1}_{l_1}s_2^a+c_{l_1l_2}^{\al_1\al_2} s_2^{a_{l_1l_2}^{\al_1}},
$$
i.e., the leading jet of the sub-cluster 
$\left[ \begin{matrix} \al_1&\al_2\\
l_1 &l_{2}
\end{matrix}\right ].$
Moreover, we clearly still have that
$$
|w(s_2)|\sim s_2^a
$$
(compare with Remark \ref{onw}).
\medskip

We are thus led to Step 3 of the resolution of singularities algorithm, in which we can apply the same arguments as in Step 2, now  applied to $\Phi^{[2]}$ in place of $\Phi,$ in order to derive the desired $L^p$-estimates for the associated  maximal operators.

\medskip
\section{Higher steps of the resolution algorithm and conclusion of the proof of Theorem \ref{thm-a-}}\label{endproof1}
\medskip

In the same way as we had passed from Step 1 to Step 2, and  from Step 2 to Step 3, we can now proceed recursively.
Suppose  we have already introduced the functions $\Phi^{[1]}, \dots,\Phi^{[p]}, p\ge 2,$ where
$$
\Phi^{[j]}(z_j,s_2)=\Phi(z_j+w_j(s_2),s_2) \qquad j=2,\dots, p,
$$
and where  $w_j(s_2)$ is the real,  leading jet of a sub-cluster 
$\left[ \begin{matrix} \al_1&\cdots &\al_{j-1}\\
l_1&\dots &l_{j-1}
\end{matrix}\right ].
$
Note that then $z_j=z_1-w_j(s_2).$ 
The transition domains associated to $\Phi^{[p]}$ can be handled as before in Step 2.
And, by narrowing down in a similar way as in the previous steps to a  homogeneous subdomain $ D^{[p+1],c},$ which corresponds to a narrow subdomain in the original coordinates $z_1$  containing only roots from a sub-cluster 
$
\left[ \begin{matrix} \al_1&\cdots &\al_{p}\\
l_1&\dots &l_{p}
\end{matrix}\right ]
$
of order $p$ of the previous sub-cluster
$
\left[ \begin{matrix} \al_1&\cdots &\al_{p-1}\\
l_1&\dots &l_{p-1}
\end{matrix}\right ]
$
of order $p-1$ 
whose  associated leading jet 
$$
w_{p+1}(s_2):=c_{l_1}^{\al_1}s_2^{a_{l_1}}+c_{l_1l_2}^{\al_1\al_2}s_2^{a_{l_1l_2}^{\al_1}}+\cdots+
c_{l_1\cdots l_p}^{\al_1\cdots \al_p}s_2^{a_{l_1\cdots
l_{p}}^{\al_1\cdots \al_{p-1}}}
$$
is real, we can  introduce the next function 
$$
\Phi^{[p+1]}(z_{p+1},s_2):=\Phi(z_{p+1}+w_{p+1}(s_2),s_2).
$$
Note that then 
$
z_{p+1}=z_p-c_{l_1\cdots l_p}^{\al_1\cdots \al_p}s_2^{a_{l_1\cdots
l_{p}}^{\al_1\cdots \al_{p-1}}},
$ 
and 
$$
\Phi^{[p+1]}(z_{p+1},s_2)=\Phi^{[p]}(z_{p+1}+c_{l_1\cdots l_p}^{\al_1\cdots \al_p}s_2^{a_{l_1\cdots
l_{p}}^{\al_1\cdots \al_{p-1}}},s_2).
$$
We have thus arrived at  Step $p+1$ and can bootstrap.
\smallskip

In this way, we shall construct nested sequences of  subclusters
$
\Big\{\left[ \begin{matrix} \al_1&\cdots &\al_{p}\\
l_1&\dots &l_{p}
\end{matrix}\right ]\Big\}_p.
$
The corresponding sequences of multiplicities of these sub-clusters are  then decreasing and must therefore eventually become constant, say  after  step $p=P,$ and we are reduced to  a narrow homogeneous domain.  But, arguing as in Subsection \ref{stopstep1}, then the sub-cluster of order $P$ can  contain only one, real root, of some multiplicity $N_0,$ and we can estimate the contribution by  this narrow homogeneous domain in the same way as we did in Subsection \ref{D1c}.

\smallskip
This shows that our resolution algorithm will always stop  after a finite number of steps; this completes the proof of Theorem \ref{thm-a-}.

\color{black}

\setcounter{equation}{0}

\section{ First steps in the proof of Theorem \ref{thm-a+}  }\label{proof+}

In this section, we assume that $\phi$ is of type $\A_{n-1}^+\setminus\A_{n-1}^e,$  with $n\ge 4.$ Since the coordinates $\x$ in the normal form \eqref{AD} are in this case already adapted to 
$\phi$ (because $n<2m),$  we shall here not really  make use of this normal form, but instead use  Property $(\A_{n-1}^+)$ in these linearly adapted coordinates $\x.$

\subsection{Stationary phase in $x_2$}
We again  first apply the method of stationary phase to the $x_2$-integration in \eqref{mulathat}, which  leads to 
\begin{eqnarray}\nonumber
\widehat{\mu^\la}(\xi)
&= &e^{-is_3\la} \chi_0(s_3s')\chi_1(s_3) \big[ \la^{-1/2} \int_{\bR} e^{-i\la s_3(s_1x_1+\phi(x_1, x_2^c(x_1,s_2)))}\tilde \eta(x_1,  s_2)\,dx_1\\
&&\hskip 8cm+r(\la,s) \big],\label{9mulahat}
\end{eqnarray}
with a slightly modified cut-off function $\chi_1,$  
where $\tilde \eta$ is another smooth bump function supported in a
sufficiently small neighborhood of  the origin,  $r(\la,s)$ is a
remainder term of order
$$
r(\la,s)=O(\la^{-\frac32})\quad \mbox{as} \quad  \la\to+\infty,
$$
and $x_2^c(x_1, s_2)$ denotes  the unique (non-degenerate)
critical point  of the  phase $\phi$ with respect to $x_2.$ Then, similarly  as before,  $\phi(x_1, x_2^c(x_1,s_2))=\breve\phi(x_1,s_2)$ is the Legendre transform of $\phi.$

Following the proof of Theorem \ref{equlind}, let us choose line-adapted coordinates  $(y_1,y_2)=(x_1-\alpha (x_2), x_2),$ in which $\phi\x$ is given by $\tilde \phi\y,$ and  decompose 
$$
\tilde \phi\y=\tilde\phi(0,y_2)+\tilde\phi_{\rm red}\y.
$$ 

We then denote by $\tilde \ga_1$ the first non-vertical edge of $\N(\tilde\phi_{\rm red}),$ which has $(n,0)$ as its left endpoint, and choose the weight $\tilde \ka^e =(\tilde \ka^e _1,\tilde\ka^e _2)=(1/n,\tilde\ka^e _2)$ so that $\tilde \ga_1$ is contained in the line $\tilde \ka^e _1t_1+\tilde \ka^e _2t_2=1.$

 Assume first  that  $\tilde \ka^e _2>0.$   Then, as we have seen,
we may decompose 
\begin{equation*}
\tilde\phi_{\rm red}(y_1,y_2)= p(y_1,y_2) +\tilde\phi_{\rm err}(y_1,y_2),
\end{equation*}
where $ p=p_{\tilde e_1}$ is a $\tilde\ka^e $-homogeneous polynomial of $\tilde\ka^e$-degree $1$ with Taylor support on the first compact edge $\tilde \ga_1$ of $\N(\tilde\phi_{\rm red}),$ consisting of at least two monomial terms, one of them being of the form $cy_1^{n},$   and where $\tilde\phi_{\rm err}$ consists of terms of higher $\tilde\ka^e $-degree.

\medskip
If we put $z_1:=x_1-\tilde\al(s_2),$  with $\tilde\al(s_2):= \al(s_2w_0(s_2)),$  then we had seen in  the proofs of Lemma \ref{legendre} and Theorem \ref{equlind} that  the new coordinates  $(z_1,s_2)$  are line-adapted to $\breve\phi(x_1,s_2).$ Moreover,  if we assume  without loss of generality (for simplicity) that $w_0(0)=1,$ then if   $\breve\phi(x_1,s_2)$ is represented  by the function $\breve\phi^{\rm la}(z_1,s_2)$ in these line-adapted coordinates, then by  \eqref{legendrephila} we have the following analogue of \eqref{brevepad1}:
\begin{equation*}
\breve\phi^{\rm la}(z_1,s_2)=s_2^2B(s_2)+ \breve\phi_1(z_1, s_2)
\end{equation*}
where
\begin{equation}\label{9newphi}
\breve\phi_1(z_1, s_2)=p(z_1,s_2) +\breve\phi_{\rm err}(z_1,s_2),
\end{equation}
with the same polynomial $p$ as in \eqref{defp} and  an error term $\breve\phi_{\rm err}$ consisting  of terms of  $\tilde\ka^e $-degree $>1,$ and where $p(0,s_2)\equiv 0\equiv \breve\phi_{\rm err}(0,s_2).$
Note that this {\it defines} here the function $\breve\phi_1 $ for the present setting.
\smallskip

On the other hand, if $\tilde \ka^e _2=0,$  then we have seen in the proof of Theorem \ref{equlind} that $\N(\breve\phi^{\rm la}_{\rm red})$ is contained in the region where $t_1\ge n.$ Since $\phi$ is analytic, this implies that 
\begin{equation}\label{9legendrephila3}
\breve\phi^{\rm la}(z_1,s_2)=s_2^2B(s_2)+ \breve\phi_1(z_1, s_2)=s_2^2B(s_2)+z_1^n \big(G(s_2) +\breve\phi_{\rm err}(z_1,s_2)\big), 
\end{equation}
where   $G$ and $ \breve\phi_{\rm err}$ are analytic  and  $G(0)\ne 0,$ whereas $\breve\phi_{\rm err}(0,s_2)\equiv 0.$

 \begin{remark}\label{ka2=0}
The identity \eqref{9legendrephila3} shows that  the function $\breve\phi^{\rm la}$  has already a good resolution of singularity when   $\tilde \ka^e _2=0,$ and that here $n_e=n.$ We shall therefore assume that  $\tilde \ka^e _2>0$ in the subsequent sections,  and  only briefly comment on the case $\tilde \ka^e _2=0$ in Subsection \ref{kae2=0}. 
\end{remark}

\color{black}
The contribution of the error term $r(\la,s)$ to $\widehat{\mu^\la}$ in \eqref{9mulahat}
and   the corresponding maximal operator are again easily estimated, and we shall henceforth ignore this error term.

After changing from the coordinate $x_1$ to the coordinate $z_1$ in the integral in \eqref{9mulahat}, we then we may finally assume that 
\begin{eqnarray*}\nonumber
\widehat{\mu^\la}(\xi)
&=& \la^{-\frac 12} \chi_0(s_3s')\chi_1(s_3) e^{-i\la s_3\big(s_2^2B(s_2)+s_1\tilde\al(s_2)+1\big)}\\
&&\hskip3cm  \times \int_{\bR} e^{-i\la s_3(s_1z_1+ \breve\phi_1(z_1, s_2))} \eta\,dz_1 \\ 
&=& \la^{-\frac 12} \chi_0(s_3s')\chi_1(s_3) e^{-i\la s_3\big(s_2^2B(s_2)+s_1\tilde\al(s_2)+1\big)}J(\la, s),\nonumber
\end{eqnarray*}
where again
\begin{equation*}
J(\la, s):=\int e^{-i\la s_3(\breve\phi_1(z_1, s_2)+s_1z_1)}  \eta\,dz_1.
\end{equation*}
In order to defray the notation, we have simply dropped the tilde from $\eta.$
Introducing also 
\begin{eqnarray}\nonumber
 F(\la,s_2,s_3,y_1) &:=&\int  J(\la, s) e^{i\la s_3s_1(y_1-\tilde\al(s_2))} \chi_0(s_3s')\chi_0(s_1)\,ds_1\\  \label{9defiF}
  &=&\iint e^{-i\la s_3\big(\breve\phi_1(z_1, s_2)+s_1(z_1+\tilde\al(s_2)-y_1)\big)}  \eta(z_1,  s_2)  \chi_0(s_3s')\chi_0(s_1)\,  \eta dz_1 ds_1,
 \end{eqnarray}
by  Fourier inversion and a change from the coordinates $\xi$ to $s,$ we then  have that (with slightly modified functions $\chi_0, \chi_1$)
\begin{eqnarray*}\nonumber
\mu^\la(y+\Ga)
&=&\la^{\frac52}
 \iint F(\la,s_2,s_3,y_1) e^{-i\la s_3(s_2^2B(s_2)-s_2 y_2-y_3)}\\ 
 &&\hskip2cm\times  \chi_0(s_2) \, \eta ds_2 \chi_1(s_3) ds_3,
\end{eqnarray*}
where we have  again put $\Gamma:=(0,0,1).$

\medskip

Observe that \eqref{9newphi} again allows to write $\breve\phi_1(z_1, s_2)$ in a similar way as we did this in \eqref{legendamins}:
\begin{equation}\nonumber
\breve\phi_1(z_1, s_2)=z_1^n\beta(z_1)+s_2z_1^{m'}\omega(z_1)+s_2^2z_1 q(z_1, s_2),
\end{equation}
where $x_1^{m'}\omega(z_1)$ stands for a finite type function, if $m'$ is finite, and a flat function, if $m'=\infty.$ Note also that here $\tilde \ka^e _2+m'/n\ge 1.$ 

\begin{remark}[On the notion of effective multiplicity for  $\A^-_{n-1}$-singularities]\label{9effectivem}

 A comparison between the identities \eqref{legendamins} and \eqref{9newphi} for $\breve\phi_1$ suggests that at this stage we could as well have 
introduced the notion of effective multiplicity $n_e$ for functions $\phi$ of type $\A^-_{n-1},$  namely by means of  the identity \eqref{legendamins} for $\breve\phi_1.$  Indeed, for such $\phi,$ the edge $\breve\ga_1$ of the Newton polyhedron $\N(\breve\phi_1)$ with right vertex $(n,0)$ lies on the line $\ka^e_1t_1+\ka^e_2t_2=1,$ with the weight 
$$
\ka^e:=\big(\frac 1n, \frac {n-m} n\big),
$$
and if we define the number $n_e$ in same way  by means of the edge $\ga_1$ as we did this for functions of type 
$\A^+_{n-1}$ in Subsection \ref{A+}, we find that here $n_e=m,$ since the point $(m,1)$ lies on the edge $e_1.$ The according exponent $p_e$ would then be given by $p_e=2m/(m+1).$ But, note that here $p_e\le h=2n/(n+2),$ since $2m\le n,$ so $p_e$ is not really relevant for $\A^-_{n-1}$-type singularities.  This was the reason why we had introduced the notion of effective multiplicity only for $\A^+_{n-1}$-singularities, for which  in addition a direct definition of $n_e$ in terms of Newton polyhedra associated to $\phi$ is available.
\end{remark}

Recall, however, that  in contrast to Section \ref{proof-}, we now have 
\begin{equation}\label{kae2<1}
\tilde\ka^e_2<1/2.
\end{equation}
Recall also  that since $\phi$ is of type $\A_{n-1}^+\setminus\A_{n-1}^e,$  $p$ does satisfy the following two assumptions:

{\begin{itemize}
\item[(A1)] $\pa_{z_1}^2 p(z_1,s_2)$ consists of  at least two distinct monomial terms, one of them of course being $ n(n-1)z_1^{n-2}.$
\item[(A2)] If $\pa_{z_1}^2 p(z_1,s_2)$  does not  vanish of maximal possible order $n-2$ along a real, non-trivial root of $\pa_{z_1}^2 p(z_1,s_2);$
more precisely,  $\pa_{z_1}^2 p(z_1,s_2)$ is not  of the form 
\begin{equation}\label{maxvan2}
\pa_{z_1}^2 p(z_1,s_2)=n(n-1)(z_1-cs_2^d)^{n-2},
\end{equation}
with $c\in \RR\setminus \{0\}$ and an integer exponent $d\in \NN_{\ge 1}.$  
 \end{itemize}
}

Let us also recall from Remark \ref {pe-h} that 
\begin{equation}\label{nne}
n<2n_e, \quad\text{in particular}\quad n_e>2.
\end{equation}

\medskip
With our new function $\breve\phi_1$ as defined in \eqref{9newphi}, we  can now  essentially follow the proof of Theorem \ref{thm-a-}, however, with  some important  modifications and  changes, and we shall therefore mostly concentrate on those.  

To this end, we shall adopt  the same notation, putting for instance again (compare \eqref{Phi0}, 
 \eqref{Phi})
$$
\Phi_0(z_1,s'):=\breve\phi_1(z_1, s_2)+s_1z_1, \qquad \Phi(z_1,s_2):=\pa_{z_1}^2\Phi_0(z_1,s') =\pa_{z_1}^2\breve\phi_1(z_1, s_2),
$$
and shall use the same notation for the associated  oscillatory integrals $J$ and $F,$ Newton polyhedra, and so forth.

It is important to observe that condition (A1) ensures that the Newton polyhedra of $\breve\phi_1$ and of $\Phi$ are closely related:
\smallskip

$\N(\Phi)$ is obtained by translating $\N(\breve\phi_1)$  by $-2$ in horizontal direction and then considering only the part with $t_1\ge 0$ (compare Figure \ref{on-ne}).

\begin{figure}[!h]
\centering
\includegraphics[scale=0.37]{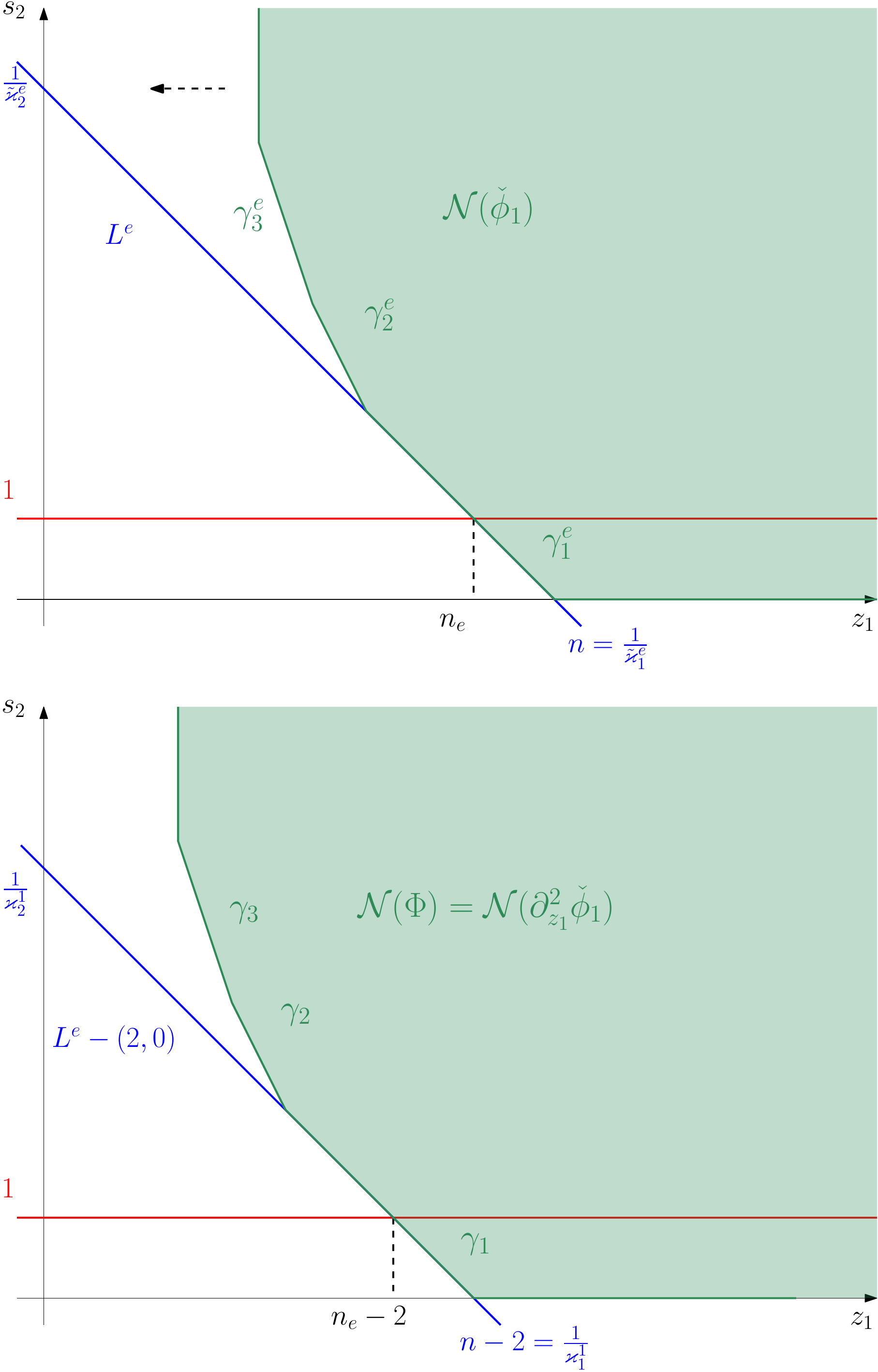}
  \caption{Comparison between $\N(\breve\phi_1)$ and $\N(\Phi)$}
  \label{on-ne}
\end{figure}

In particular, since $n_e>2,$ the first edge $\ga_1$ of $\N(\Phi)$, whose endpoints are  $(B_0,A_0)=(n-2,0)$ and $(B_1,A_1),$ where $A_1\ge 1,$ and the first edge $\ga^e_1$ of $\N(\breve\phi_1),$ whose right endpoint is $(n,0),$ are having the same  modulus of slope
\begin{equation}\label{9a1}
a_1=\frac {\ka_1^1}{\ka^1_2}=\frac 1{n-n_e}.
\end{equation}
Actually, a similar result holds true also for the other edges  of $\N(\phi).$ To this end, let us denote by $\breve \ka^l$ the weight associated to the $l$th edge $\ga^e_l$ of $\N(\breve \phi_1).$ 
Then 
$$
a_l=\frac {\ka_1^l}{\ka^l_2}=\frac {\breve\ka_1^l}{\breve\ka^l_2}\ge a_1, \qquad l=1,\dots, L.
$$
In particular, in a similar way as  in Subsection \ref{geolemma}, by  Lemma \ref{key} this will now imply (for Step 1) that
\begin{equation}\label{9leqm}
\frac{A_l-1}{a_l}+B_l+2\le n_e
\end{equation}
for $l=1,\dots, L.$  
This is the analogue of the estimate \eqref{leqm}, which had been crucial for the proof of Theorem \ref{thm-a-}.
\smallskip

{\it One should therefore basically think of simply replacing $m$ by $n_e$ in our  proof of Theorem \ref{thm-a-} in order to get a proof  of Theorem \ref{thm-a+}. However, this is, of course, an oversimplification, and  there are also some important difference to be noticed, on which we shall  concentrate subsequently.}

\medskip

\setcounter{equation}{0}

\section{Analysis in  Step 1 of the resolution algorithm  }

\subsection{$L^2$ - estimation of $\M_{j,k}^\la$ for the contribution by $E_l$}
\medskip
Since our arguments in Subsection \ref {L2onEl} did not make use of the special form of $\breve\phi_1,$  the $L^2$-estimates  
\eqref{L2Eljk} remain valid also here.

\medskip
\subsection{$L^{1+\ve}$ - estimates of $\M_{j,k}^\la$ for the contribution by $E_l, l=1,\dots,L,$ when $\la\ge 2^j$}

\subsubsection{The case where $\tilde\al(s_2)\not\equiv 0$} This case behaves more like a simplified version of the Step 2 - situation. We shall therefore defer its discussion to the end of Section \ref{9step2}.
 
\smallskip

\subsubsection{The case where $\tilde\al(s_2)\equiv 0$} 

Here, we can basically follow our discussion of the Step 1 - situation in Section \ref{step1}.
Observe first that all estimates up to \eqref{Mjk>1} will again remain valid. However, after having applied Lemma \ref{pinvers}, the identity \eqref{applemma5.2} for $F_{j,k,|Y_1|\lesssim1}(\la,s_2,s_3,y_1)$ makes  explicit use of  the form of the phase $\breve\phi_1,$ as well as the oscillatory integral giving $\mu_{j,k,|Y_1|\lesssim1}^\la(y+\Ga).$ We shall see that this part will indeed require a different discussion here. The complete phase is now given by
\begin{eqnarray*}
&& 2^{-2k}\big[\si_2^2 B(2^{-k}\si_2) -2^{k}\si_2 y_2 \big]  +\breve\phi_1(y_1,2^{-k}\si_2) -y_3\\
&=&2^{-2k}\big[\si_2^2 B(2^{-k}\si_2) -2^{k}\si_2 y_2 \big]  +p(y_1,2^{-k}\si_2) +\breve\phi_{\rm err}(y_1,2^{-k}\si_2)-y_3.
\end{eqnarray*}
Recall also that here $Y_1:=2^jy_1,$ so that the condition $|Y_1|\lesssim 1$ is equivalent to $|y_1|\lesssim 2^{-j}.$

But, since $p$ is $\tilde\ka^e$-homogeneous of degree $1,$ we have 
\begin{equation}\label{pdil}
p(y_1,2^{-k}\si_2)=2^{-k/\tilde \ka^e_2}p(2^{ka_1}y_1,\si_2). 
\end{equation}
And, by \eqref{kae2<1}, we may write $2^{-k/\tilde \ka^e_2}=2^{-\ve_0 k}2^{-2k},$ for some $\ve_0>0.$ Moreover, 
$$
|2^{ka_1}y_1|=|2^{-j+ka_1}Y_1|\lesssim 2^{-j+ka_1}\le 2^{-(j-ka_l)},
$$ 
and since we are on $E_l,$  by \eqref{Ejk1} we have that $j-ka_l\gtrsim 1.$
Thus, $|2^{ka_1}y_1|\lesssim 1.$ 
These estimates show that the term $p(y_1,2^{-k}\si_2) +\breve\phi_{\rm err}(y_1,2^{-k}\si_2)$ can be absorbed into  the term 
$2^{-2k}\si_2^2 B(2^{-k}\si_2),$ i.e., we may assume that the complete phase is of the form
\begin{equation*}
2^{-2k}\big[\si_2^2 B(2^{-k}\si_2,Y_1) -2^{k}\si_2 y_2 \big]  -y_3,
\end{equation*}
where $B(2^{-k}\si_2,Y_1)$ has similar properties as  $B(2^{-k}\si_2),$ in particular $|B(2^{-k}\si_2,Y_1)|\sim 1.$ 
\smallskip

From here on, we can thus again follow the arguments from Subsection \ref{L1onEl}, if we assume there that $\om(y_1)\equiv 0,$ and arrive at  the same $L^{1+\ve}$-estimates as in that subsection.

\medskip
\subsection{$L^p$ - estimation  of $\M^{\tau_l}$  for the contribution by $E_l,l=1,\dots,L,$ when $\la\ge 2^j$}\label{9LpEl}
\medskip
The previous discussions show that  also the $L^p$-estimates \eqref{Mlajkp}  remain valid. Since, for the summation in $\la$ of these estimates, only the condition $p>3/2$ was needed, and since the subsequent summation over all $k\le j/a,$  with $a:=a_l\ge a_1$  works as well, 
 we can thus eventually reduce to proving  that the estimate
\begin{equation}\label{9crucsum}
\sum\limits_{j\gg1}\Big( 2^{j\big[(\frac{A-1}a+B+2)(\frac 2p-1+\ve')-1\big]}
+  2^{j\big[\frac Aa+B+2)(\frac 1p-\frac 12+\ve')-1\big]}\Big)<\infty
\end{equation}
still  holds true. 

But, by \eqref{9leqm}, 
$$
\frac{A-1}a+B+2\le n_e \quad \text{and} \quad \frac 2p-1< \frac 1{n_e}, 
$$ 
since $1/p<1/p_e=1/2+1/(2n_e),$ so that we may assume that the exponent of the first summand in \eqref{9crucsum} is strictly negative, and thus we can sum in $j.$

Similarly, by \eqref{9a1} and \eqref{nne},
$$
\frac{A}a+B+2\le n_e +\frac 1a\le n_e+\frac 1{a_1}=n_e+(n-n_e)=n\quad \text{and} \quad \frac 1p-\frac 12<\frac 1{2n_e}<\frac 1n, 
$$ 
so that we may assume that the exponent of the second summand in \eqref{9crucsum} is strictly negative  too and we can again  sum in $j.$

\medskip
\subsection{$L^p$-estimation of $\M^{\tau_0}$ when $\la\ge 2^j$}\label{9Lpl0}
\medskip
Here we need to argue in a different way in order to see that the $L^{1+\ve}$- estimate \eqref{Mjk} still remains true.

Since $p$ is a $\tilde\ka^e$-homogeneous polynomial function, we can write  it in the form
$$
p(y_1,s_2)=y_1^n\beta(y_1)+s_2\tilde \rho(y_1) +s_2^2\tilde q(y_1,s_2),
$$
where $\tilde \rho$ and $\tilde q$ are $\tilde\ka^e$-homogeneous polynomial functions. Suppose we had $\tilde q(0,0)\ne 0,$ so that $\tilde q$ would be a constant function. Then $s_2^2\tilde q(y_1,s_2)$ would be $\tilde\ka^e$-homogeneous of degree $2\tilde\ka^e_2=1,$ which would contradict \eqref{kae2<1}.
Hence we must have $\tilde q(0,0)=0.$

Essentially in the same way we can show that 
$$
\breve \phi_1(y_1,s_2)=y_1^n\beta(y_1)+s_2 \rho(y_1) +s_2^2q(y_1,s_2),
$$
with analytic functions $\rho$ and $q$ such that $q(0,0)=0.$ Thus, the complete phase can be written as
$$
2^{-2k}\big[\si_2^2 B(2^{-k}\si_2) +\si_2^2q(y_1,2^{-k}\si_2)-2^{k}\si_2 (y_2-\rho(y_1) \big]  +y_1^n\beta(y_1)  -y_3,
$$
and since we may assume that $|q|\ll 1,$ the term $\si_2^2q(y_1,2^{-k}\si_2)$ can again be absorbed into the first term. Thus, we can indeed now 
argue as in Subsection \ref{L1onEl}, simply replacing the term $y_1^m\om(y_1)$ in this argument by $\rho(y_1).$
\smallskip

 As in  Subsection \ref{Lpl0}, what then remains to show is that 
\begin{equation*}
\sum\limits_{j,k\gg1:  k\ge j/a_1}\Big(( 2^{jn-k})^{\frac 2p-1+\ve'}\, 2^{-j} 
+  ( 2^{jn})^{\frac 1p-\frac 12+\ve'}\, 2^{-j}\Big)<\infty.
\end{equation*}
But, we  can now essentially follow the arguments from  Subsection \ref{Lpl0}. 
For instance, in the first summand, we can first sum over all $k\ge j/a_1$ and arrive at the sum
$$
\sum\limits_{j\gg 1} 2^{j\big[(n-\frac 1{a_1})(\frac 2p-1+\ve')-1\big]}.
$$
But, 
$$
n-\frac 1{a_1}=n_e, \quad \text{and} \quad \frac 2p-1< \frac 1{n_e},
$$
so that by choosing $\ve'$ sufficiently small we can indeed  sum in $j.$ The remaining arguments from Subsection \ref{Lpl0} can easily be adapted in a similar way, since for $p>p_e$ we still have 
$$
n (\frac 1p-\frac 12)<n\frac 1{2n_e}<n\frac 1n=1.
$$
\medskip

Before we come to discuss the contributions by the homogeneous domains $D_l,$ we need to  prove a crucial lemma on multiplicities of roots.

\medskip

\subsection{On the condition $\tilde A_1\ge 1:$ a multiplicity lemma}
\medskip

\begin{lemma}[Multiplicity Lemma]\label{multi}
Let $P(y_1,y_2)$ be a polynomial function which is $\ka^1$-homogeneous  of degree 1 and satisfies $P(y_1,0)=y_1^{n-2},$  where $0<\ka^1_2\le1.$ Define 
$n_e$ as in Figure \ref{multiplicity}, i.e.,
$$
n_e-2=\frac{1-\ka^1_2}{\ka^1_1}.
$$
Note that then $a:=\ka^1_1/\ka^1_2$ is the  modulus of the slope of the line $L_1.$ 
Assume that $P$ has a non-trivial real root of multiplicity strictly greater then $n_e-2.$ 

Then, either $P$ is of the form
\begin{equation}\label{Pfac}
P(y_1,y_2)=(y_1-\la y_2^p)^{n-2}, \quad\text{where}\quad p\in\NN_{>0}, \la\in \RR\setminus\{0\},
\end{equation}
or all non-trivial real roots of $P$ have multiplicity 1, and, more precisely, $P(y_1,y_2)$ is of the form
$$
P(y_1,y_2)=y_1^{\nu_1} (y_1^q-\la y_2^p), \quad \text{where}\quad p,q \in\NN_{>0}, q\ge 2, \frac pq=a, \la\in \RR\setminus\{0\}.
$$
\end{lemma}

\begin{figure}[!h]
\centering
\includegraphics[scale=0.4]{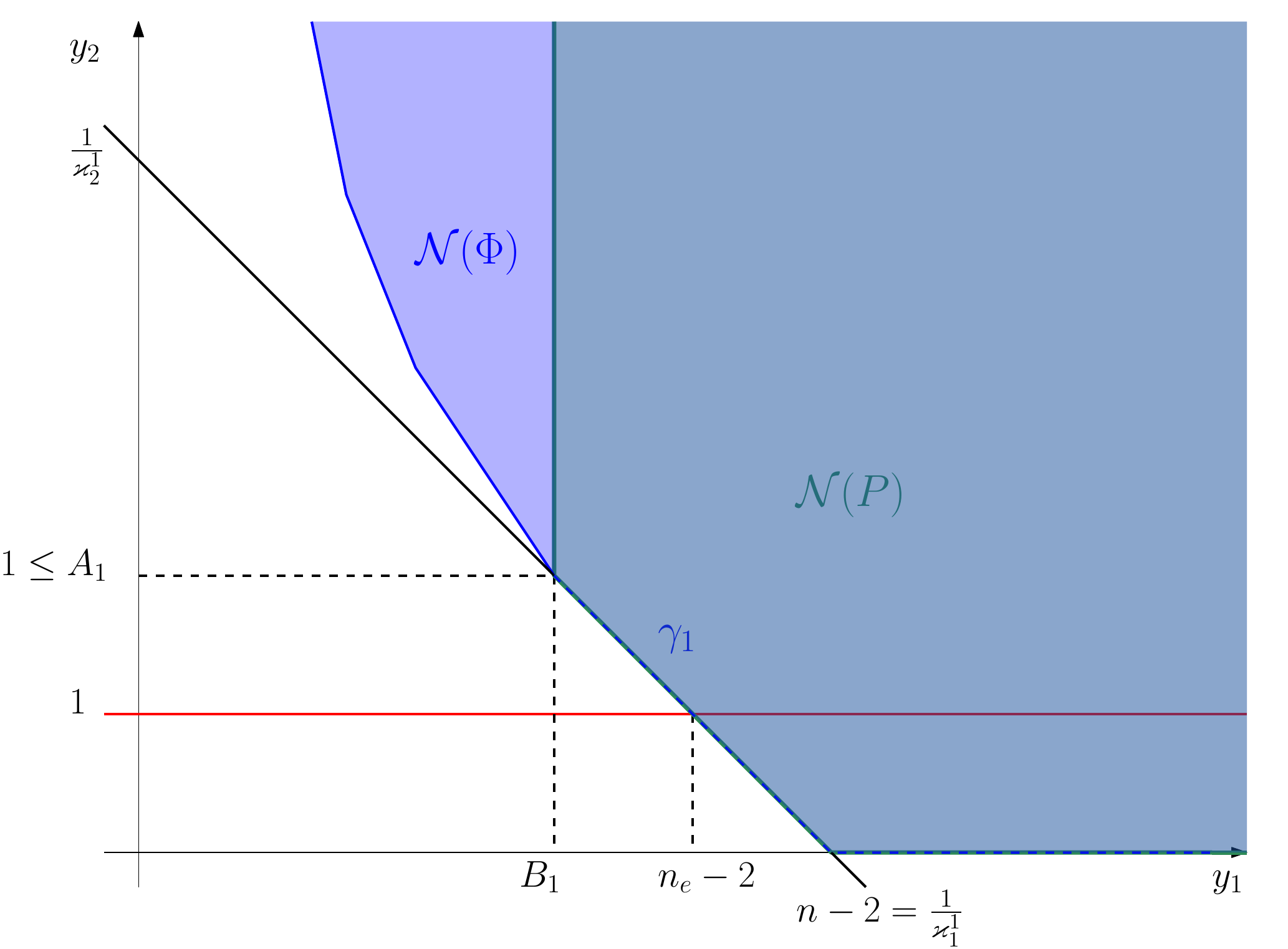}
  \caption{}
  \label{multiplicity}
\end{figure}

\begin{proof} 
According to \cite[Proposition 2.2]{IM-ada}, 
$P$ can be written in the form
\begin{equation}\label{Pfac2}
P(y_1,y_2)= y_1^{\nu_1}\prod_{k} (y_1^q-\la_k y_2^p)^{n_k},
\end{equation}
with distinct $\la_k\in \CC\setminus\{0\}$ and multiplicities $n_k\in \bN\setminus\{0\},$ with  $\nu_1\in\NN.$
\smallskip

Let us put $N:=\sum\limits_k n_k,$ so that $n-2=\nu_1+Nq,$  and note that $a=p/q.$ Note also that the  monomial of lowest degree in $y_1$ in $P$ is then of the form $cy_1^{\nu_1} y_2^{N p},$ which implies that $(B_1,A_1)=(\nu_1,Np).$ But then
$$
\ka^1_1\nu_1+\ka^1_2 Np=1, \quad \ka^1_1=a\ka^1_2,
$$
hence
$
\ka^1_2= 1/(a\nu_1+Np).
$
This implies that
\begin{equation}\label{ne-2}
n_e-2=\frac{1-\ka^1_2}{\ka^1_1}=\frac{1/\ka^1_2-1}a=\frac {a\nu_1+Np-1}a.
\end{equation}

Now, by our assumption, there is some $k'$ such that $\la_{k'}\in \RR\setminus\{0\}$ and $n_{k'}>n_e-2.$ In particular, \eqref{ne-2} implies that 
\begin{equation}\label{aN}
N\ge n_{k'} >\nu_1+\frac {Np-1}a,
\end{equation}
hence
\begin{equation}\label{paN}
(p-a)N<1-a\nu_1.
\end{equation}
\smallskip

\noi{\bf (i) If $q=1,$} then $a=p,$ so that by \eqref{paN} we have  $p\nu_1<1,$ hence $\nu_1=0.$ Thus, by \eqref{Pfac2}
$$
P(y_1,y_2)= \prod_{k} (y_1-\la_k y_2^p)^{n_k}.
$$
But, by \eqref{aN}, $n_{k'}>N-1/p,$ hence $n_{k'}=N=n-2.$ This implies 
$$
P(y_1,y_2)=(y_1-\la_{k'} y_2^p)^{n-2}.
$$

\smallskip

\noi{\bf (i) If $q\ge 2,$} then by \eqref{paN},
$$
N<\frac {1-a\nu_1}{p-a}\le \frac 1{p-a}=\frac 1{p(1-1/q)}\le \frac 1{p/2}=\frac 2p\le 2.
$$
Thus, $N=1,$ and without loss of generality we may assume that $N=n_1=1.$ This implies that
$$
P(y_1,y_2)= y_1^{\nu_1} (y_1^q-\la_1 y_2^p),
$$
hence all non-trivial real  roots of $P$ are simple.
\end{proof} 

\medskip
\subsection{On the contributions by the homogeneous domains $D_l, l=1,\dots,L$}\label{9Dlcont}
\medskip
As in Subsection \ref{9Dlcont}, the case $l=1$ requires special considerations. To this end, 
recall that the $\ka^e$-principal part of $\breve\phi_1(z_1,s_2)$ is the polynomial $p(z_1,s_2),$  and accordingly
$$
P(z_1,s_2):=\pa_{z_1}^2p(z_1,s_2)
$$
is the $\ka^1$-principal part of $\Phi(z_1,s_2)=\pa_{z_1}^2\breve\phi_1(z_1,s_2),$  which is associated to the first edge $\ga_1$ of $\N(\Phi).$
Moreover, if $a=a_1,$ then every non-trivial  real root of $P$ will be of the form $cs_2^a,$  with $c\ne 0.$ 

Now, suppose again that we change in a narrow homogeneous neighborhood $D_1^c$ of this root to the coordinates $\tilde z:=z_1-cs_2^a,$ so that we pass from $\Phi$ to $\tilde \Phi(\tilde z,s_2):=\Phi(\tilde z+cs_2^a,s_2),$  and from $P$ to $\tilde P(\tilde z,s_2):=P(\tilde z+cs_2^a,s_2).$ 
\smallskip

Observe first that Subcase (ii) of Case 2 from Subsection \ref{9Dlcont} will appear  here if and only if $\tilde P$ will be of the form $c_0z_1^{n-2},$ i.e., if the edge $\ga_1$ will ``shrink'' to its right vertex $(\tilde B_0,\tilde A_0)=(B_0, A_0)=(n-2,0)$  under the change of coordinates. But, this would mean that 
$$
P(z_1,s_2)=n(n-1)(z_1-cs_2^a)^{n-2}.
$$
However, this is excluded by our assumption (A2) (cf. \eqref{maxvan2}).
\smallskip

Thus, again, only Subcase (i)  of Case 2 can appear (compare Remark \ref{cinCii} (a)), so that the first compact edge  $\tilde\ga_1$ of  $\N(\tilde\Phi)$ 
lies on the same line $L_1$ as $\ga_1.$ Denote again by $(\tilde B_0,\tilde A_0)=(B_0, A_0)=(n-2,0)$ its right endpoint, and by  $(\tilde B_1,\tilde A_1)$ its left endpoint. Then $\tilde \ga_1$ is associated to $\tilde P.$ 

\smallskip

Assume now that we had $\tilde \nu_1=\tilde B_1>n^e-2,$ where $\tilde \nu_1$ denotes the power of $y_1$ of the first factor in \eqref{Pfac2}, if we look at 
$\tilde P$ in place of $P.$ But that would mean that $P(z_1,s_2)$ would have a  non-trivial real root of order  strictly greater than $n^e-2,$  so that we could   apply Lemma \ref{multi} to $P(z_1,s_2)$  and conclude that  either 
$$
P(z_1,s_2)=(z_1-c s_2^p)^{n-2},
$$
(compare \eqref{Pfac}), or  all non-trivial real roots of $P$ have multiplicity 1. The first case is, however, again excluded by assumption  (A1).

\medskip

In conclusion, we see that either all of the non-trivial real roots of $P$ have multiplicity  1 (as it had been in Section \ref{step1}), or we must have 
$\tilde B_1\le n^e-2.$ But, this is equivalent to having $\tilde A_1\ge 1$ (compare Figure \ref{multiplicity}).

\smallskip
From here on, all of the remaining arguments from Section \ref{step1} carry over to our present situation.

\setcounter{equation}{0}

\section{ Analysis in Step 2 of the resolution algorithm  }\label{9step2}
Fixing again $l_1\in\{1,\dots,L\}$ as in Section \ref{step2}, putting  accordingly 
$$
a:=a_{l_1}\ge a_1,
$$  and proceeding in the same way as before,  we are again led to considering  changes of coordinates of the form 
$$
\tilde z:=z_1-w(s_2), 
$$
where $w(s_2)=cs_2^a$ (or a full root $r_0$ with this leading term). The function $\Phi(z_1)$ will then lead to the function $\tilde \Phi(\tilde z).$ Note, however, one difference in the effect on the function $F:$

Due to the presence of the extra term $\tilde\al(s_2)$  in the phase for  $F(\la,s_2,s_3,y_1)$ in \eqref{9defiF}, we should  here indeed better  define 
$\Phi_0(x_1,s'):=\breve\phi_1(x_1, s_2)+s_1(x_1+\tilde \al(s_2))$ , and therefore 
\begin{eqnarray*}
\tilde\Phi_{0,0}(\tilde z, s_2)&:=& \breve\phi_1(\tilde z+ w(s_2), s_2),\\
\tilde\Phi_{0}(\tilde z, s')&:=& \Phi_{0}(\tilde z+ w(s_2),s_2)=\tilde\Phi_{0,0}(\tilde z, s_2)+s_1\tilde z+s_1\big(w(s_2)+\tilde \al(s_2)\big).
\end{eqnarray*}

\noi We shall therefore also introduce the function
$$
W(s_2):=w(s_2)+\tilde \al(s_2),
$$
besides the function $w(s_2),$ in Step 2, and similarly  also in every higher order step.
\smallskip

Now, either $W(s_2)\equiv 0,$ but then we can proceed in Step 2 (or higher) essentially  in the same way as in Step 1, so that we shall ignore this case in the sequel.
\smallskip

Or, $W(s_2)\not\equiv 0,$ and then
$$
|W(s_2)|\sim s_2^{a'} 
$$
for some rational $a'>0,$ since $W(s_2)$ is fractionally analytic. 

We remark that the value of $a'$ can depend on the  step of the resolution algorithm, since in higher steps of our  algorithm we shall have to replace the function $w(s_2)$ by the leading jet of some  higher-order sub-cluster of roots.  Observe also: since $\tilde \al(s_2)$ is analytic, we have that 
$|\tilde \al(s_2)|\sim s_2^{a_0}$ for some integer $a_0\ge 1.$  The following lemma is then immediate:

 \begin{lemma}\label{ona'}
Either $a'\ge 1,$ or $a'<1,$ and then $a<1$ and hence $1>a'=a\ge a_1.$
\end{lemma}

Next,  as in Subsection \ref{dyadictildez}, we can reduce to estimating the contributions by the $\tilde \ka^l$- homogeneous  domains  
$\tilde D_l$ of the form
$$
 \tilde D_l:=\{(\tilde z,s_2):   2^{-M} s_2^{\tilde a_l}\le |\tilde z|\le 2^{M} s_2^{\tilde a_l}\}, \quad l\le \tilde L,
$$
 and  the transition domains of the form
\begin{eqnarray*}
\tilde E_l&:=&\{(\tilde z,s_2):   2^M s_2^{\tilde a_{l+1}}<|\tilde z|<2^{-M} s_2^{\tilde a_l}\}, \quad l\le \tilde L-1,\\
\tilde E_{\tilde L}&:=&\{(\tilde z,s_2):  |\tilde z|<2^{-M} s_2^{\tilde a_{\tilde L}}\} .
\end{eqnarray*}

More precisely, in Case 1 where $c\notin \mathcal C_{l_1}, $ according to Remark \ref{cnotinC} we just need the transition domain $\tilde E_{\tilde L},$ where $\tilde L=l_1.$  

In Case 2, where $c\in \mathcal C_{l_1},$ we need to distinguish two subcases: if we are in  Subcase (i), then   
according to  Remark \ref{cinCi}   (a) we shall need the transition domains $\tilde E_l$ with $l\ge l_1\ge 1$ and the homogeneous domains  $\tilde D_l$ with $l\ge l_1+1\ge 2.$ Moreover,  in analogy with  Remark \ref{onstop} (a), we may and shall   treat  the case where $l_1=1$ and where  all non-trivial real roots of $P$ have multiplicity 1 separartely,
 so that we shall here even assume that $l_1\ge 2.$

In Subcase (ii) of Case 2, according to Remark \ref{cinCii} we may assume that $l_1\ge 2$ and shall then need the transition domains $\tilde E_l$ with $l\ge l_1-1\ge 1$ and the homogeneous domains  $\tilde D_l$ with $l\ge l_1\ge 2.$

\smallskip

Moreover, our discussion in Subsection \ref{9Dlcont} shows that under these assumptions we may assume that  $\tilde A_1\ge 1,$ 
and that the edges $\tilde \ga_1$ and $\ga_1$ are supported on the same line $\tilde L_1=L_1$ of ``slope'' $\tilde a_1=a_1,$ so that, by Lemma \ref{key}, we now have  (in analogy with  the identities \eqref{keyestl=1} and  \eqref{keyestlg1}) that 

 \begin{equation}\label{9keyestl=1}
\frac{\tilde A_1-1}{\tilde a_1}+\tilde B_1+2=n_e,
\end{equation}
(even when $\tilde A_1<1$),  and \begin{equation}\label{9keyestlg1}
\frac{\tilde A_l-1}{\tilde a_l}+\tilde B_l+2\le n_e,\quad  l=1,\dots, \tilde L
\end{equation}
(compare with \eqref{9leqm}).

Having fixed $l_1$ and $l$  according to these assumptions, we shall put 
$$
 \tilde a:= \tilde a_l.
$$
Note that if we are in Case 1, or Subcase (i) of Case 2, then $\tilde a\ge \tilde a_{l_1}=a,$ but in Subcase (ii) it may also happen that $\tilde a<a$ (for instance if $l=l_1-1,$ where $\tilde a=a_{l_1-1}<a_{l_1}=a$).

\medskip
\subsection{Resolution of singularity on the transition domain $\tilde E_l$}

Fix $\tilde E_l$ with $l\ge l_1,$ respectively $l\ge l_1-1\ge 1$  if we are in the Subcase (ii) of Case 2. Arguing as in Subsection \ref{tElresolution}, we here find that

\begin{equation}\label{9tPhi0onEl}
\tilde\Phi_{0}(\tilde z, s')=\tilde U(\tilde z,s_2)s_2^{\tilde A_l} \tilde z^{\tilde B_l+2}+ s_1(\tilde z+W(s_2))+\tilde z \tilde g(s_2)+\tilde h(s_2).
\end{equation}
And, after applying the change of  coordinates $\tilde z=2^{-j} z, \ s_2=2^{-k} \si_2$  putting  $A:=\tilde A_l, B:= \tilde B_l$ and 
$$
d_{j,k}:=2^{-kA-j(B+2)},
$$
assuming that \eqref{tEjk1} holds true, we obtain
$$
\tilde\Phi_0^s(z,s_1,\si_2)- s_1y_1=d_{j,k} U^s(z,\si_2) \si_2^{A} z^{B+2}+ s_1(2^{-j} z+W(s_2)-y_1)+2^{-j} z \tilde g(s_2)
+\tilde h(s_2),
$$
where $U$ has similar properties as in \eqref{Phi0s}. Let us then write
$$
W(s_2)=2^{-ka'} \si_2^{a'}\tilde W(\si_2),
$$
so that $|\tilde W|\sim 1,$ whereas all derivatives of $\tilde W$ can be assumed to be very small. We shall use the following  abbreviations:
\begin{eqnarray}\nonumber
Y_1:&=& 2^{ka'} y_1,\\
\gamma(\si_2,Y_1)&:=&\si_2^{a'}\tilde W(\si_2)-Y_1,  \label{9Y1K}\\
K&:=&2^{j-ka'}.\nonumber
\end{eqnarray}
These allow us  to re-write again as in \eqref{tPhi0s} and \eqref{tFjk}
\begin{eqnarray*}\nonumber
&&\tilde\Phi_0^s(z,s_1,\si_2)- s_1y_1\\
&=&d_{j,k} U^s(z,\si_2) \si_2^{A} z^{B+2}+ 2^{-j} s_1\big(K\ga(\si_2,Y_1)+z\big)+2^{-j} z \tilde g(s_2)+\tilde h(s_2)
\end{eqnarray*}
and 
\begin{eqnarray*}\nonumber
&&\tilde F_{j,k}(\la,2^{-k}\si_2,s_3,y_1) = 2^{-j}\iint e^{-i\la s_3(\tilde\Phi_0^s(z,s_1,\si_2)- s_1y_1)} \chi_0(s_1)\chi_1(z)  \,\eta dz ds_1\\ 
&&=2^{-j} e^{-i\la s_3\tilde h(s_2)} \,\int \vp\big(s_3\la2^{-j} [K\ga(\si_2,Y_1)+z]\big)\\ \nonumber
&&\hskip5cm \times e^{-i\la s_3\big(d_{j,k} U^s(z,\si_2) \si_2^{A} z^{B+2} +2^{-j}z \tilde g(2^{-k}\si_2)\big)}\chi_1(z)\eta dz,
\end{eqnarray*}
where $z\sim 1\sim \si_2,$ and where $\vp:=\widehat {\chi_0}\in \S(\RR)$ is rapidly decaying.

\subsection{$L^2$ - estimation of $\M_{j,k}^\la$ for the contribution by $\tilde E_l$}
\medskip
The $L^2$-estimates \eqref{tL2Eljk} again remain valid here.

\subsection{$L^{1+\ve}$ - estimates of $\M_{j,k}^\la$ for the contribution by $\tilde E_l$ when $\la\ge 2^j$}\label{9L1ontEl}
\medskip

Again, we can essentially follow the discussions from Subsection \ref{tElresolution}, and shall only highlight places where the arguments require  modifications. Recall from \eqref{tgh} that
$$
\tilde h(s_2)= \breve\phi_1(w(s_2), s_2), \quad  \tilde g(s_2)=\pa_{z_1}\breve\phi_1(w(s_2), s_2).
$$
Thus,  if we  put again
$$
\tilde g^s(\si_2):=\tilde g(2^{-k}\si_2),\quad  \tilde h^s(\si_2):=\tilde h(2^{-k}\si_2),
$$
we see that 
$$
\tilde h^s(\si_2)= \breve\phi_1(2^{-ka}\si_2^a\tilde w(\si_2), 2^{-k}\si_2), \quad  \tilde g(\si_2)=\pa_{z_1}\breve\phi_1(2^{-ka}\si_2^a\tilde w(\si_2), 2^{-k}\si_2).
$$
But,  \eqref{pdil} implies that here
$$
p(2^{-ka}\si_2^a\tilde w(\si_2), 2^{-k}\si_2)=2^{-k/\tilde \ka^e_2}p(2^{-k(a-a_1)}\si_2^a\tilde w(\si_2),\si_2), 
$$
where $a-a_1\ge 0,$ and, by \eqref{kae2<1}, $2^{-k/\tilde \ka^e_2}=2^{-\ve_0 k}2^{-2k},$ for some $\ve_0>0.$ 

Similarly, since $\pa_{y_1}p$ is $\tilde\ka^e$-homogeneous of degree $1-\tilde \ka^e_1,$ we see that 
$$
\pa_{y_1}p(2^{-ka}\si_2^a\tilde w(\si_2), 2^{-k}\si_2)=2^{-k\frac {1-\tilde \ka^e_1}{\tilde \ka^e_2}}\pa_{y_1}p(2^{-k(a-a_1)}\si_2^a\tilde w(\si_2),\si_2).
$$
And, if we again assume that $j\gtrsim k\tilde a,$ recalling that $\tilde a\ge a_1,$ we see that also 
$$
2^{-j}2^{-k\frac {1-\tilde \ka^e_1}{\tilde \ka^e_2}} \lesssim  2^{-k/\tilde \ka^e_2}.
$$
 These estimates imply that
 $$
 \tilde h^s(\si_2)=O(2^{-\ve_0 k}2^{-2k}), \quad 2^{-j} \tilde g^s(\si_2)=O(2^{-\ve_0 k}2^{-2k}) \text{   for  } j\gtrsim k\tilde a,
 $$
 which means that {\it we can absorb these terms (as functions of $\si_2$) into  the term $s_2^2B(s_2)=2^{-2k}\si_2^2B(2^{-k}\si_2)$ of the amplitude of the oscillatory integral \eqref{tmujk}.}
\medskip

From here on, we can simply  follow the arguments for the  reduction to the case $K|\ga(\si_2,Y_1)|\lesssim 1$ in Subsection 
\ref {Kgagg1}, distinguishing now between the {\bf Subcase 2.a: $\la 2^{-ka'}\lesssim 1$}, and 
 {\bf Subcase 2.b: $\la 2^{-ka'}\gg  1,$} and replacing accordingly $a$ by $a'$ in these steps of proof. 
 
 Note that the main difference compared to the previous discussion lies in how we get rid of the terms $\tilde h^s(\si_2)$ and 
 $2^{-j} \tilde g^s(\si_2),$ which turns out to be even a bit easier here than in Subsection \ref {Kgagg1}.
\smallskip

Next, we can also easily adapt the arguments which led to \eqref{applemma5.2} and the identity
 \begin{eqnarray*}
\mu_{j,k,K|\ga|\lesssim1}^\la(y+\Ga)&=&\la^{\frac32}2^{-k}
\int  e^{-i\la s_3\big(\breve\phi_1(y_1, 2^{-k}\si_2)+2^{-2k}\si_2^2B(2^{-k}\si_2)-2^{-k}\si_2 y_2-y_3\big)} \\
&&\hskip2cm \times \chi_1(\si_2) \chi_1(s_3) \, \chi_0(K\ga(\si_2,Y_1))\, \eta d\si_2ds_3.
\end{eqnarray*}
But, due to \eqref{pdil}, we may absorb the term $\breve\phi_1(y_1, 2^{-k}\si_2)$ into $2^{-2k}\si_2^2B(2^{-k}\si_2)$ within the amplitude, so that without loss of generality we may assume that 
 \begin{eqnarray*}
\mu_{j,k,K|\ga|\lesssim1}^\la(y+\Ga)&=&\la^{\frac32}2^{-k}
\int  e^{-i\la s_3\big(2^{-2k}\si_2^2B(2^{-k}\si_2)-2^{-k}\si_2 y_2-y_3\big)} \\
&&\hskip2cm \times \chi_1(\si_2) \chi_1(s_3) \, \chi_0(K\ga(\si_2,Y_1))\, \eta d\si_2ds_3.
\end{eqnarray*}
This is essentially as in Subsection \ref{Kgagg1}, if we assume there that $\om(y_1)\equiv 0,$ and replace $a$ by $a'$ in these arguments.

\medskip
In this way,  combining all these estimates, we see that we here obtain the following analogue to estimate \eqref{tMjk}:
\begin{equation*}
\|\M_{j,k}^\la\|_{1+\ve\to 1+\ve}\lesssim \la^{1+\ve}2^{-j-k}+  \la^{\frac 12+\ve}2^{-j} +\la^{\frac 12+\ve}2^{-k}+\la^{\frac 12+\ve}2^{-ka'}.
\end{equation*}

By means of interpolation with our $L^2$ - estimate  \eqref{tL2Eljk} and summation in $\la$ we arrive then at the following analogue to 
\eqref{tsuminla1}:
Arguing as in Subsection \ref{LpEl}, we  see that we can sum over all dyadic $\la$s and obtain that, for every $\ve'>0,$ 
\begin{eqnarray}\nonumber
\sum\limits_{\la\gg 1}\|\M_{j,k}^\la\|_{p\to p}&\lesssim& (d_{j,k}^{-1}2^{-k})^{\frac 2p-1+\ve'}\,2^{-j} 
+  (d_{j,k}^{-1})^{\frac 1p-\frac 12+\ve'}\, 2^{-j}\\  \label{9tsuminla1}
&+&(d_{j,k}^{-1})^{\frac 1p-\frac 12+\ve'}\, 2^{-k(\frac 2p-1)}2^{-j(2-\frac 2p)}\\ 
&+&(d_{j,k}^{-1})^{\frac 1p-\frac 12+\ve'}\, 2^{-ka'(\frac 2p-1)}2^{-j(2-\frac 2p)}. \nonumber
\end{eqnarray}
Putting again  $A:=\tilde A_l, B:=\tilde B_l,$ we recall from \eqref{9keyestl=1} and \eqref{9keyestlg1} that 
$$
\frac{A-1}{\tilde a}+ B+2\le n_e, \quad  \text{and} \quad d_{j,k}^{-1}=2^{kA+j(B+2)}.
$$
Moreover, since we are on $\tilde E_l,$ by \eqref{tEjk1} we have $j\ge k\tilde a.$ 
\smallskip

This shows that we can argue in exactly the same way as in Subsection \ref{geolemma}  in order to show that we can sum the first two terms on the right-hand side of \eqref{9tsuminla1} over all $k$s and $j$s such that $j\ge k\tilde a.$

\smallskip 
We are thus left with showing that 
\begin{equation}\label{9tsuminjk1}
\sum\limits_{j,k\gg1:  j\ge k \tilde a}\Big((d_{j,k}^{-1})^{\frac 1p-\frac 12+\ve'}\, 2^{-k(\frac 2p-1)}2^{-j(2-\frac 2p)}
+  (d_{j,k}^{-1})^{\frac 1p-\frac 12+\ve'}\, 2^{-ka'(\frac 2p-1)}2^{-j(2-\frac 2p)}\Big)<\infty,
\end{equation}
provided $\ve'>0$ is sufficiently small.

The first sum in \eqref{9tsuminjk1} can be estimated in the same way as in  Subsection \ref{geolemma}  - we just have to replace $m$ be $n_e$ in the argument.

\smallskip
As for the second sum, note that the terms in the second sum can be estimated by the terms of the first sum, provided $a'\ge 1.$ 
\medskip

This leaves as with the case where $a'<1.$ Then, by Lemma \ref{ona'}, $a'=a\ge a_1,$  and as before we can distinguish between the cases $\tilde a\ge a,$ and $\tilde a<a.$ 
\smallskip

In the first case, we can follow the estimates  from Subsection \ref{geolemma}, with $m$ replaced by $n_e,$ and only need to verify that $(n-n_e)a\ge 1.$ But, $ a\ge a_1=1/(n-n_e)$ by \eqref{9a1}, so that the desired estimate holds true.
\smallskip

In the second case where  $\tilde a<a,$  we can again argue as in Subsection \ref{geolemma}, since we also have that $\tilde a\ge a_1.$ 

This completes the proof of \eqref{9tsuminjk1}.

\medskip
\subsection{$L^p$ - estimates of $\M_{j,k}^\la$  for the contribution by $E_l$ and $\tilde E_l$  when $\la<2^{j}$}
\medskip

Arguing in a similar way as in Subsection \ref{tLpElla<j}, we again  have the  $L^2$-estimate \eqref{L2Eljk2}.

\medskip

As for the $L^{1+\ve}$-estimates, let us first look at the {\bf Step 1-situation and the contribution by $E_l$.}  Again, we shall first assume that $\tilde\al(s_2)\equiv 0.$ 
\smallskip

Arguing in a similar way as in Subsection \ref{tLpElla<j}, we may here again assume that
$$
 \Phi_0^s(z,s_1,\si_2)- s_1y_1=-s_1y_1.
$$ 
We can therefore adopt our discussion from Subsection \ref{tLpElla<j} verbatim,  in which we only made use of the assumption $p>3/2.$

\bigskip
Let us next look at the {\bf Step 2-situation  and the contribution by $\tilde E_l$:}
By \eqref{9tPhi0onEl}, we can here assume that
$$
\tilde\Phi_0^s(z,s_1,\si_2)- s_1y_1= s_12^{-ka'}\ga(\si_2,Y_1)+\tilde h(s_2),
$$
where $z\sim 1\sim \si_2,$  with $Y_1,K$ and $\ga(\si_2,Y_1)$ now defined by \eqref{9Y1K}. But, the term  $\tilde h(s_2)$ can again be absorbed into  the term $s_2^2B(s_2)$ and shall henceforth be dropped.

We then  find that
\begin{equation}\nonumber
\tilde F_{j,k}(\la,2^{-k}\si_2,s_3,y_1) 
=2^{-j} \,\int \vp\big(s_3\la2^{-ka'} \ga(\si_2,Y_1)\big)\chi_1(z)\eta dz,
\end{equation}
where $\vp:=\widehat {\chi_0}\in \S(\RR)$ is rapidly decaying, and
\begin{eqnarray*}
\mu_{j,k}^\la(y+\Ga)&=&\la^{\frac52}2^{-k}\times \\
&&\iint  \tilde F_{j,k}(\la,2^{-k}\si_2,s_3,y_1) e^{-i\la s_3(s_2^2B(s_2)-2^{-k}\si_2 y_2-y_3)} \chi_1(\si_2) d\si_2 \chi_1(s_3) ds_3. 
\end{eqnarray*}

\medskip

{\bf Assume first that $\la2^{-ka'}\ge 1.$}  We can here argue in the same way as in Subsection  \ref{tLpElla<j}, only with $a$ replaced by $a',$ since we only made use of the assumption $p>3/2.$

\medskip

{\bf Assume next that $\la2^{-ka'}<1.$}  Since we  had already seen that we may  assume that $\tilde h^s\equiv 0,$ we can again follow the reasoning from Subsection  \ref{tLpElla<j}, and  write
\begin{eqnarray*}
\mu_{j,k}^\la(y+\Ga)&=&\la^{\frac52}2^{-j-k}\times \\
&&\hskip-2cm \iint  e^{-i\la s_3(2^{-2k}\si_2^2B(2^{-k}\si_2)-2^{-k}\si_2 y_2-y_3)} \chi_1(\si_2)  \vp\big(s_3\la2^{-ka'} \ga(\si_2,Y_1)\big)\eta d\si_2\, \chi_1(s_3) ds_3. 
\end{eqnarray*}
But note  that $\la 2^{-ka'} \ga(\si_2,Y_1)=\la 2^{-ka'}\si_2^ {a'}\tilde W(\si_2)- \la y_1,$ where the first term is harmless, since  $\la 2^{-ka'}\le 1.$
Thus we may again argue as in the Step 1 situation discussed before.
\smallskip

This concludes our discussion of the case where $\la<2^j.$

\subsubsection{Completion of Step 1: the case where $\tilde\al(s_2)\not\equiv 0$} This case can be handled like a Step 2 case, where 
formally $a=+\infty$ and $W(s_2):=\tilde\al(s_2),$ so that $a'=a_0\ge 1,$  since $a_0\in \NN.$ We leave the details to the interested reader.

\subsection{On the contributions by the homogeneous domains $\tilde D_l$ }\label{9tDlcont}
\medskip

These can easily  be handled in analogy to our discussions in Subsection \ref{9tDlcont}, with the obvious modifications of the functions $W(s_2)$ as indicated at the beginning of Section \ref{9step2}.

\medskip

\subsection{ The case where $\tilde\ka^e_2=0$} \label{kae2=0} In view of  Remark \ref{ka2=0}, this case behaves very much like the cases  explained before  where we would have stopped our resolution algorithm, after applying a change of coordinates by subtracting a complete real roots. It can therefore be handled in a similar way as we did in such situations. We leave the details to the interested reader.

\medskip
\section{Higher steps of the resolution algorithm and conclusion of the proof of Theorem \ref{thm-a+}}\label{9endproof1}
\medskip
Here, we can again just proceed as we did in the corresponding Section \ref{endproof1} for the proof of Theorem \ref{thm-a-}; this completes the proof of Theorem \ref{thm-a+}.

\medskip
\section{Proof of the Iosevich-Sawyer-Seeger conjecture for surfaces of class $\A^e_{n-1}$}\label{isasee}
\medskip
In this final section, we shall briefly sketch how the proof of Theorem \ref{thm-a+} can be modified in order to prove the following
\begin{prop}\label{prop-ae}
Assume that $S$ is the graph of $1+\phi,$ and accordingly $x^0=(0,0,1),$  where $\phi$ is  analytic and of type 
$\A_{n-1}^e$. Then, if  the density $\rho$ is supported in a sufficiently small neighborhood  of $x^0,$  the  condition  
$p>\max\{ 3/2, h, {2n}/(n+1)\}=max\{ 3/2,  {2n}/(n+1)\}$  is  sufficient for $\M$ to be $L^p$-bounded. 
\end{prop}

Details are again left to the interested reader.  Recall  that this result will complete the proof of the Iosevich-Sawyer-Seeger conjecture for arbitrary analytic hypersurfaces in $\RR^3.$ 

\smallskip
\begin{proof}

We again assume that $(y_1,y_2)$ are line-adapted coordinates in which \eqref{defp} and \eqref{extype} hold true, so that 
$$
p\y=y_1^n +c_1 y_1 y_2^{\beta+2}, \qquad \text{with} \ \beta\in\NN.
$$
Since $\pa_{y_1}^2 p\y=n(n-1) y_1^{n-2}, $ we now find that the first edge $\ga_1$ of $\N(\Phi)$ has a modulus of slope $a_1'$ such that $+\infty\ge a_1'>a_1=1/(n-n_e).$ This first edge will intersect the line $s_2=1$ at a point $(n'_e-2,1),$ 
so that $a'_1=1/(n-n'_e),$ and
\begin{equation}\label{ne'}
\frac n2<n_e<n'_e\le n. 
\end{equation}
\smallskip

In Step 1 of the resolution algorithm, we can then proceed as in the proof of Theorem \ref{thm-a+}, only with $a_1$ replaced by $a'_1,$ and $n_e$ replaced by $n'_e.$ Note that now $a\ge a'_1,$ and $1/p-1/2<1/(2n)\le 1/(2n'_e),$  so that the discussion in Subsection \ref{9LpEl} carries over directly. Similarly, since now
$$
n-\frac 1{a'_1}=n-(n-n'_e)=n'_e, \quad \text{and} \quad \frac 2p-1< \frac 1{n'_e},
$$
also  the discussion in Subsection \ref{9Lpl0} carries over, since now we can therein assume that $k\ge j/a'_1.$

\smallskip

In Step 2 (and higher later on), recall  that we had put $z_1:=y_1.$ We first note that  here we apply changes of coordinates of the form
$$
\tilde z:= z_1-w(s_2),
$$
where now $|w(s_2)|\sim s_2^a,\quad\text{with} \quad a\ge a'_1.$ It is important that the control of the functions $\tilde g$ and $\tilde h$ that we had discussed in Subsection \ref{9L1ontEl} still works here: recall that
$$
p(2^{-ka}\si_2^a\tilde w(\si_2), 2^{-k}\si_2)=2^{-k/\tilde \ka^e_2}p(2^{-k(a-a_1)}\si_2^a\tilde w(\si_2),\si_2), 
$$
and since $a-a_1\ge a'_1-a_1>0,$ we again can conclude that $2^{-k(a-a_1)}\lesssim 1.$ Similarly, since now $j\ge ka\ge ka'_1\ge ka_1,$ we still have that 
$$
2^{-j}2^{-k\frac {1-\tilde \ka^e_1}{\tilde \ka^e_2}} \lesssim  2^{-k/\tilde \ka^e_2}.
$$

Thus, we can basically proceed as before, but there is one {\bf major difference:}
\smallskip 

In our application of the Multiplicity Lemma \ref{multi} for the study of the contribution by the domain $D_1,$ we can no longer exclude the first possibility \eqref{Pfac}, i.e., it may happen that 
$$
P(z_1,s_2)=(z_1-c_1 s_2^p)^{n-2},\qquad c_1\in\RR\setminus\{0\}, p\in\NN_{\ge 1},
$$
vanishes of maximal order $n-2$ along a nontrivial real root. Note that here we would perform a change of coordinates 
$\tilde z:= z_1-w(s_2),$ with $w(s_2):=c_1s_2^{a'_1},$ where $a'_1=p$ is a natural number, in order to pass from $\Phi$ to 
$\tilde \Phi.$ This is now an analytic change of coordinates, so that $\N(\tilde\Phi)$ is still a Newton polyhedron (and not just Newton-Puiseux), which implies that automatically we will have $\tilde A_1\ge 1.$ 

However, since in the new coordinates, $\tilde P(\tilde z,s_2)=\tilde z^{n-2},$ the first edge $\tilde \ga_1$ of $\N(\tilde\Phi)$ will have a (modulus of) slope $a''_1$ such that $\infty \ge a''_1>a'_1>a_1,$ and we shall be reduced to a region of the form
$$
\tilde D_1^{c_1}:=\{|\tilde z|<\ve s_2^{a'_1}\},
$$
in which we shall again have to devise transition domains $\tilde E_l$ and homogeneous domains $\tilde D_l,$  including the case $l=1.$ 

For  all other cases in our application of Lemma \ref{multi}, the discussion at the end of Subsection \ref{9Dlcont} carries over.
\medskip

 Now, if we proceed to higher steps of our resolution algorithm, we shall either no longer encounter the first possibility \eqref{Pfac}  from a certain step on, so that we can directly adapt our discussion from the proof of Theorem \ref{thm-a+},
 or  the first possibility \eqref{Pfac}  will arise at {\it every} Step j, within the corresponding domain $D_1^{(j)}.$ 
 However, when $j$ is sufficiently large, say when $j\ge j_0,$ then $D_1^{(j)}$ will contain exactly one real root, of maximal multiplicity $n-2,$ (compare the discussion in Section \ref{endproof1}), if we consider this domain in our original coordinates $(z_1,s_2),$ and we see that this root $r(s_2)$ will have a Taylor series expansion of the form
 $$
 r(s_2)=c_1s_2^{a'_1}+c_2s_2^{a'_2}+\cdots +c_js_2^{a'_j}+\cdots,
 $$
 with natural numbers $a'_1<a'_2<\cdots <a'_j<\cdots,$ and real coefficients $c_j.$  Moreover, we may then write 
 $\Phi(z_1,s_2)=V(z_1,s_2)(z_1-r(s_2))^{n-2},$ with $|V|\sim 1.$ As what we always did   in such a situation, we can now  perform the analytic change of coordinates
 $$
\tilde z:= z_1-r(s_2),
$$
and  find that in these new coordinates $(\tilde z, s_2),$ we have 
$$
\tilde\Phi(\tilde z,s_2)=\tilde V(\tilde z,s_2)\tilde z^{n-2}.
$$
As in earlier situations where we had stopped the resolution of singularities algorithm, we now see that we have arrived at a good resolution of singularities and can apply our usual arguments to prove $L^p$ - boundedness of $\M$ for
  $p>\max\{ 3/2,2n/(n+1)\}$ (note that here $n'_e=n$ if we define $n'_e$ with respect to these new coordinates).

\end{proof}


\end{document}